\documentclass[11pt]{amsart}
\usepackage[utf8]{inputenc}
\usepackage[T1]{fontenc}
\usepackage{mathtools}
\usepackage{extarrows}
\usepackage{amsthm}
\usepackage{amsfonts}
\usepackage{amssymb}
\usepackage{mathrsfs}
\usepackage{amscd}
\usepackage{eucal}
\usepackage{xcolor}
\usepackage[enableskew,vcentermath]{youngtab}
\usepackage{bm}
\usepackage{leftindex}
\usepackage{graphicx}
\usepackage{float}
\usepackage{needspace}
\usepackage{tikz}
\usetikzlibrary{arrows.meta}
\usepackage{pstricks}
\usepackage{booktabs,longtable,array}

\usepackage{comment}

\usepackage[pagebackref=true,hidelinks]{hyperref}

\allowdisplaybreaks[4]
	\newcommand{\hs}{\hspace*}
	
\newtheorem{thm}[equation]{Theorem}
\theoremstyle{plain}

\newtheorem{lem}[equation]{Lemma}

\newtheorem{prop}[equation]{Proposition}
\newtheorem{cor}[equation]{Corollary}
\newtheorem{notation}[equation]{Notation}
\theoremstyle{definition}
\newtheorem{defn}[equation]{Definition}
\newtheorem{example}[equation]{Example}

\theoremstyle{remark}
\newtheorem{rem}[equation]{Remark}

\definecolor{A}{rgb}{.75,1,.75}

\numberwithin{equation}{section}

\newcommand{\Z}{\mathbb Z}
\newcommand{\N}{\mathbb N}

\newcommand{\undQ}{\underline{Q}}

\newcommand{\Hom}{{\rm Hom}}

\def\rd{{\rm{d}}}
\def\Sym{\mathfrak{S}}
\def\pr{{\rm pr}}
\def\tQ{\widetilde{Q}}
\let\<=\langle
\let\>=\rangle
\def\({\big(}
\def\){\big)}

\newcommand{\fourlinerightarrow}{
  \mathrel{
    \vcenter{
      \offinterlineskip
      \halign{##\cr
        $\rule{0.6em}{0.3pt}\kern-0.2em\rule{0.3em}{0.3pt}$\cr
        \noalign{\kern0.9pt}
        $\rule{0.6em}{0.3pt}\kern-0.2em\rule{0.3em}{0.3pt}$\cr
        \noalign{\kern0.9pt}
        $\rule{0.6em}{0.3pt}\kern-0.2em\rule{0.3em}{0.3pt}$\cr
        \noalign{\kern0.9pt}
        $\rule{0.6em}{0.3pt}\kern-0.2em\rule{0.3em}{0.3pt}$\cr
      }
    }
    \kern-0.2em\raisebox{0.0ex}{$\succ$}
  }
}

\newcommand{\fourlineleftarrow}{
  \mathrel{
    \raisebox{-0.2ex}{$\prec$}
    \kern-0.2em
    \vbox{
      \offinterlineskip
      \halign{##\cr
        $\rule{0.6em}{0.3pt}\kern-0.2em\rule{0.3em}{0.3pt}$\cr
        \noalign{\kern0.9pt}
        $\rule{0.6em}{0.3pt}\kern-0.2em\rule{0.3em}{0.3pt}$\cr
        \noalign{\kern0.9pt}
        $\rule{0.6em}{0.3pt}\kern-0.2em\rule{0.3em}{0.3pt}$\cr
        \noalign{\kern0.9pt}
        $\rule{0.6em}{0.3pt}\kern-0.2em\rule{0.3em}{0.3pt}$\cr
      }
    }
  }
}

\usepackage{enumitem}

\newlist{caselist}{enumerate}{1}
\setlist[caselist,1]{
	label=\textbf{Case \arabic{section}.\arabic{subsection}.\arabic*:},
	ref=\arabic{section}.\arabic{subsection}.\arabic*,
	before=\setcounter{caselisti}{0}
}

\newcounter{case}

\begin{document}

	\title[Frobenius form and semisimplicity]{Frobenius forms and semisimplicity of cyclotomic quiver Hecke superalgebras}

	\author{Shuo Li}\address{Mathematics Research Center, School of Science and Engineering\\
       The Chinese University of Hong Kong, Shenzhen 518172\\
       P.R.China}

	\email{lishuo@cuhk.edu.cn}

    \author{Lei Shi}
    \address{Max-Planck-Institut f\"ur Mathematik\\
	    Vivatsgasse 7, 53111 Bonn\\
	    Germany}
    \email{leishi202406@163.com}

	\begin{abstract}
We determine the Nakayama automorphism of the canonical homogeneous
Frobenius form on cyclotomic quiver Hecke superalgebras. We give explicit
necessary and sufficient conditions on the dominant weight and the rank
for semisimplicity of cyclotomic
quiver Hecke superalgebras of types $A_e^{(1)}$, $C_e^{(1)}$,
$A_{2e}^{(2)}$, $D_{e+1}^{(2)}$, $A_\infty$, $B_\infty$, and $C_\infty$.
As an application, we prove that a cyclotomic Hecke--Clifford superalgebra
is semisimple if and only if the corresponding Poincar\'e polynomial is
nonzero.
\end{abstract}

	\maketitle

	\setcounter{tocdepth}{1}
	\tableofcontents

\section{Introduction}
Finite Hecke--Clifford superalgebras arise in the spin representation
theory of symmetric groups as $q$-deformations of Sergeev algebras.
Jones and Nazarov introduced their affine counterparts \cite[\S3]{JN}.
The level-one cyclotomic quotients defined by $X_1=1$ recover the finite
algebras \cite[Remark~4.2]{BK}. Brundan and Kleshchev related the
branching rules of these affine and cyclotomic superalgebras to crystals
of type $A_{2\ell}^{(2)}$ \cite[Theorems~8.10--8.11]{BK}.
They further showed that the Grothendieck groups of the cyclotomic
module categories 
realize integrable highest-weight modules of the Kac--Moody algebra of this type
\cite[Theorem~7.16]{BK}.

A $\mathbb Z$-grading is not apparent from the original
Hecke--Clifford presentation. Kang, Kashiwara and Tsuchioka
introduced $\mathbb Z$-graded quiver Hecke superalgebras and quiver
Hecke--Clifford superalgebras. On a suitable finite spectral set, their
cyclotomic isomorphism identifies cyclotomic Hecke--Clifford
superalgebras with cyclotomic quiver Hecke--Clifford superalgebras
\cite[Corollary~4.8]{KKT}; the two quiver constructions are weakly
Morita superequivalent \cite[Theorem~3.13]{KKT}.
For Cartan data satisfying their hypotheses, Hill and Wang used quiver
Hecke superalgebras to categorify the half quantum covering algebra;
its specializations recover the corresponding halves of quantum
Kac--Moody algebras and superalgebras \cite[Theorem~6.14]{HW}.
Kang, Kashiwara and Oh showed that the Grothendieck groups of module
categories over cyclotomic quiver Hecke superalgebras realize
integrable highest-weight modules over quantum Kac--Moody algebras
\cite[Theorem~10.2]{KKO1}.
In their second paper, they incorporated parity shifts into the
Grothendieck groups and obtained, through specialization, a
supercategorification of quantum Kac--Moody superalgebras and their
integrable highest-weight modules
\cite[Introduction and Theorem~8.14]{KKO2}.
In these cyclotomic categorifications, induction and restriction
functors realize the actions of the Chevalley generators, while grading
shifts record the quantum parameter.

For the cyclotomic algebras considered here, adjunctions between these
functors supply the data for our homogeneous Frobenius forms. We use
Rouquier's counit construction
\cite[\S2.3.2, Proposition~2.8]{Rou1} in the graded-super setting of
Brundan--Kleshchev \cite[\S2]{BK:spin}, with the normalization of the
second adjunction in Brundan--Ellis
\cite[(1.17)--(1.18), Proposition~6.2]{BE}.
Having obtained such a form, we ask how it behaves when the order of
multiplication is reversed. This behavior is measured by its Nakayama
automorphism, whose explicit computation is our first main result.

Throughout this paper, ${\rm R}$\label{pag:integral domain R} is an integral
domain and $\mathbb K$ is an algebraically closed field.  When the odd part
of the superalgebra is zero, there is no restriction on the characteristics
of ${\rm R}$ and $\mathbb K$; otherwise, they are different from $2$.

For a dominant integral weight $\Lambda$ and a root element $\alpha$, let
$t_{\Lambda,\alpha}$ denote the homogeneous Frobenius form on
$R^\Lambda_\alpha\ne0$ defined in
Definition~\ref{def:normalized-frobenius}, and let $\mathtt{N}$ be its
Nakayama automorphism,
characterized by $t_{\Lambda,\alpha}(ab)=t_{\Lambda,\alpha}(b\mathtt{N}(a))$.
Here ${\rm p}$ is the original parity, $\pi_{\Lambda,\alpha}$ is the
parity of the form, and $\mathsf r$ is the second $\mathbb Z_2$-grading
defined in \eqref{eq:auxiliary-parity}.

\Needspace{12\baselineskip}
\begin{thm}\label{thm:intro-nakayama}
For the normalized Frobenius form on $R^\Lambda_\alpha\ne0$, the
Nakayama automorphism satisfies
\begin{equation}\label{eq:intro-nakayama}
 \mathtt{N}(a)=(-1)^{\pi_{\Lambda,\alpha}{\rm p}(a)+\mathsf r(a)}a
 \quad\bigl(a\in R^\Lambda_\alpha, \text{ $({\rm p},\mathsf r)$-homogeneous}\bigr),
 \qquad \mathtt{N}^2=\mathrm{id},
\end{equation}
and consequently
\begin{equation*}
 t_{\Lambda,\alpha}(ab)=(-1)^{\pi_{\Lambda,\alpha}{\rm p}(a)+\mathsf r(a)}t_{\Lambda,\alpha}(ba)\qquad(b\in R^\Lambda_\alpha).
\end{equation*}
\end{thm}

We prove the formula using the diagrammatic calculus of the Kac--Moody
$2$-supercategory: the Frobenius form is represented by closing diagrams
with the adjunction cups and caps, and rotating the generating diagrams
determines the Nakayama automorphism.
In the purely even case, \eqref{eq:intro-nakayama} gives
$\mathtt{N}=\mathrm{id}$, recovering the symmetricity conclusion for cyclotomic quiver Hecke
algebras proved by Shan, Varagnolo and Vasserot
\cite[Proposition~3.10]{SVV}.

The Frobenius form also yields a semisimplicity criterion:
Theorem~\ref{thm:structural-criterion} shows that a nonzero
$R^\Lambda_\alpha$ is semisimple if and only if
$\operatorname{def}(\Lambda,\alpha)=0$.
We seek a more explicit criterion that can be read directly from the
cyclotomic parameters.
Such criteria are known for cyclotomic quiver Hecke
algebras of affine types $A_e^{(1)}$ and $C_e^{(1)}$ ($e\ge2$ in the
latter case); see Mathas's survey \cite[Corollary~1.6.11]{Mathas:survey}
and Speyer \cite[Theorem~1.1]{Speyer}, respectively. For twisted affine
type $A_{2e}^{(2)}$, the second author of this paper
\cite[Corollary~8.11]{Shi:semisimplicity}
obtained a criterion in characteristic zero and in characteristic $p$
with $p\nmid(2e+1)$. As an application of the Frobenius form, we give
new combinatorial proofs of these criteria and extend the criterion in type
$A_{2e}^{(2)}$ to every characteristic allowed above.  We also obtain a new
semisimplicity criterion in twisted affine type $D_{e+1}^{(2)}$.
The following theorem summarizes
Theorem~\ref{thm:combinatorial-zero-defect}.

\begin{thm}\label{thm:intro-combinatorial}
Let $I$ be a connected diagram of type $A_e^{(1)}$, $C_e^{(1)}$,
$A_{2e}^{(2)}$, $D_{e+1}^{(2)}$, $A_\infty$, $B_\infty$, or $C_\infty$
as in Section~\ref{DynkinDiagrams}, and let $0\ne\Lambda\in P^+$ satisfy
$\langle h_i,\Lambda\rangle=0$ for all but finitely many $i\in I$.
For $n\ge0$, the following conditions are equivalent:
\begin{enumerate}[label=\textup{(\roman*)}]
\item the condition $\mathtt{S}_\Lambda(n)$ in Definition~\ref{def:local-condition} holds;
\item $R_n^\Lambda(I)$ is semisimple.
\end{enumerate}
\end{thm}

Our combinatorial proof starts from the dimension formula in
\cite[Theorem~1.1]{LS}. Together with the characteristic independence
in \cite[Theorem~1.2]{HS}, it characterizes the $\alpha\in Q^+$ with
$R^\Lambda_\alpha\ne0$ as the elements $\alpha_{\boldsymbol\lambda}$
associated to multipartitions; see Proposition~\ref{prop:multipartition-support}.
We then decompose ordinary partitions into diagonal hooks and strict
partitions into rows. The resulting bilinear pairings determine when
$\operatorname{def}(\Lambda,\alpha_{\boldsymbol\lambda})=0$.

Through the Kang--Kashiwara--Tsuchioka isomorphism,
Theorem~\ref{thm:intro-combinatorial} can be applied to the
semisimplicity of cyclotomic Hecke--Clifford superalgebras.
Fix $q\in\mathbb K^\times$ with $q^4\ne1$, an integer $m\ge0$, and
$\undQ=(Q_1,\ldots,Q_m)\in(\mathbb K^\times)^m$.
For $\bullet\in\{\mathtt{0},\mathtt{s},\mathtt{ss}\}$, let $f^{(\bullet)}_{\undQ}$ be the
cyclotomic relation in \eqref{eq:cyclotomic-parameters}, with zero,
one or two additional linear factors. Write
$\mathcal H^f_{\mathbb K}(n)$ for the corresponding cyclotomic
Hecke--Clifford superalgebra. The Poincar\'e polynomials
$P_n^{(\bullet)}(q^2,\undQ)$ expressing the separate condition are
given in \eqref{eq:nondeg-separation}--\eqref{eq:nondeg-extra-factor}.

The second author of this paper and Wan proved split semisimplicity and constructed the simple
supermodules under the separate condition
$P_n^{(\bullet)}(q^2,\undQ)\ne0$ \cite[Theorem~1.1]{SW}. They
conjectured that this condition is also necessary for semisimplicity
\cite[Conjecture~4.13]{SW}. Under the separate condition, we
constructed primitive idempotents and seminormal bases
\cite[Theorems~1.2 and~1.3]{LS:seminormal}. These bases make the action
of the generators and the multiplication in the simple blocks
explicit. We then used these seminormal bases to construct (super)symmetrizing forms
for a large family of cyclotomic Hecke--Clifford algebras.
We also computed the corresponding Schur elements in the
semisimple case \cite[Theorems~1.1 and~1.2]{LS:forms}.

Recently, the second author of this paper
\cite[Theorem~1.3]{Shi:semisimplicity} used
(super)symmetrizing forms and Schur elements to establish this criterion
for $f=f^{(\mathtt{0})}_{\undQ}$ and, under an additional parameter condition,
for $f=f^{(\mathtt{s})}_{\undQ}$.  At roots of unity, Chen and Wan
established the semisimplicity criterion for the level-one specialization
$f(X)=X-1$ \cite[Theorem~6.11]{CW}.  As an application of
Theorem~\ref{thm:intro-combinatorial}, we establish the criterion in
full generality for the three families above. This is our third main result.

\begin{thm}\label{mainthm}
Let $n\ge1$ and let $f=f^{(\bullet)}_{\undQ}$ for $\bullet\in\{\mathtt{0},\mathtt{s},\mathtt{ss}\}$ be of positive level.  The following conditions are equivalent:
\begin{enumerate}[label=\textup{(\roman*)}]
\item $\mathcal H^f_{\mathbb K}(n)$ is semisimple;
\item $P_n^{(\bullet)}(q^2,\undQ)\ne0$.
\end{enumerate}
\end{thm}

We first extend Rostam's decomposition \cite[\S6]{Salim} to the super
case in Proposition~\ref{prop:rostam-decomposition}. Together with the
cyclotomic Kang--Kashiwara--Tsuchioka isomorphism
\cite[Corollary~4.8]{KKT}, this reduces the problem to a single spectral
orbit, where Theorem~\ref{thm:intro-combinatorial} applies.

For cyclotomic Sergeev algebras, Theorem~\ref{thm:degenerate-semisimplicity}
gives the corresponding criterion for arbitrary cyclotomic parameters.

Here is the layout of this paper.
Section~\ref{preli} recalls quiver Hecke superalgebras and the adjunctions
used to construct the homogeneous Frobenius form.
Section~\ref{sec:frobenius-nakayama} computes its Nakayama automorphism.
Section~\ref{sec:semisimplicity-decomposition} proves the zero-defect
criterion for semisimplicity and the decomposition over disconnected diagrams.
Section~\ref{basic-non-dege} recalls cyclotomic Hecke--Clifford algebras
and the Kang--Kashiwara--Tsuchioka isomorphism, and describes the nonzero
cyclotomic summands using $\alpha_{\boldsymbol\lambda}$.
Section~\ref{sec:single-orbit} gives the combinatorial zero-defect
criterion and proves it in twisted affine type $D$.
Section~\ref{sec:general-semisimplicity} proves
Theorem~\ref{mainthm} and gives the corresponding criterion in the degenerate case.
Appendix~\ref{app:other-affine-types} sketches the proof of
Theorem~\ref{thm:combinatorial-zero-defect} for the remaining three
affine types.

\Needspace{8\baselineskip}
\bigskip
\centerline{\bf Acknowledgements}
\bigskip

The first author is partially supported by the National Natural Science Foundation of China (Nos. 12341101 and 12771460).
Both authors thank Professor Jun Hu for helpful discussions.
The authors used AI-assisted tools for proofreading.
\bigskip

\section{Preliminaries}\label{preli}\label{Quiver Hecke superalgebra and Quiver Hecke-Clifford superalgebra}
\subsection{Superalgebras and semisimplicity}
Semisimplicity of a superalgebra means semisimplicity of the algebra
obtained by forgetting its grading.
We write $\mathbb Z_2=\{\bar0,\bar1\}$; all parities take values in
$\mathbb Z_2$. We use the conventions of \cite[\S2-b]{BK}.
\begin{defn}\label{def:supermodule-conventions}
    For an ${\rm R}$-module $V$ with an ${\rm R}$-module decomposition $V=V_{\bar0}\oplus V_{\bar1}$,
    we denote by ${\rm p}(v):=\varepsilon \in\mathbb{Z}_2$, \label{pag:||} the parity of the element $v$, if $v\in V_{\varepsilon}$.
    By a superalgebra, we mean a
	$\mathbb{Z}_2$-graded associative ${\rm R}$-algebra.
    Let $\mathcal{A}=\mathcal{A}_{\bar0}\oplus \mathcal{A}_{\bar1}$ be an
	${\rm R}$-superalgebra.
    The linear extension of
    \[
    \phi_{\mathcal{A}}(a)=(-1)^{{\rm p}(a)}a
    \qquad(a\in \mathcal{A}\text{ homogeneous})
    \]
    is an algebra involution.

    By an $\mathcal{A}$-(super)module, we mean a $\mathbb{Z}_2$-graded
	left $\mathcal{A}$-module.	A homomorphism $f:V\rightarrow W$ of
	$\mathcal{A}$-modules $V$ and $W$ means an ${\rm R}$-linear map such that
    \[
    f(av)=(-1)^{{\rm p}(f){\rm p}(a)}af(v)
    \quad(a\in\mathcal A,\ v\in V,\ a,f\text{ homogeneous}).
    \]
    All homogeneous-map formulas are extended linearly.

    For an $\mathcal A$-supermodule $V$, the parity shift $\Pi V$
    has the same elements and
    \begin{equation*}
      (\Pi V)_{a}=V_{a+\bar1},\qquad
      b\cdot v=(-1)^{{\rm p}(b)}bv
      \quad(a\in\mathbb Z_2,\ b\in\mathcal A\text{ homogeneous},\ v\in V).
    \end{equation*}
\end{defn}
    Thus, for $a\in\mathbb Z_2$ and $\mathcal A$-supermodules $V,W$,
    \[
    \operatorname{Hom}_{\mathcal A}(V,W)_{a}
     =\operatorname{Hom}_{\mathcal A}(\Pi V,W)_{a+\bar1}
     =\operatorname{Hom}_{\mathcal A}(V,\Pi W)_{a+\bar1}
    \]
    under the identification of the same linear maps.

    For superalgebras $\mathcal A,\mathcal B$ we use the standard supertensor-product convention: for homogeneous
    $a,a'\in\mathcal A$ and $b,b'\in\mathcal B$,
    \begin{equation*}
     (a\otimes b)(a'\otimes b')
     =(-1)^{{\rm p}(b){\rm p}(a')}aa'\otimes bb'.
    \end{equation*}
    Let $V$ be an $\mathcal{A}$-supermodule and $W$ a $\mathcal{B}$-supermodule. For any $s,t\in\mathbb Z_2$,
    we have the following $\mathcal{A}\otimes \mathcal{B}$-supermodule isomorphism:
    \begin{align}\label{eq:Pi-can-iso}
    \Pi^s V \otimes \Pi^t W \longmapsto \Pi^{s+t} (V \otimes W),\quad v\otimes w \longmapsto (-1)^{t{\rm p}(v)} v\otimes w
    \quad (v\in V,w\in W).
    \end{align}

    For a nonzero finite-dimensional irreducible $\mathbb K$-supermodule $V$,
    the super Schur lemma \cite[\S2-b]{BK} gives
    \[
    \dim_{\mathbb K}\operatorname{End}_{\mathcal A}(V)=
    \begin{cases}1,&V\text{ has type }\texttt M,\\
                  2,&V\text{ has type }\texttt Q.
    \end{cases}
    \]

\begin{example}\label{simple algebra}
		1). Let $V$ be a superspace with superdimension $(m,n)$ over $\mathbb{K}$, then $\mathcal{M}_{m,n}:={\text{End}}_{\mathbb{K}}(V)$ is a simple superalgebra with simple module $V$ of {\em type }\texttt{M}. A complete set of pairwise orthogonal idempotents primitive in
        $(\mathcal{M}_{m,n})_{\bar0}$ is
		$\{E_{ii} \mid i=1,\ldots,m+n\}.$ The even supermodule isomorphisms are
        \[
        \mathcal M_{m,n}E_{ii}\cong
        \begin{cases}V,&1\le i\le m,\\
                      \Pi V,&m<i\le m+n.
        \end{cases}
        \]
		2). Let $V$ be a superspace with superdimension $(n,n)$ over $\mathbb{K}$. We define $\mathcal{Q}_n:=
\Biggl\{\biggl(\begin{matrix} A & B\\
                             -B & A
		\end{matrix}\biggr) \biggm| A,B\in  M_n\Biggr\}\subset \mathcal{M}_{n,n}$. A complete set of pairwise orthogonal idempotents primitive in
        $(\mathcal{Q}_n)_{\bar0}$ is
		$\Biggl\{\biggl(\begin{matrix}
            E_{ii} & 0\\
			0& E_{ii}
		\end{matrix}\biggr) \biggm| i\in  \{1,\ldots,n\} \Biggr\}$
		and, as even $\mathcal Q_n$-supermodules,
        \[
         V\cong\mathcal Q_n
         \begin{pmatrix}E_{ii}&0\\0&E_{ii}\end{pmatrix}
         \qquad(1\le i\le n).
        \]
	\end{example}

 \subsection{$\Z\times\Z_2$-graded algebras and modules}
Throughout this subsection we use the grading, superstructure and shift
conventions of \cite[\S2]{BK:spin}.
\begin{defn}\label{def:graded-supermodule-shifts}
 A graded ${\rm R}$-supermodule is an ${\rm R}$-module with
 \[
 M=\bigoplus_{(d,a)\in\Z\times\Z_2}M_{d,a}.
 \]

For $d\in\mathbb Z$, $a\in\mathbb Z_2$ and homogeneous $m\in M_{d,a}$, set
\[
M_d=\bigoplus_{b\in\mathbb Z_2}M_{d,b},\qquad
M_{a}=\bigoplus_{r\in\mathbb Z}M_{r,a},\qquad
(\deg m,{\rm p}(m))=(d,a).
\]
The symbols $M_{\mathbb{Z}}$, $M_{\mathbb{Z}_2}$ and $\underline{M}$
denote the ${\rm R}$-module obtained from $M$ by forgetting, respectively,
the $\Z_2$-, $\Z$- and $\Z\times\Z_2$-grading.
For $l,d\in\mathbb Z$ and $a,b\in\mathbb Z_2$, the shifts are
\begin{equation}\label{eq:degree-parity-shift}
 (q^l\Pi^{b}M)_{d,a}=M_{d-l,a+b}.
\end{equation}
\end{defn}

\begin{defn}\label{def:graded-algebra-module-categories}
A $\Z\times\Z_2$-graded ${\rm R}$-algebra is a unital associative
${\rm R}$-algebra $\mathcal{A}=\bigoplus_{(d,a)\in\Z\times\Z_2}\mathcal{A}_{d,a}$
such that
\[
\mathcal{A}_{d,a}\mathcal{A}_{e,b}\subseteq \mathcal{A}_{d+e,a+b}
\quad(d,e\in\Z,\ a,b\in\Z_2).
\]
A graded (left) $\mathcal{A}$-module is a $\Z\times\Z_2$-graded ${\rm R}$-module $M$ such that $\underline{M}$ is
an $\underline{\mathcal{A}}$-module and $\mathcal{A}_{d,a} M_{e,b}\subseteq M_{d+e,a+b}$, for all
$d,e\in\Z,$ $a,b\in\Z_2.$

 Let $\mathcal{A}$ be a $\Z\times\Z_2$-graded ${\rm R}$-algebra and
 let $M,N$ be $\Z\times\Z_2$-graded $\mathcal{A}$-modules. Set
 \[
 {\rm hom}_{\mathcal{A}}(M,N)=\{f\in{\rm Hom}_{\underline{\mathcal{A}}}(\underline{M},\underline{N})\mid
 f(M_{d,a})\subseteq N_{d,a}\text{ for all }(d,a)\in\Z\times\Z_2 \},
 \]
 \[
 {\rm Hom}_{\mathcal{A}}(M,N)=\bigoplus_{(d,a)\in\Z\times\Z_2}{\rm hom}_{\mathcal{A}}(q^d \Pi^{a} M,N).
 \]
 Its homogeneous components are
 \[
 {\rm Hom}_{\mathcal A}(M,N)_{d,a}
   ={\rm hom}_{\mathcal A}(q^d\Pi^{a}M,N)
   \qquad(d\in\mathbb Z,\ a\in\mathbb Z_2).
 \]
 We use the following categories of finitely generated left modules:
 \[
 \begin{array}{c|c|c}
 \text{category}&\text{module grading}&\text{morphisms}\\\hline
 \mathcal A\text{-}\mathrm{gsmod}&\mathbb Z\times\mathbb Z_2&
                 \text{even maps of degree }0\\
 \mathcal A\text{-gmod}&\mathbb Z&\text{maps of degree }0\\
 \mathcal A\text{-smod}&\mathbb Z_2&\text{even maps}\\
 \mathcal A\text{-mod}&\text{none}&\text{module maps}
 \end{array}
 \]
 The corresponding homogeneous morphism spaces are obtained by all
 permitted grading shifts, as above.
\end{defn}

 \begin{lem}\label{lem:grading-semisimplicity}
 Let $\mathcal A$ be a finite-dimensional $\mathbb Z\times\mathbb Z_2$-graded $\mathbb K$-algebra.
The following conditions are equivalent:
\begin{enumerate}[label=\textup{(\roman*)}]
\item $\mathcal A$ is semisimple.
\item $\mathcal A_{\mathbb Z}$ is semisimple.
\item $\mathcal A_{\mathbb Z_2}$ is semisimple.
\item $\underline{\mathcal A}$ is semisimple.
\end{enumerate}
 \end{lem}
 \begin{proof}
 Since $\phi_{\mathcal A}$ is a $\mathbb Z$-graded algebra automorphism,
 it preserves the $\mathbb Z$-graded radical:
 \[
 \phi_{\mathcal A}(\mathcal J(\mathcal A_{\mathbb Z}))
       =\mathcal J(\mathcal A_{\mathbb Z}).
 \]
 This radical is also
 $\mathbb Z_2$-graded. Maximality of the respective graded nilpotent
 ideals therefore gives
 \[
 \mathcal J(\mathcal A_{\mathbb Z})\subseteq\mathcal J(\mathcal A),
 \qquad
 \mathcal J(\mathcal A)\subseteq\mathcal J(\mathcal A_{\mathbb Z}).
 \]
 Since $\mathcal A$ is finite-dimensional, its $\mathbb Z$-support is
 finite. By \cite[Corollaries~2.9.4(i) and~4.4.7]{NO04},
 \begingroup\postdisplaypenalty=10000\interlinepenalty=10000
\[
 \mathcal J(\mathcal A_{\mathbb Z})
 \overset{\text{Cor.~2.9.4(i)}}\subseteq
 \mathcal J(\underline{\mathcal A})
 \overset{\text{Cor.~4.4.7}}\subseteq
 \mathcal J(\mathcal A_{\mathbb Z}),
 \]
 where the second inclusion uses that $\mathbb Z$ is abelian and
 torsion-free.
\par\endgroup\pagebreak[0] This common radical is $\mathbb Z_2$-graded by the
 first paragraph and nilpotent by finite dimensionality. Thus maximality
 of the graded nilpotent ideal and \cite[Corollary~2.9.4(i)]{NO04} give
 \[
 \mathcal J(\underline{\mathcal A})
 \subseteq\mathcal J(\mathcal A_{\mathbb Z_2})
 \subseteq\mathcal J(\underline{\mathcal A}).
 \]
 Consequently,
\[
\mathcal J(\mathcal A)=\mathcal J(\mathcal A_{\mathbb Z})
=\mathcal J(\mathcal A_{\mathbb Z_2})=\mathcal J(\underline{\mathcal A}).
\]
Each assertion in the statement is equivalent to the vanishing of this
common radical.
 \end{proof}

\subsection{Cartan superdatum}
For the finite symmetrizable case of the following definition, compare
\cite[\S1.1.1]{Lus}.
\begin{defn}\label{def:generalized-Cartan-matrix}
Let $I$ be an index set. A generalized Cartan matrix
${\rm A}=(a_{ij})_{i,j\in I}$ satisfies
\[
a_{ii}=2,\qquad a_{ij}\in\mathbb Z_{\le0}\quad(i\ne j),
\]
and, for $i,j\in I$, the following conditions are equivalent:
\begin{enumerate}[label=\textup{(\roman*)}]
\item $a_{ij}=0$.
\item $a_{ji}=0$.
\end{enumerate}
It is symmetrizable if there is a diagonal matrix
${\rm D}=\operatorname{diag}(\rd_i)_{i\in I}$, with
$\rd_i\in\mathbb Z_{>0}$, such that
\[
 \rd_i a_{ij}=\rd_j a_{ji}\qquad(i,j\in I).
\]
\end{defn}
\label{pag:Cartan superdatum}
\Needspace{6\baselineskip}\begin{defn}\cite[\S4.1]{KKO2}
A Cartan superdatum $\bigl({\rm{A}}=(a_{ij})_{i,j\in I},P,\Pi,\Pi^\vee\bigr)$ satisfies
\begin{enumerate}
	\item[CS1)] ${\rm{A}}$ is a symmetrizable generalized Cartan matrix;
	\item[CS2)] $P$ is a free abelian group, which is called the weight lattice;
	\item[CS3)] $\Pi=\{\nu_i\in P|i\in I\}$, called the set of simple roots, is $\Z$-linearly independent;
	\item[CS4)] $\Pi^\vee=\{h_i\mid i\in I\}\subset P^\vee=\Hom_\Z(P,\Z)$, called the set of simple coroots, satisfies that $\<h_i,\nu_j\>=a_{ij}$ for all $i,j\in I$;
	\item[CS5)] there is a decomposition $I=I_{\rm{even}}\sqcup I_{\rm{odd}}$ such that \begin{equation}\label{evenodd}
		a_{ij}\in 2\Z, \quad \text{for all $i\in I_{\rm{odd}}$ and $j\in I$.}
	\end{equation}
\end{enumerate}
\end{defn}

\begin{notation}\label{not:root-weight-parity}
We use a symmetric bilinear form $(-\mid-)$ on $P$ satisfying
\[
(\nu_i\mid\lambda)={\rm d}_i\langle h_i,\lambda\rangle
\qquad(i\in I,\ \lambda\in P),
\]
and hence
\[
(\nu_i\mid\nu_j)={\rm d}_i a_{ij},\qquad
{\rm d}_i=\tfrac12(\nu_i\mid\nu_i)
\qquad(i,j\in I).
\]

We define the root lattice $Q:=\oplus_{i\in I}\Z\nu_i$
and the positive root lattice $Q^+:=\oplus_{i\in I}\Z_{\geq 0}\nu_i$.
For any $\alpha=\sum_{i\in I}k_i\nu_i\in Q^+$, we define ${\rm ht}(\alpha):=\sum_{i\in I}k_i$ and $\operatorname{supp}\alpha:=\{i\in I\mid k_i\ne0\}$.
For any $n\in\Z_{\geq 0}$, we define $Q_n^+:=\{\alpha\in Q^+\mid {\rm ht}(\alpha)=n\}$.
\label{pag:dominant integral weights}
Let $P^+:=\{\Lambda\in P|\text{$\<h_i,\Lambda\>\in\Z_{\geq 0}$ for all $i\in I$}\}$.
Any element $\Lambda\in P^+$ is called a dominant integral weight.
A weight $\Lambda\in P$ has finite support if the set
$\{i\in I\mid\langle h_i,\Lambda\rangle\ne0\}$ is finite.
Let $\Lambda_i$, $i\in I$ be the fundamental dominant integral weights, which satisfy $\<h_i,\Lambda_j\>=\delta_{i,j}$, $i,j\in I$.
For $i\in I$, define the simple reflection
$s_i(\lambda):=\lambda-\langle h_i,\lambda\rangle\nu_i$ on $P$.
We write
\begin{equation*}
 W:=W_I:=\langle s_i:i\in I\rangle,\qquad
 W_J:=\langle s_j:j\in J\rangle
 \quad(J\subseteq I\text{ a full subdiagram}).
\end{equation*}

\label{pag:parity function on I}
For a Cartan superdatum $({\rm{A}},P,\Pi,\Pi^\vee)$,
we define the parity function ${\rm p}: I\rightarrow\mathbb{Z}_2$ by
\begin{equation*}
{\rm p}(i):=\begin{cases} \bar1, &\text{if $i\in I_{\rm{odd}}$,}\\
\bar0, &\text{if $i\in I_{\rm{even}}$.}
\end{cases}
\end{equation*}
\end{notation}

\subsection{Quiver Hecke superalgebra}
Let $\bigl({\rm{A}},P,\Pi,\Pi^\vee\bigr)$ be a Cartan superdatum.

\begin{defn}\label{def:quiver-Hecke-polynomials}\cite[\S4, (4.2)--(4.3)]{KKO1}
For $i,i'\in I$, let
${\rm Pol}_{i,i'}={\rm R}\langle u,v \rangle / \langle uv-(-1)^{{\rm p}(i){\rm p}(i')}vu \rangle$
and choose $Q_{i,i'}(u,v)\in {\rm Pol}_{i,i'}$ of the form
\label{pag:Q-polys}
\begin{align*}
Q_{i,i'}(u,v)=\sum_{r,s\geq 0}t_{i,i';(r,s)}u^r v^s,
\end{align*}
where $t_{i,i';(r,s)}\in {\rm R}$ for $i,i'\in I$ and $r,s\ge0$, subject to
\begin{align}
 &t_{i,i';(r,s)}\neq 0 \text{ only if } -2(\nu_i|\nu_{i'})-r(\nu_i|\nu_i)-s(\nu_{i'}|\nu_{i'})=0;\notag
\\
\label{condition1}&t_{i,i';(r,s)}=t_{i',i;(s,r)},\quad t_{i,i';(-a_{i,i'},0)}\in {\rm R}^\times\quad(i\ne i');\\
\label{condition2} &t_{i,i';(r,s)}=0\text{ if either $i=i'$ or $i\in I_{\rm{odd}}$ and $r$ is odd.}
\end{align}
\end{defn}

\label{pag:quiver Hecke superalgebras}
\begin{defn}\cite[Definition 3.1]{KKT}
Let $({\rm A},P,\Pi,\Pi^\vee)$ be a Cartan superdatum, let
$\{Q_{i,j}:i,j\in I\}$ be as above, and let $n\in\mathbb N$.
For finite $I$, the quiver Hecke superalgebra $R_n$ over ${\rm R}$ is
generated by
\[
e({\bf i})\ ({\bf i}=(i_1,\ldots,i_n)\in I^n),\qquad
x_k\ (1\leq k\leq n),\qquad \tau_a\ (1\leq a<n),
\]
with parities
\[
{\rm p}(e({\bf i}))=\bar0,\qquad
{\rm p}(x_ke({\bf i}))={\rm p}(i_k),\qquad
{\rm p}(\tau_ae({\bf i}))={\rm p}(i_a){\rm p}(i_{a+1}),
\]
subject to the following relations, for $1\le p,q\le n$,
$1\le a,b<n$, with $1\le a<n-1$ in the braid relation:
\begin{align*}
& e({\bf j})e({\bf i})=\delta_{{\bf j},{\bf i}}e({\bf i}),\,\,\text{for ${\bf j},{\bf i}\in I^n$}, \,\,\sum_{{\bf i}\in I^n}e({\bf i})=1, \\
& x_px_qe({\bf i})=(-1)^{{\rm p}(i_{p}) {\rm p}(i_{q})}x_qx_pe({\bf i}),\,\,\text{if $p\neq q$,}\\
& x_pe({\bf i})=e({\bf i})x_p,\,\,\,\tau_ae({\bf i})=e(s_a{\bf i})\tau_a,\,\,\text{where $s_a=(a,a+1)$,}\\
& \tau_ax_pe({\bf i})=(-1)^{{\rm p}(i_{p}) {\rm p}(i_{a}){\rm p}(i_{a+1})}x_p\tau_ae({\bf i}),\,\,\text{if $p\neq a,a+1$,}\\
& \bigl(\tau_ax_{a+1}-(-1)^{{\rm p}(i_{a}){\rm p}(i_{a+1})}x_a\tau_a\bigr)e({\bf i})\\
& \qquad =\bigl(x_{a+1}\tau_a-(-1)^{{\rm p}(i_{a}){\rm p}(i_{a+1})}\tau_ax_a\bigr)e({\bf i})=\delta_{i_a,i_{a+1}}e({\bf i}),\\
& \tau_a^2e({\bf i})=Q_{i_a,i_{a+1}}(x_a,x_{a+1})e({\bf i}),\\
& \tau_a\tau_be({\bf i})=(-1)^{{\rm p}(i_{a}){\rm p}(i_{a+1}){\rm p}(i_{b}) {\rm p}(i_{b+1})}\tau_b\tau_ae({\bf i}),\,\,\text{if $|a-b|>1$},\\
& (\tau_{a+1}\tau_a\tau_{a+1}-\tau_a\tau_{a+1}\tau_a)e({\bf i})\\
&=\begin{cases}
\frac{Q_{i_a,i_{a+1}}(x_{a+2},x_{a+1})-Q_{i_a,i_{a+1}}(x_{a},x_{a+1})}{x_{a+2}-x_a}e({\bf i}), &\text{if $i_a=i_{a+2}\in I_{\rm{even}}$,}\\
(-1)^{{\rm p}(i_{a+1})}(x_{a+2}-x_a)\frac{Q_{i_a,i_{a+1}}(x_{a+2},x_{a+1})-Q_{i_a,i_{a+1}}(x_{a},x_{a+1})}{x_{a+2}^2-x_a^2}e({\bf i}), &\text{if $i_a=i_{a+2}\in I_{\rm{odd}}$,}\\
0, &\text{otherwise.}
\end{cases}
\end{align*}
The $\mathbb Z$-grading is determined by
\[
\deg(e({\bf i}))=0,\qquad
\deg(x_ke({\bf i}))=(\nu_{i_k}\mid\nu_{i_k}),\qquad
\deg(\tau_ae({\bf i}))=-(\nu_{i_a}\mid\nu_{i_{a+1}}).
\]
\end{defn}

\begin{notation}\label{not:PBW-monomials}
For $w\in\Sym_n$, fix a reduced expression $w=s_{j_1}\cdots s_{j_l}$
and set $\tau_w:=\tau_{j_1}\cdots\tau_{j_l}$.  For
$a=(a_1,\ldots,a_n)\in\mathbb Z_{\ge0}^n$, set
$x^a:=x_1^{a_1}\cdots x_n^{a_n}$.
\end{notation}
\begin{prop}\cite[Corollary~3.15]{KKT}\label{PBWbasis}
The algebra $R_n$ is free over ${\rm R}$ with basis
\[
 \{x^a\tau_we(\mathbf i):a\in\mathbb Z_{\ge0}^n,
               \ w\in\Sym_n,\ \mathbf i\in I^n\}.
\]
\end{prop}

\label{pag:cyclotomic quiver Hecke superalgebra}
\begin{defn}\cite[Section 3.7]{KKT}
Let $\Lambda\in P^+$.
Set
\[
a^\Lambda(x_1):=\sum_{{\bf i}\in I^n}x_1^{\<h_{i_1},\Lambda\>}e({\bf i})\in R_n.
\]
The cyclotomic quiver Hecke superalgebra is
\[
R^\Lambda_n=R_n/\langle a^\Lambda(x_1)\rangle.
\]
\end{defn}

$R^\Lambda_n$ inherits the $\Z\times\Z_2$-grading from $R_n$.
We use the same symbols for the images of the generators in the quotient.

\begin{notation}\label{not:cyclotomic-components}
For $\alpha\in Q_n^+$, set
\[
I^\alpha=\Bigl\{{\bf i}=(i_1,\cdots,i_n)\in I^n\Bigm|\sum_{s=1}^{n}\nu_{i_s}=\alpha\Bigr\},
\qquad
e_\alpha=\sum_{{\bf i}\in I^{\alpha}}e({\bf i}),
\]
where $e_\alpha$ is a central idempotent, and define \label{pag:central-summands}
\[
R_\alpha=e_\alpha R_n,\qquad R^\Lambda_\alpha=e_\alpha R^\Lambda_n.
\]
For infinite $I$, define $R_\alpha$ and $R^\Lambda_\alpha$ by the same
presentations using the finite set of words $I^\alpha$, and set
\begin{equation}\label{eq:infinite-index-direct-sum}
 R_n(I)=\bigoplus_{\alpha\in Q_n^+}R_\alpha(I),\qquad
 R_n^\Lambda(I)=\bigoplus_{\alpha\in Q_n^+}R_\alpha^\Lambda(I).
\end{equation}
Infinite sums of word idempotents and the symbols
$x_k$, $\tau_a$, $a^\Lambda(x_1)$ act on each summand in
\eqref{eq:infinite-index-direct-sum}; such sums need not be elements
of these direct-sum algebras.
\end{notation}

\subsection{The homogeneous Frobenius form}\label{subsec:frobenius}
For graded superalgebras $\mathcal B,\mathcal C$ and graded
$(\mathcal B,\mathcal C)$-superbimodules $M,N$,
a homogeneous bimodule map $g:M\to N$ satisfies
\begin{equation}\label{eq:homogeneous-bimodule-map}
 g(bmc)=(-1)^{{\rm p}(g){\rm p}(b)}bg(m)c
 \qquad(b\in\mathcal B,\ m\in M,\ c\in\mathcal C\text{ homogeneous}).
\end{equation}
We use the parity-shift convention of \cite[\S2, (2.5)]{BK:spin}.
\begin{defn}\label{def:graded-shift}
For $r\in\mathbb Z$ and $s\in\mathbb Z_2$, the shift functor
$q^r\Pi^{s}$ on graded superbimodules, their homogeneous maps, and
natural transformations is given below. The underlying grading is
given by \eqref{eq:degree-parity-shift}, and the elements of $q^r\Pi^{s}M$
are still written $m$.
\begingroup\postdisplaypenalty=10000\interlinepenalty=10000
\begin{equation}\label{eq:parity-shift-bimodule}
 b\mathbin{\cdot}m\mathbin{\cdot}c=(-1)^{s{\rm p}(b)}bmc,
\end{equation}
where $b,m,c$ are homogeneous as in \eqref{eq:homogeneous-bimodule-map}
and the product on the right is taken in $M$.
\par\endgroup\pagebreak[0]
For a homogeneous map $u:M\to N$, the action on morphisms is
\begin{equation}\label{eq:parity-shift-on-maps}
 q^r\Pi^{s}(u):q^r\Pi^{s}M\longrightarrow q^r\Pi^{s}N,
 \qquad m\longmapsto(-1)^{s{\rm p}(u)}u(m)
 \quad(r\in\mathbb Z,\ s\in\mathbb Z_2).
\end{equation}
For a homogeneous natural transformation $u:F\Rightarrow G$ between
functors with graded supermodule values, its shifted natural
transformation is
\begin{equation*}
 q^r\Pi^{s}(u):q^r\Pi^{s}F\Rightarrow q^r\Pi^{s}G,
 \qquad (q^r\Pi^{s}(u))_M=q^r\Pi^{s}(u_M).
\end{equation*}
\end{defn}

\begin{defn}\label{def:local-induction-restriction}
Fix $\Lambda\in P^+$. For $\alpha\in Q^+$ and $i\in I$, put
\begin{equation*}
 e(\alpha,i)=\sum_{\mathbf j\in I^\alpha}e(\mathbf j,i)
 \in R^\Lambda_{\alpha+\nu_i}.
\end{equation*}
For $m=\operatorname{ht}(\alpha)$, the homomorphism
$R^\Lambda_\alpha\to e(\alpha,i)R^\Lambda_{\alpha+\nu_i}e(\alpha,i)$ is given by
\begin{equation*}
 \begin{aligned}
 e(\mathbf j)&\longmapsto e(\mathbf j,i),&&\mathbf j\in I^\alpha,\\
 x_a e(\mathbf j)&\longmapsto x_a e(\mathbf j,i),&&1\le a\le m,\\
 \tau_a e(\mathbf j)&\longmapsto\tau_a e(\mathbf j,i),&&1\le a<m.
 \end{aligned}
\end{equation*}
The induction and restriction functors are
\begin{equation}\label{eq:induction-restriction}
\begin{aligned}
 \mathsf{F}_i:R^\Lambda_\alpha\text{-}\mathrm{gsmod}&\longrightarrow
       R^\Lambda_{\alpha+\nu_i}\text{-}\mathrm{gsmod},
 &V&\longmapsto R^\Lambda_{\alpha+\nu_i}e(\alpha,i)
                   \otimes_{R^\Lambda_\alpha}V,\\
 \mathsf{E}_i:R^\Lambda_{\alpha+\nu_i}\text{-}\mathrm{gsmod}&\longrightarrow
       R^\Lambda_\alpha\text{-}\mathrm{gsmod},
 &W&\longmapsto e(\alpha,i)W.
\end{aligned}
\end{equation}
We write these functors as $\mathsf F_{i,\alpha}$ and
$\mathsf E_{i,\alpha}$ when the base component index $\alpha$ must be shown.
Both module structures in \eqref{eq:induction-restriction} use this
homomorphism. For $\lambda=\Lambda-\alpha$, set
\begin{equation}\label{eq:local-shifts}
 D_{i,\alpha}=2{\rm d}_i(\langle h_i,\lambda\rangle-1),\qquad
 s_{i,\alpha}={\rm p}(i)(\langle h_i,\lambda\rangle+1)\in\mathbb Z_2.
\end{equation}
For $i\in I_{\mathrm{odd}}$, \eqref{condition1}--\eqref{condition2}
force $a_{ij}\in2\mathbb Z$ for $j\ne i$, consistently with
\eqref{evenodd}. Hence $ {\rm p}(i)\langle h_i,\alpha\rangle=\bar0$
for every $i\in I$, and \eqref{eq:local-shifts} gives
\begin{equation}\label{eq:local-parity-independent}
 s_{i,\alpha}=s_{i,0}
 ={\rm p}(i)(\langle h_i,\Lambda\rangle+1).
\end{equation}
For graded vector superspaces $V,W$, the degree-shift convention of
Section~\ref{preli} gives, for an even degree-zero map
$g:q^d\Pi^{s}V\to W$,
\begin{equation}\label{eq:shift-map-degree}
 g(V_{r,a})\subseteq W_{r+d,a+s}
 \qquad(r,d\in\mathbb Z,\ a,s\in\mathbb Z_2).
\end{equation}

For $i,j\in I$, let $t_{ij}=t_{i,j;(-a_{ij},0)}$ when $i\ne j$ and
$t_{ii}=1$. We normalize the second adjunction using the scalars of
\cite[(1.17)]{BE}:
\begin{equation}\label{eq:c-normalization}
c_{\Lambda;i}=1,\qquad
 c_{\lambda+\nu_j;i}=t_{ij}c_{\lambda;i}
 \qquad(\lambda\in\Lambda+Q,\ i,j\in I).
\end{equation}
Compositions of shifted functors and homogeneous maps use
\cite[Definition~1.6]{BE}; the degree-shift functor denoted $Q$ there is
$q$ in our notation.
\end{defn}

The biadjunctions of \cite[(1.16)--(1.18), Proposition~6.2]{BE},
transported via \cite[\S16]{BK:spin}, give the following adjunctions
in our conventions.
\begin{lem}\label{thm:Brbiadjoint}
For $\alpha\in Q^+$ and $i\in I$, the normalization
\eqref{eq:c-normalization} gives
\begin{equation}\label{eq:two-adjunctions}
 \mathsf{F}_i\dashv \mathsf{E}_i,\qquad q^{-D_{i,\alpha}}\Pi^{s_{i,\alpha}}\mathsf{E}_i\dashv \mathsf{F}_i.
\end{equation}
The second counit in \eqref{eq:two-adjunctions}, evaluated at
$R^\Lambda_\alpha$, is the even degree-zero homogeneous linear map
\begin{equation}\label{eq:local-counit}
\begin{gathered}
 \widehat{\varepsilon}_{i,\alpha}:
 q^{-D_{i,\alpha}}\Pi^{s_{i,\alpha}}
 e(\alpha,i)R^\Lambda_{\alpha+\nu_i}e(\alpha,i)\longrightarrow R^\Lambda_\alpha,\\
 (\deg\widehat{\varepsilon}_{i,\alpha},{\rm p}(\widehat{\varepsilon}_{i,\alpha}))=(0,\bar0).
\end{gathered}
\end{equation}
\end{lem}
\begin{proof}
The first adjunction is induction--restriction.  Under the
correspondence of \cite[\S16]{BK:spin}, the weight
$\alpha-\Lambda$ in \cite{BE} corresponds to
$R^\Lambda_\alpha\text{-}\mathrm{gsmod}$, and the simple root
$\alpha_i$ there is $\nu_i$ here.  Thus
\[
 E_i1_{\alpha-\Lambda}\quad\text{acts by}\quad\mathsf F_{i,\alpha},
 \qquad
 F_i1_{\alpha-\Lambda+\nu_i}\quad\text{acts by}\quad
 q^{-D_{i,\alpha}/2}\mathsf E_{i,\alpha}.
\]
We use the leftward cup $\eta'$ and cap $\varepsilon'$ of
\cite[(1.16), (2.10)--(2.11)]{BE} in these components, with the
normalization in \eqref{eq:c-normalization}.  Their transported cup and
cap have bidegrees $(D_{i,\alpha},s_{i,\alpha})$ and
$(-D_{i,\alpha},s_{i,\alpha})$, respectively.  The triangle identities
in \cite[(1.18), Proposition~6.2]{BE}, after these shifts and this
normalization, give the second adjunction in \eqref{eq:two-adjunctions};
the counit is the even degree-zero map in \eqref{eq:local-counit}.
\end{proof}

\begin{defn}\label{def:word-functors}
Let $\alpha\in Q^+$ have height $n$ and
$\mathbf i=(i_1,\ldots,i_n)\in I^\alpha$.  For $0\le k\le n$, set
\[
\alpha_k:=\sum_{a=1}^k\nu_{i_a},\qquad \alpha_0:=0.
\]
The word induction and restriction functors are
\begin{equation*}
\begin{aligned}
 \mathsf{F}_{\mathbf i}=& \mathsf{F}_{i_n}\circ\cdots\circ \mathsf{F}_{i_1}:
 \mathbb K\text{-}\mathrm{gsmod}\longrightarrow R^\Lambda_\alpha\text{-}\mathrm{gsmod},\\
 \mathsf{E}_{\mathbf i}=& \mathsf{E}_{i_1}\circ\cdots\circ \mathsf{E}_{i_n}:
 R^\Lambda_\alpha\text{-}\mathrm{gsmod}\longrightarrow\mathbb K\text{-}\mathrm{gsmod}.
\end{aligned}
\end{equation*}
For $1\le k\le n$,
\[
 \begin{aligned}
 \mathsf{F}_{i_k}&:R^\Lambda_{\alpha_{k-1}}\text{-}\mathrm{gsmod}
       \longrightarrow R^\Lambda_{\alpha_k}\text{-}\mathrm{gsmod},\\
 \mathsf{E}_{i_k}&:R^\Lambda_{\alpha_k}\text{-}\mathrm{gsmod}
       \longrightarrow R^\Lambda_{\alpha_{k-1}}\text{-}\mathrm{gsmod},
 \end{aligned}
 \qquad \mathsf{F}_\emptyset=\mathsf{E}_\emptyset=\mathrm{id}.
\]
\end{defn}
\begin{notation}\label{not:global-shifts}
For $\alpha=\sum_{i\in I}n_i\nu_i\in Q^+$, write
\begin{equation*}
 d_{\Lambda,\alpha}=2(\Lambda\mid\alpha)-(\alpha\mid\alpha),
\end{equation*}
\begin{equation}\label{eq:global-parity-multiplicities}
 \pi_{\Lambda,\alpha}:=\sum_{i\in I}n_i{\rm p}(i)
 (\langle h_i,\Lambda\rangle+1)\in\mathbb Z_2.
\end{equation}
\end{notation}
\begin{lem}\label{lem:word-shifts}
For every word $\mathbf i=(i_1,\ldots,i_n)\in I^\alpha$, its local shifts satisfy
\begin{equation}\label{eq:sum-degree-shifts}
 \sum_{k=1}^nD_{i_k,\alpha_{k-1}}=d_{\Lambda,\alpha},
\end{equation}
\begin{equation}\label{eq:sum-parity-shifts}
 \sum_{k=1}^n s_{i_k,\alpha_{k-1}}
 =\sum_{k=1}^n{\rm p}(i_k)(\langle h_{i_k},\Lambda\rangle+1)
 =\pi_{\Lambda,\alpha}.
\end{equation}
\end{lem}
\begin{proof}
Sum \eqref{eq:local-shifts} and use
$(\nu_i\mid\lambda)={\rm d}_i\langle h_i,\lambda\rangle$ to get
\eqref{eq:sum-degree-shifts}. In the parity shift, the terms
${\rm p}(i_k)a_{i_k,i_a}$ vanish in $\mathbb Z_2$ by \eqref{evenodd}.
The first equality in \eqref{eq:sum-parity-shifts} follows from
\eqref{eq:local-parity-independent}; the second follows from
\eqref{eq:global-parity-multiplicities}.
\end{proof}
\begin{lem}\label{lem:word-adjunctions}
For $\mathbf i\in I^\alpha$, the following adjunctions hold:
\begin{equation}\label{eq:word-adjunctions}
 \mathsf{F}_{\mathbf i}\dashv \mathsf{E}_{\mathbf i},\qquad
 q^{-d_{\Lambda,\alpha}}\Pi^{\pi_{\Lambda,\alpha}}\mathsf{E}_{\mathbf i}
 \dashv \mathsf{F}_{\mathbf i}.
\end{equation}
\end{lem}
\begin{proof}
Iterate \eqref{eq:two-adjunctions}. The degree and parity shifts add
to $(-d_{\Lambda,\alpha},\pi_{\Lambda,\alpha})$ by
\eqref{eq:sum-degree-shifts} and \eqref{eq:sum-parity-shifts}.
\end{proof}
\begin{notation}\label{not:word-counit-source}
Write $\widehat{\varepsilon}_{\mathbf i}$ for the second counit. Then its
component at $\mathbb K$ is the even degree-zero map over $\mathbb K$
\begin{equation}\label{eq:word-adjunction-maps}
 (\widehat{\varepsilon}_{\mathbf i})_{\mathbb K}:
 q^{-d_{\Lambda,\alpha}}\Pi^{\pi_{\Lambda,\alpha}}
 e(\mathbf i)R^\Lambda_\alpha e(\mathbf i)\longrightarrow\mathbb K.
\end{equation}
\end{notation}

\Needspace{9\baselineskip}
\begin{defn}\label{def:normalized-frobenius}
With the normalization \eqref{eq:c-normalization}, the canonical
homogeneous linear form is defined by
\begingroup\postdisplaypenalty=10000\interlinepenalty=10000
\begin{equation}\label{eq:normalized-trace}
 t_{\Lambda,\alpha}(a)=\sum_{\mathbf i\in I^\alpha}
 (\widehat{\varepsilon}_{\mathbf i})_{\mathbb K}
 \bigl(e(\mathbf i)ae(\mathbf i)\bigr)
 \qquad(a\in R^\Lambda_\alpha),
\end{equation}
where $e(\mathbf i)ae(\mathbf i)$ is regarded as an element of the
source in \eqref{eq:word-adjunction-maps}.
\par\endgroup\pagebreak[0]
\end{defn}
The adjunctions \eqref{eq:word-adjunctions} allow us to apply
\cite[\S2.3.2, Proposition~2.8]{Rou1} in the graded superalgebra form of
\cite[\S2, following (2.11)]{BK:spin}.
\begin{lem}\label{Frobeniius}
If $R^\Lambda_\alpha\ne0$, then $t_{\Lambda,\alpha}$ is a homogeneous
Frobenius form of bidegree $(-d_{\Lambda,\alpha},\pi_{\Lambda,\alpha})$.
\end{lem}
\begin{proof}
The finite dimensionality of $R^\Lambda_\alpha$ follows from
\cite[Lemma~7.3]{KKO1}.
The linear form $t_{\Lambda,\alpha}$ in \eqref{eq:normalized-trace}
is therefore Frobenius, of bidegree
$(-d_{\Lambda,\alpha},\pi_{\Lambda,\alpha})$ by
\eqref{eq:word-adjunction-maps} and \eqref{eq:shift-map-degree}.
\end{proof}
\section{The Nakayama automorphism}\label{sec:frobenius-nakayama}
We compute the Nakayama automorphism of the Frobenius form in
Definition~\ref{def:normalized-frobenius}. Fix
$\alpha=\sum_{i\in I}n_i\nu_i\in Q^+$ with $R^\Lambda_\alpha\ne0$.
Write $n=\operatorname{ht}(\alpha)$.
All sign exponents in this section are read in $\mathbb Z_2$.

\begin{defn}\label{def:nakayama}
The Nakayama automorphism of $t_{\Lambda,\alpha}$ is the unique algebra
 automorphism $\mathtt{N}:R^\Lambda_\alpha\to R^\Lambda_\alpha$ satisfying
\begin{equation*}
 t_{\Lambda,\alpha}(ab)=t_{\Lambda,\alpha}(b\mathtt{N}(a))\qquad(a,b\in R^\Lambda_\alpha).
\end{equation*}
\end{defn}
\subsection{Adjunctions and diagrammatic calculus}\label{subsec:graphical-calculus}
\raggedbottom
To compare $t_{\Lambda,\alpha}(gb)$ and $t_{\Lambda,\alpha}(bg)$, we first
identify multiplication by $g$ with a map between the word induction
functors. We then compute its two successive mates using the
adjunctions \eqref{eq:word-adjunctions}. All functors and natural
transformations in this subsection are understood in the corresponding
graded-supermodule categories.

For $\mathbf i,\mathbf j\in I^\alpha$ and a homogeneous element
$g\in e(\mathbf j)R^\Lambda_\alpha e(\mathbf i)$, we use the following
signed right-multiplication transformation.
\begin{defn}\label{def:signed-right-multiplication}
\begingroup\postdisplaypenalty=10000\interlinepenalty=10000
\begin{equation}\label{eq:signed-right-multiplication}
 \begin{aligned}
 \rho_g&:\mathsf{F}_{\mathbf j}\longrightarrow \mathsf{F}_{\mathbf i},\\
 (\rho_g)_V(m\otimes v)&=(-1)^{{\rm p}(g){\rm p}(m)}mg\otimes v,
 \end{aligned}
\end{equation}
where $V\in\mathbb K\text{-}\mathrm{gsmod}$, $m\in R^\Lambda_\alpha e(\mathbf j)$
and $v\in V$ are homogeneous,
$\mathsf{F}_{\mathbf j}(V)=R^\Lambda_\alpha e(\mathbf j)\otimes_{\mathbb K}V$ and
$\mathsf{E}_{\mathbf i}(M)=e(\mathbf i)M$ for
$M\in R^\Lambda_\alpha\text{-}\mathrm{gsmod}$, with
$(\deg\rho_g,{\rm p}(\rho_g))=(\deg g,{\rm p}(g))$.
\par\endgroup\pagebreak[0]
\end{defn}

\newcommand{\traceSideEquality}[3]{%
 \begingroup
 \setbox0=\hbox{$\displaystyle #1#2#3$}%
 \ifdim\wd0<.99\linewidth
  #1#2#3%
 \else
  \begin{aligned}& #1\\[3pt]&\quad #2 #3\end{aligned}%
 \fi
 \endgroup}

\newcommand{\bdBar}[5]{%
\draw[red!70!black,line width=.4pt](#1,#3)--(#2,#3);
\node[text=red!70!black,font=\scriptsize,inner sep=1pt,left=2pt]at(#1,#3){$#4$};
\node[text=red!70!black,font=\scriptsize,inner sep=1pt,right=2pt]at(#2,#3){$#5$};}

\tikzset{trace diagram/.style={line width=.45pt,line cap=round,>=Stealth,
 every node/.style={font=\footnotesize,inner sep=1.5pt},
 box/.style={draw,fill=white,inner xsep=3pt,inner ysep=2.5pt},
 sh/.style={font=\scriptsize,inner sep=1pt}}}

\newcommand{\trRightPlain}[1][\rho_g]{%
\vcenter{\hbox{\begin{tikzpicture}[trace diagram,x=0.55cm,y=0.35cm]
\draw (0,0)--(0,4.8);
\draw[->] (0,1.04)--(0,0.56);
\draw[->] (0,4.24)--(0,3.76);
\node[box] at (0,2.4) {$#1^R$};
\node[above=4pt] at (0,4.8) {$\mathbf j$};
\node[below=4pt] at (0,0) {$\mathbf i$};
\end{tikzpicture}}}%
}

\newcommand{\trDoublePlain}[1][\rho_g]{%
\vcenter{\hbox{\begin{tikzpicture}[trace diagram,x=0.55cm,y=0.35cm]
\draw (0,0)--(0,4.8);
\draw[->] (0,0.56)--(0,1.04);
\draw[->] (0,3.76)--(0,4.24);
\node[box] at (0,2.4) {$#1^{RL}$};
\node[above=4pt] at (0,4.8) {$\mathbf i$};
\node[below=4pt] at (0,0) {$\mathbf j$};
\end{tikzpicture}}}%
}

\newcommand{\trRightMate}[1][\rho_g]{%
\vcenter{\hbox{\begin{tikzpicture}[trace diagram,x=0.55cm,y=0.35cm]
\begin{scope}[xscale=1.4]\draw (0,4.8)--(0,.85) .. controls (0,.10) and (2.6,.10) .. (2.6,.85)--(2.6,3.95) .. controls (2.6,4.7) and (5.2,4.7) .. (5.2,3.95)--(5.2,0);
\draw[->] (0,3.44)--(0,2.96);
\draw[->] (2.6,1.06)--(2.6,1.54);
\draw[->] (2.6,3.26)--(2.6,3.74);
\draw[->] (5.2,1.74)--(5.2,1.26);
\node[above=4pt] at (0,4.8) {$\mathbf j$};
\node[below=4pt] at (5.2,0) {$\mathbf i$};

\node[box] at (2.6,2.4) {$#1$};
\end{scope}\end{tikzpicture}}}%
}

\newcommand{\trSecondMate}[1][\rho_g]{%
\vcenter{\hbox{\begin{tikzpicture}[trace diagram,x=.55cm,y=.65cm]
\draw (-.35,4.8)--(-.35,.65)..controls(-.35,-.1)and(2.95,-.1)..(2.95,.65)--(2.95,4.05)..controls(2.95,4.6)and(4.85,4.6)..(4.85,4.05)--(4.85,0);
\draw (.35,4.8)--(.35,1.1)..controls(.35,.6)and(2.25,.6)..(2.25,1.1)--(2.25,3.8)..controls(2.25,5.3)and(5.55,5.3)..(5.55,3.8)--(5.55,0);
\foreach \x in {-.35,.35}{\draw[->](\x,2.9)--(\x,3.4);}
\foreach \x in {2.25,2.95}{\draw[->](\x,1.7)--(\x,1.2);\draw[->](\x,3.5)--(\x,3.0);}
\foreach \x in {4.85,5.55}{\draw[->](\x,1.2)--(\x,1.7);}
\node[above=3pt]at(-.35,6.3){$i_n$};\node[above=3pt]at(.35,6.3){$i_1$};
\node[below=3pt]at(4.85,-.9){$j_n$};\node[below=3pt]at(5.55,-.9){$j_1$};
\node at (0,4.55){$\cdots$};\node at (5.2,.25){$\cdots$};

\bdBar{1.95}{3.25}{1.4}{-d_{\Lambda,\alpha}}{\pi_{\Lambda,\alpha}}
\bdBar{1.95}{3.25}{3.45}{-d_{\Lambda,\alpha}}{\pi_{\Lambda,\alpha}}
\node[box,minimum width=.62cm]at(2.6,2.4){$#1^R$};
\foreach \x in {-.35,.35}{\draw(\x,4.8)--(\x,6.3);}
\foreach \x in {4.85,5.55}{\draw(\x,0)--(\x,-.9);}
\end{tikzpicture}}}}

\newcommand{\trCup}{%
\vcenter{\hbox{\begin{tikzpicture}[trace diagram,x=0.45cm,y=0.4cm]
\draw (0,3)--(0,1) .. controls (0,-.2) and (3.6,-.2) .. (3.6,1)--(3.6,3);
\draw[->] (0,1.74)--(0,1.26);
\draw[->] (3.6,1.26)--(3.6,1.74);
\node[above=4pt] at (0,3) {$\mathbf k$};
\node[above=4pt] at (3.6,3) {$\mathbf k$};
\node[below=3pt] at (1.8,0.1) {$\eta_{\mathbf k}$};
\end{tikzpicture}}}%
}

\newcommand{\trCap}{%
\vcenter{\hbox{\begin{tikzpicture}[trace diagram,x=0.45cm,y=0.4cm]
\draw (0,0)--(0,2) .. controls (0,3.2) and (3.6,3.2) .. (3.6,2)--(3.6,0);
\draw[->] (0,1.26)--(0,1.74);
\draw[->] (3.6,1.74)--(3.6,1.26);
\node[below=4pt] at (0,0) {$\mathbf k$};
\node[below=4pt] at (3.6,0) {$\mathbf k$};
\node[above=3pt] at (1.8,2.9) {$\varepsilon_{\mathbf k}$};
\end{tikzpicture}}}%
}

\newcommand{\trRawWordCup}{%
\vcenter{\hbox{\begin{tikzpicture}[trace diagram,x=.40cm,y=.38cm]
  \draw (0,6)--(0,1)..controls(0,.15)and(9.5,.15)..(9.5,1)--(9.5,6);
  \draw (1.4,6)--(1.4,2)..controls(1.4,1.35)and(4.6,1.35)..(4.6,2)--(4.6,6);
  \draw[->](0,2.8)--(0,3.3); \draw[->](1.4,2.8)--(1.4,3.3);
  \draw[->](4.6,3.3)--(4.6,2.8); \draw[->](9.5,3.3)--(9.5,2.8);
  \node[above]at(0,6){$k_n$}; \node[above]at(1.4,6){$k_1$};
  \node[above]at(4.6,6){$k_1$}; \node[above]at(9.5,6){$k_n$};
  \node at (.7,5.65){$\cdots$}; \node at (7.05,5.65){$\cdots$};
  \node[below=1pt]at(4.75,.35){$\eta'_{k_n}$}; \node[above=1pt]at(3,2.0){$\eta'_{k_1}$};
  \bdBar{4.36}{4.84}{3.9}{-D_{k_1,\beta_0}}{s_{k_1,\beta_0}}
  \bdBar{9.26}{9.74}{5.1}{-D_{k_n,\beta_{n-1}}}{s_{k_n,\beta_{n-1}}}

\end{tikzpicture}}}}

\newcommand{\trRawWordCap}{%
\vcenter{\hbox{\begin{tikzpicture}[trace diagram,x=.40cm,y=.38cm]
  \draw (-4.5,0)--(-4.5,3)..controls(-4.5,3.85)and(6,3.85)..(6,3)--(6,0);
  \draw (1.4,0)--(1.4,2)..controls(1.4,2.65)and(4.6,2.65)..(4.6,2)--(4.6,0);
  \draw[->](-4.5,1.2)--(-4.5,.7); \draw[->](1.4,1.2)--(1.4,.7);
  \draw[->](4.6,.7)--(4.6,1.2); \draw[->](6,.7)--(6,1.2);
  \node[below]at(-4.5,0){$k_1$}; \node[below]at(1.4,0){$k_n$};
  \node[below]at(4.6,0){$k_n$}; \node[below]at(6,0){$k_1$};
  \node at (-1.55,.3){$\cdots$}; \node at (5.3,.3){$\cdots$};
  \node[above=1pt]at(3,3.65){$\varepsilon'_{k_1}$}; \node at(3,.6){$\varepsilon'_{k_n}$};
  \bdBar{-4.74}{-4.26}{1.85}{-D_{k_1,\beta_0}}{s_{k_1,\beta_0}}
  \bdBar{1.16}{1.64}{1.25}{-D_{k_n,\beta_{n-1}}}{s_{k_n,\beta_{n-1}}}

\end{tikzpicture}}}}

\newcommand{\trFlatWordCup}{%
\vcenter{\hbox{\begin{tikzpicture}[trace diagram,x=.40cm,y=.38cm]
  \draw (0,5.2)--(0,1)..controls(0,.15)and(6,.15)..(6,1)--(6,5.2);
  \draw (1.4,5.2)--(1.4,2)..controls(1.4,1.35)and(4.6,1.35)..(4.6,2)--(4.6,5.2);
  \draw[->](0,2.8)--(0,3.3); \draw[->](1.4,2.8)--(1.4,3.3);
  \draw[->](4.6,3.3)--(4.6,2.8); \draw[->](6,3.3)--(6,2.8);
  \node[above]at(0,5.2){$k_n$}; \node[above]at(1.4,5.2){$k_1$};
  \node[above]at(4.6,5.2){$k_1$}; \node[above]at(6,5.2){$k_n$};
  \node at (.7,4.85){$\cdots$}; \node at (5.3,4.85){$\cdots$};
  \node[below=1pt]at(3,.35){$\eta'_{k_n}$}; \node[above=1pt]at(3,2.0){$\eta'_{k_1}$};
  \bdBar{4.35}{6.25}{5.2}{-d_{\Lambda,\alpha}}{\pi_{\Lambda,\alpha}}
\end{tikzpicture}}}}

\newcommand{\trFlatWordCap}{%
\vcenter{\hbox{\begin{tikzpicture}[trace diagram,x=.40cm,y=.38cm]
  \draw (0,0)--(0,3)..controls(0,3.85)and(6,3.85)..(6,3)--(6,0);
  \draw (1.4,0)--(1.4,2)..controls(1.4,2.65)and(4.6,2.65)..(4.6,2)--(4.6,0);
  \draw[->](0,1.2)--(0,.7); \draw[->](1.4,1.2)--(1.4,.7);
  \draw[->](4.6,.7)--(4.6,1.2); \draw[->](6,.7)--(6,1.2);
  \node[below]at(0,0){$k_1$}; \node[below]at(1.4,0){$k_n$};
  \node[below]at(4.6,0){$k_n$}; \node[below]at(6,0){$k_1$};
  \node at (.7,.3){$\cdots$}; \node at (5.3,.3){$\cdots$};
  \node[above=1pt]at(3,3.65){$\varepsilon'_{k_1}$}; \node[below=1pt]at(3,1.9){$\varepsilon'_{k_n}$};
  \bdBar{-.25}{1.65}{0}{-d_{\Lambda,\alpha}}{\pi_{\Lambda,\alpha}}
\end{tikzpicture}}}}

\newcommand{\trPrimeCup}{%
\begin{aligned}
\widehat\eta_{\mathbf k}&=(-1)^{\sum_{a=1}^{n}s_{k_a,\beta_{a-1}}}\,\trRawWordCup\\
&\xlongequal{\eqref{eq:envelope-sign},\,\eqref{eq:local-parity-independent}}
(-1)^{\sum_a s_{k_a,0}+\sum_{a<b}s_{k_a,0}s_{k_b,0}}\,\trFlatWordCup
\end{aligned}}
\newcommand{\trPrimeCap}{%
\begin{aligned}
\widehat\varepsilon_{\mathbf k}&=\trRawWordCap\\
&\xlongequal{\eqref{eq:envelope-sign}}\trFlatWordCap
\end{aligned}}

\newcommand{\trTriF}{%
\vcenter{\hbox{\begin{tikzpicture}[trace diagram,x=0.43cm,y=0.4cm]
\begin{scope}[xscale=1.35]\draw (0,0)--(0,2.9) .. controls (0,3.6) and (2,3.6) .. (2,2.9)--(2,.7) .. controls (2,0) and (4,0) .. (4,.7)--(4,3.6);
\node[below=4pt] at (0,0) {$\mathbf k$};
\node[above=4pt] at (4,3.6) {$\mathbf k$};

\draw[->] (0,1.56)--(0,2.04);
\draw[->] (2,2.04)--(2,1.56);
\draw[->] (4,1.56)--(4,2.04);
\node[] at (6.1,1.8) {$=$};
\draw (8.2,0)--(8.2,3.6);
\draw[->] (8.2,1.56)--(8.2,2.04);
\node[above=4pt] at (8.2,3.6) {$\mathbf k$};
\node[below=4pt] at (8.2,0) {$\mathbf k$};
\end{scope}\end{tikzpicture}}}%
}

\newcommand{\trTriE}{%
\vcenter{\hbox{\begin{tikzpicture}[trace diagram,x=0.43cm,y=0.4cm]
\begin{scope}[xscale=1.35]\draw (0,3.6)--(0,.7) .. controls (0,0) and (2,0) .. (2,.7)--(2,2.9) .. controls (2,3.6) and (4,3.6) .. (4,2.9)--(4,0);
\node[above=4pt] at (0,3.6) {$\mathbf k$};
\node[below=4pt] at (4,0) {$\mathbf k$};

\draw[->] (0,2.04)--(0,1.56);
\draw[->] (2,1.56)--(2,2.04);
\draw[->] (4,2.04)--(4,1.56);
\node[] at (6.1,1.8) {$=$};
\draw (8.2,0)--(8.2,3.6);
\draw[->] (8.2,2.04)--(8.2,1.56);
\node[above=4pt] at (8.2,3.6) {$\mathbf k$};
\node[below=4pt] at (8.2,0) {$\mathbf k$};
\end{scope}\end{tikzpicture}}}%
}

\newcommand{\trTriPrimeF}{%
(-1)^{\sum_a s_{k_a,0}+\sum_{a<b}s_{k_a,0}s_{k_b,0}}\,
\vcenter{\hbox{\begin{tikzpicture}[trace diagram,x=.70cm,y=.58cm]
\draw(-.25,4.8)--(-.25,.65)..controls(-.25,-.1)and(2.25,-.1)..(2.25,.65)--(2.25,4)..controls(2.25,4.6)and(3.75,4.6)..(3.75,4)--(3.75,0);
\draw(.25,4.8)--(.25,1.1)..controls(.25,.55)and(1.75,.55)..(1.75,1.1)--(1.75,3.8)..controls(1.75,5.1)and(4.25,5.1)..(4.25,3.8)--(4.25,0);
\foreach \x in {-.25,.25,3.75,4.25}{\draw[->](\x,2)--(\x,2.45);}
\foreach \x in {1.75,2.25}{\draw[->](\x,2.45)--(\x,2);}
\node[above]at(-.25,5.8){$k_n$};\node[above]at(.25,5.8){$k_1$};\node[below]at(3.75,-.8){$k_n$};\node[below]at(4.25,-.8){$k_1$};

\node at(0,4.5){\scalebox{.55}[1]{$\cdots$}};\node at(4,.3){\scalebox{.55}[1]{$\cdots$}};
\bdBar{1.4}{2.6}{1.6}{-d_{\Lambda,\alpha}}{\pi_{\Lambda,\alpha}}
\node at(6.5,2.4){$=$};\foreach \x in {8.75,9.25}{\draw(\x,0)--(\x,4.8);\draw[->](\x,2)--(\x,2.45);}
\node[above]at(8.75,4.8){$k_n$};\node[above]at(9.25,4.8){$k_1$};\node[below]at(8.75,0){$k_n$};\node[below]at(9.25,0){$k_1$};\node at(9,2.8){\scalebox{.55}[1]{$\cdots$}};
\foreach \x in {-.25,.25}{\draw(\x,4.8)--(\x,5.8);}
\foreach \x in {3.75,4.25}{\draw(\x,0)--(\x,-.8);}
\end{tikzpicture}}}}
\newcommand{\trTriPrimeE}{%
(-1)^{\sum_a s_{k_a,0}+\sum_{a<b}s_{k_a,0}s_{k_b,0}}\,
\vcenter{\hbox{\begin{tikzpicture}[trace diagram,x=.70cm,y=.58cm]
\draw(-.25,0)--(-.25,4.15)..controls(-.25,4.9)and(2.25,4.9)..(2.25,4.15)--(2.25,.8)..controls(2.25,.2)and(3.75,.2)..(3.75,.8)--(3.75,4.8);
\draw(.25,0)--(.25,3.7)..controls(.25,4.25)and(1.75,4.25)..(1.75,3.7)--(1.75,1)..controls(1.75,-.4)and(4.25,-.4)..(4.25,1)--(4.25,4.8);
\foreach \x in {-.25,.25,3.75,4.25}{\draw[->](\x,2.45)--(\x,2);}
\foreach \x in {1.75,2.25}{\draw[->](\x,2)--(\x,2.45);}
\node[below]at(-.25,-.8){$k_1$};\node[below]at(.25,-.8){$k_n$};\node[above]at(3.75,5.8){$k_1$};\node[above]at(4.25,5.8){$k_n$};

\node at(0,.3){\scalebox{.55}[1]{$\cdots$}};\node at(4,4.5){\scalebox{.55}[1]{$\cdots$}};

\node at(6.5,2.4){$=$};\foreach \x in {8.75,9.25}{\draw(\x,0)--(\x,4.8);\draw[->](\x,2.45)--(\x,2);}
\node[above]at(8.75,4.8){$k_1$};\node[above]at(9.25,4.8){$k_n$};\node[below]at(8.75,0){$k_1$};\node[below]at(9.25,0){$k_n$};\node at(9,2.8){\scalebox{.55}[1]{$\cdots$}};
\bdBar{8.4}{9.6}{.5}{-d_{\Lambda,\alpha}}{\pi_{\Lambda,\alpha}}
\bdBar{8.4}{9.6}{4.3}{-d_{\Lambda,\alpha}}{\pi_{\Lambda,\alpha}}
\foreach \x in {-.25,.25}{\draw(\x,0)--(\x,-.8);}
\foreach \x in {3.75,4.25}{\draw(\x,4.8)--(\x,5.8);}
\bdBar{-.6}{.6}{-.8}{-d_{\Lambda,\alpha}}{\pi_{\Lambda,\alpha}}
\bdBar{3.4}{4.6}{5.8}{-d_{\Lambda,\alpha}}{\pi_{\Lambda,\alpha}}
\end{tikzpicture}}}}

\newcommand{\trInterOne}{%
\vcenter{\hbox{\begin{tikzpicture}[trace diagram,x=0.55cm,y=0.4cm]
\draw (0,0)--(0,3.6);
\draw[->] (0,0.26)--(0,0.74);
\draw[->] (0,2.86)--(0,3.34);
\node[box] at (0,2.45) {$v$};
\draw (3,0)--(3,3.6);
\draw[->] (3,0.26)--(3,0.74);
\draw[->] (3,2.86)--(3,3.34);
\node[box] at (3,1.15) {$u$};
\end{tikzpicture}}}%
}

\newcommand{\trInterTwo}{%
\vcenter{\hbox{\begin{tikzpicture}[trace diagram,x=0.55cm,y=0.4cm]
\draw (0,0)--(0,3.6);
\draw[->] (0,0.26)--(0,0.74);
\draw[->] (0,2.86)--(0,3.34);
\node[box] at (0,1.15) {$v$};
\draw (3,0)--(3,3.6);
\draw[->] (3,0.26)--(3,0.74);
\draw[->] (3,2.86)--(3,3.34);
\node[box] at (3,2.45) {$u$};
\end{tikzpicture}}}%
}

\newcommand{\trConventions}{%
\vcenter{\hbox{\begin{tikzpicture}[trace diagram,x=0.55cm,y=0.4cm]
\draw (0,0)--(0,4.8);
\draw[->] (0,0.56)--(0,1.04);
\draw[->] (0,3.76)--(0,4.24);
\node[box] at (0,2.4) {$\rho_g$};
\node[above=4pt] at (0,4.8) {$\mathbf i$};
\node[below=4pt] at (0,0) {$\mathbf j$};
\begin{scope}[shift={(4,0)}]
\draw (0,0)--(0,4.8);
\draw[->] (0,0.56)--(0,1.04);
\draw[->] (0,3.76)--(0,4.24);
\node[circle,fill=black,inner sep=1.35pt] at (0,2.4) {};
\node[right=4pt] at (0,2.4) {$x_a$};
\node[above=4pt] at (0,4.8) {$i_a$};
\node[below=4pt] at (0,0) {$i_a$};
\end{scope}
\begin{scope}[shift={(8,0)}]
\draw[->] (0,0)--(2.6,4.8);
\draw[->] (2.6,0)--(0,4.8);
\node[below=4pt] at (0,0) {$i_a$};
\node[below=4pt] at (2.6,0) {$i_{a+1}$};
\node[above=4pt] at (0,4.8) {$i_{a+1}$};
\node[above=4pt] at (2.6,4.8) {$i_a$};
\end{scope}
\end{tikzpicture}}}%
}

\newcommand{\trCounitStart}{\vcenter{\hbox{\begin{tikzpicture}[trace diagram,x=.52cm,y=.52cm]
\draw(-.3,0)--(-.3,3.9)..controls(-.3,4.8)and(3.3,4.8)..(3.3,3.9)--(3.3,0);
\draw(.3,0)--(.3,3.5)..controls(.3,4.1)and(2.7,4.1)..(2.7,3.5)--(2.7,0);
\foreach \x in {-.3,.3}{\draw[->](\x,2.9)--(\x,2.4);} \foreach \x in {2.7,3.3}{\draw[->](\x,2.4)--(\x,2.9);}

\node[box,minimum width=.65cm]at(3,1.35){$\rho_g^{RL}$};
\node[below]at(-.3,-.65){$i_1$};\node[below]at(.3,-.65){$i_n$};\node[below]at(2.7,-.65){$j_n$};\node[below]at(3.3,-.65){$j_1$};
\node at(0,.25){$\cdots$};\node at(3,.25){$\cdots$};
\bdBar{-.65}{.65}{.65}{-d_{\Lambda,\alpha}}{\pi_{\Lambda,\alpha}}
\bdBar{-.65}{.65}{3.3}{-d_{\Lambda,\alpha}}{\pi_{\Lambda,\alpha}}
\foreach \x in {-.3,.3,2.7,3.3}{\draw(\x,0)--(\x,-.65);}

\end{tikzpicture}}}}
\newcommand{\trCounitExpand}{\vcenter{\hbox{\begin{tikzpicture}[trace diagram,x=.52cm,y=.52cm]
\draw(-.3,0)--(-.3,4.2)..controls(-.3,4.9)and(3,4.9)..(3,4.2)--(3,.85)..controls(3,.3)and(5.1,.3)..(5.1,.85)--(5.1,3.6)..controls(5.1,4.3)and(8.4,4.3)..(8.4,3.6)--(8.4,0);
\draw(.3,0)--(.3,3.75)..controls(.3,4.45)and(2.4,4.45)..(2.4,3.75)--(2.4,.4)..controls(2.4,-.3)and(5.7,-.3)..(5.7,.4)--(5.7,3.15)..controls(5.7,3.8499999999999996)and(7.8,3.8499999999999996)..(7.8,3.15)--(7.8,0);
\foreach \x in {-.3,.3,5.1,5.7}{\draw[->](\x,1.75)--(\x,1.3);} \foreach \x in {2.4,3,7.8,8.4}{\draw[->](\x,1.3)--(\x,1.75);}

\node[box,minimum width=.65cm]at(5.4,2.4){$\rho_g^R$};
\node[below]at(-.3,-.9){$i_1$};\node[below]at(.3,-.9){$i_n$};\node[below]at(7.8,-.9){$j_n$};\node[below]at(8.4,-.9){$j_1$};
\node at(0,.25){$\cdots$};\node at(8.1,.25){$\cdots$};
\bdBar{-.65}{.65}{.65}{-d_{\Lambda,\alpha}}{\pi_{\Lambda,\alpha}}
\bdBar{4.9}{5.9}{1.05}{-d_{\Lambda,\alpha}}{\pi_{\Lambda,\alpha}}
\bdBar{4.9}{5.9}{3.3}{-d_{\Lambda,\alpha}}{\pi_{\Lambda,\alpha}}
\bdBar{-.65}{.65}{3.7}{-d_{\Lambda,\alpha}}{\pi_{\Lambda,\alpha}}
\foreach \x in {-.3,.3,7.8,8.4}{\draw(\x,0)--(\x,-.9);}

\end{tikzpicture}}}}
\newcommand{\trCounitMove}{\vcenter{\hbox{\begin{tikzpicture}[trace diagram,x=.52cm,y=.52cm]
\draw(-.3,0)--(-.3,2)..controls(-.3,2.7)and(3,2.7)..(3,2)--(3,.85)..controls(3,.3)and(5.1,.3)..(5.1,.85)--(5.1,4.2)..controls(5.1,4.9)and(8.4,4.9)..(8.4,4.2)--(8.4,0);
\draw(.3,0)--(.3,1.55)..controls(.3,2.25)and(2.4,2.25)..(2.4,1.55)--(2.4,.4)..controls(2.4,-.3)and(5.7,-.3)..(5.7,.4)--(5.7,3.75)..controls(5.7,4.45)and(7.8,4.45)..(7.8,3.75)--(7.8,0);
\foreach \x in {-.3,5.1,5.7}{\draw[->](\x,1.75)--(\x,1.3);}
\foreach \x in {3,7.8,8.4}{\draw[->](\x,1.3)--(\x,1.75);}
\draw[->](.3,1.45)--(.3,1.0);
\draw[->](2.4,1.0)--(2.4,1.45);

\node[box,minimum width=.65cm]at(5.4,3){$\rho_g^R$};
\node[below]at(-.3,-.9){$i_1$};\node[below]at(.3,-.9){$i_n$};\node[below]at(7.8,-.9){$j_n$};\node[below]at(8.4,-.9){$j_1$};
\node at(0,.25){$\cdots$};\node at(8.1,.25){$\cdots$};
\bdBar{-.65}{.65}{.65}{-d_{\Lambda,\alpha}}{\pi_{\Lambda,\alpha}}
\bdBar{4.9}{5.9}{1.05}{-d_{\Lambda,\alpha}}{\pi_{\Lambda,\alpha}}
\bdBar{4.9}{5.9}{3.9}{-d_{\Lambda,\alpha}}{\pi_{\Lambda,\alpha}}
\foreach \x in {-.3,.3,7.8,8.4}{\draw(\x,0)--(\x,-.9);}

\end{tikzpicture}}}}
\newcommand{\trCounitEnd}{\vcenter{\hbox{\begin{tikzpicture}[trace diagram,x=.52cm,y=.52cm]
\draw(-.3,0)--(-.3,3.9)..controls(-.3,4.8)and(3.3,4.8)..(3.3,3.9)--(3.3,0);
\draw(.3,0)--(.3,3.5)..controls(.3,4.1)and(2.7,4.1)..(2.7,3.5)--(2.7,0);
\foreach \x in {-.3,.3}{\draw[->](\x,2.9)--(\x,2.4);} \foreach \x in {2.7,3.3}{\draw[->](\x,2.4)--(\x,2.9);}

\node[box,minimum width=.65cm]at(0,1.35){$\rho_g^R$};
\node[below]at(-.3,-.65){$i_1$};\node[below]at(.3,-.65){$i_n$};\node[below]at(2.7,-.65){$j_n$};\node[below]at(3.3,-.65){$j_1$};
\node at(0,.25){$\cdots$};\node at(3,.25){$\cdots$};
\bdBar{-.65}{.65}{.65}{-d_{\Lambda,\alpha}}{\pi_{\Lambda,\alpha}}
\bdBar{-.65}{.65}{2.6}{-d_{\Lambda,\alpha}}{\pi_{\Lambda,\alpha}}
\foreach \x in {-.3,.3,2.7,3.3}{\draw(\x,0)--(\x,-.65);}

\end{tikzpicture}}}}

\newcommand{\trBubbleBefore}[1][\alpha]{%
\vcenter{\hbox{\begin{tikzpicture}[trace diagram,x=0.45cm,y=0.4cm]
\draw (0,0)--(0,3.5);
\draw[->] (0,1.76)--(0,2.24);
\node[above=4pt] at (0,3.5) {$\mathbf k$};
\node[below=4pt] at (0,0) {$\mathbf k$};
\draw (-2,1.7) circle (.27); \draw (-2.19,1.51)--(-1.81,1.89); \draw (-2.19,1.89)--(-1.81,1.51);
\node[above=1pt] at (-2,2.1) {$i$};
\node[above=3pt] at (-2,3.6) {$\Lambda-#1$};
\node[above=3pt] at (2,3.6) {$\Lambda$};
\end{tikzpicture}}}%
}

\newcommand{\trBubbleAfter}[1][\alpha]{%
\vcenter{\hbox{\begin{tikzpicture}[trace diagram,x=0.45cm,y=0.4cm]
\draw (0,0)--(0,3.5);
\draw[->] (0,1.76)--(0,2.24);
\node[above=4pt] at (0,3.5) {$\mathbf k$};
\node[below=4pt] at (0,0) {$\mathbf k$};
\draw (2,1.7) circle (.27); \draw (1.81,1.51)--(2.19,1.89); \draw (1.81,1.89)--(2.19,1.51);
\node[above=1pt] at (2,2.1) {$i$};
\node[above=3pt] at (-2,3.6) {$\Lambda-#1$};
\node[above=3pt] at (2,3.6) {$\Lambda$};
\end{tikzpicture}}}%
}

\newcommand{\mrArrow}[5]{\draw[#1,-{Stealth[length=2.6pt]}] (#2,#3)--(#4,#5);}
\newcommand{\mrLab}[4]{\node at (#2,#3) {$#4$};}
\newcommand{\mrShiftMark}[4]{\bdBar{#1-.18}{#1+.18}{#2}{#3}{#4}}
\newcommand{\mrPairShift}[5]{\bdBar{#1-.18}{#2+.18}{#3}{#4}{#5}}
\newcommand{\mrBoundary}{%
\draw(0,0)--(0,-.65);\draw(1,0)--(1,-.65);
\draw(4,5)--(4,5.7);\draw(5,5)--(5,5.7);
\mrLab{}{0}{-1.1}{i}\mrLab{}{1}{-1.1}{j}
\mrLab{}{4}{6.15}{j}\mrLab{}{5}{6.15}{i}
\mrPairShift{0}{5}{-.65}{-D_{i,\alpha}-D_{j,\alpha+\nu_i}}{s_{i,0}+s_{j,0}}
\mrPairShift{0}{5}{5.7}{-D_{j,\alpha}-D_{i,\alpha+\nu_j}}{s_{j,0}+s_{i,0}}}
\newcommand{\mrCommonArrows}{\mrArrow{ii}{0}{1.1}{0}{.65}\mrArrow{jj}{4}{4.5}{4}{4.05}\mrArrow{ii}{5}{4.5}{5}{4.05}}
\newcommand{\mrCrossStage}[1]{\vcenter{\hbox{\begin{tikzpicture}[x=.60cm,y=.43cm,line width=.45pt,line cap=round,ii/.style={},jj/.style={},every node/.style={font=\footnotesize,inner sep=1pt}]
\ifcase#1\relax

\draw[ii] (0,0)--(0,4) .. controls (0,4.4) and (3,4.4) .. (3,4)--(3,2.4)--(2,1.6)--(2,.5) .. controls (2,.0) and (5,.0) .. (5,.5)--(5,5);
\draw[jj] (1,0)--(1,3) .. controls (1,3.4) and (2,3.4) .. (2,3)--(2,2.4)--(3,1.6)--(3,1) .. controls (3,.6) and (4,.6) .. (4,1)--(4,5);
\mrCommonArrows\mrArrow{jj}{1}{1.1}{1}{.65}\mrArrow{ii}{2}{1.05}{2}{1.45}\mrArrow{jj}{3}{1.05}{3}{1.45}

\mrBoundary

\or

\draw[ii] (0,0)--(0,4) .. controls (0,4.4) and (1,4.4) .. (1,4)--(1,2.4)--(2,1.6)--(2,.5) .. controls (2,0) and (5,0) .. (5,.5)--(5,5);
\draw[jj] (1,0)--(1,1.6)--(2,2.4)--(2,3) .. controls (2,3.4) and (3,3.4) .. (3,3)--(3,1) .. controls (3,.6) and (4,.6) .. (4,1)--(4,5);
\mrCommonArrows\mrArrow{jj}{1}{1.1}{1}{.65}\mrArrow{ii}{2}{1.05}{2}{1.45}\mrArrow{jj}{3}{1.6}{3}{2.05}

\mrBoundary

\or

\draw[ii] (0,0)--(0,4.3) .. controls (0,4.7) and (1,4.7) .. (1,4.3)--(1,1.9)--(2,1.1)--(2,.5) .. controls (2,0) and (5,0) .. (5,.5)--(5,5);
\draw[jj] (1,0)--(1,1.1)--(2,1.9)--(2,3.5) .. controls (2,3.9) and (3,3.9) .. (3,3.5)--(3,2.5) .. controls (3,2.1) and (4,2.1) .. (4,2.5)--(4,5);
\mrCommonArrows\mrArrow{jj}{1}{.9}{1}{.5}\mrArrow{ii}{1}{2.5}{1}{2.95}\mrArrow{jj}{3}{2.7}{3}{3.15}

\mrBoundary

\or

\draw[ii] (0,0)--(0,4.3) .. controls (0,4.7) and (1,4.7) .. (1,4.3)--(1,1.9)--(2,1.1)--(2,.5) .. controls (2,0) and (5,0) .. (5,.5)--(5,5);
\draw[jj] (1,0)--(1,1.1)--(2,1.9)--(2,3.1) .. controls (2,3.8) and (4,3.5) .. (4,4.2)--(4,5);
\mrCommonArrows\mrArrow{jj}{1}{.9}{1}{.5}\mrArrow{ii}{1}{2.5}{1}{2.95}

\mrBoundary

\or

\draw[ii] (0,0)--(0,4.3) .. controls (0,4.7) and (1,4.7) .. (1,4.3)--(1,1) .. controls (1,.65) and (2,.65) .. (2,1)--(2,1.6)--(5,2.4)--(5,5);
\draw[jj] (1,0) .. controls (1,.4) and (5,.25) .. (5,.7)--(5,1.6)--(2,2.4)--(2,3.1) .. controls (2,3.8) and (4,3.5) .. (4,4.2)--(4,5);
\mrCommonArrows\mrArrow{ii}{1}{2.5}{1}{2.95}\mrArrow{jj}{5}{1.5}{5}{1.05}

\mrBoundary

\or

\draw[ii] (0,0)--(0,2.1) .. controls (0,2.5) and (1,2.5) .. (1,2.1)--(1,1) .. controls (1,.65) and (2,.65) .. (2,1)--(2,3)--(5,4)--(5,5);
\draw[jj] (1,0) .. controls (1,.4) and (5,.25) .. (5,.7)--(5,3)--(2,4) .. controls (2,4.5) and (4,4.2) .. (4,4.65)--(4,5);
\mrArrow{ii}{0}{1.1}{0}{.65}\mrArrow{ii}{1}{1.3}{1}{1.8}\mrArrow{ii}{2}{2.6}{2}{2.15}\mrArrow{jj}{5}{2.6}{5}{2.15}\mrArrow{ii}{5}{4.75}{5}{4.3}

\mrBoundary

\or

\draw[ii] (0,0)--(0,1.5)--(2,3.5)--(2,5);
\draw[jj] (2,0)--(2,1.5)--(0,3.5)--(0,5);
\mrArrow{ii}{0}{1.1}{0}{.65}\mrArrow{jj}{2}{1.1}{2}{.65}\mrArrow{ii}{2}{4.5}{2}{4.05}\mrArrow{jj}{0}{4.5}{0}{4.05}
\mrLab{icol}{0}{-.35}{i}\mrLab{jcol}{2}{-.35}{j}\mrLab{jcol}{0}{5.35}{j}\mrLab{icol}{2}{5.35}{i}
\mrPairShift{0}{2}{0}{-D_{i,\alpha}-D_{j,\alpha+\nu_i}}{s_{i,0}+s_{j,0}}\mrPairShift{0}{2}{5}{-D_{j,\alpha}-D_{i,\alpha+\nu_j}}{s_{j,0}+s_{i,0}}\fi\end{tikzpicture}}}}
\newcommand{\mrCapLeft}{\vcenter{\hbox{\begin{tikzpicture}[x=1.5cm,y=.48cm,line width=.45pt,>=Stealth,every node/.style={font=\footnotesize,inner sep=1pt}]

\draw (0,0)--(0,2.5)..controls(0,3.05)and(1,3.05)..(1,2.5)--(1,1.9)--(2,1.1)--(2,0);
\draw (1,0)--(1,1.1)--(2,1.9)--(2,3.4);
\draw[->](0,.8)--(0,.35);\draw[->](1,.35)--(1,.8);\draw[->](2,.35)--(2,.8);\draw[->](2,2.65)--(2,3.1);
\node[below=5pt]at(0,0){$j$};\node[below=5pt]at(1,0){$i$};\node[below=5pt]at(2,0){$j$};\node[above=5pt]at(2,3.4){$i$};\node[above=5pt]at(.5,3.02){$\varepsilon'$};

\mrShiftMark{0}{0}{-D_{j,\alpha+\nu_i}}{s_{j,0}}
\end{tikzpicture}}}}
\newcommand{\mrCapRight}{\vcenter{\hbox{\begin{tikzpicture}[x=1.5cm,y=.48cm,line width=.45pt,>=Stealth,every node/.style={font=\footnotesize,inner sep=1pt}]
\draw (0,0)--(0,1.1)--(1,1.9)--(1,2.5)..controls(1,3.05)and(2,3.05)..(2,2.5)--(2,0);
\draw (1,0)--(1,1.1)--(0,1.9)--(0,3.4);
\draw[->](0,.8)--(0,.35);\draw[->](1,.35)--(1,.8);\draw[->](2,.35)--(2,.8);\draw[->](0,2.65)--(0,3.1);
\node[below=5pt]at(0,0){$j$};\node[below=5pt]at(1,0){$i$};\node[below=5pt]at(2,0){$j$};\node[above=5pt]at(0,3.4){$i$};\node[above=5pt]at(1.5,3.02){$\varepsilon'$};
\mrShiftMark{0}{0}{-D_{j,\alpha+\nu_i}}{s_{j,0}}

\end{tikzpicture}}}}


\newcommand{\mrRightCrossDefinition}{\vcenter{\hbox{\begin{tikzpicture}[x=.65cm,y=.55cm,line width=.45pt,>=Stealth,every node/.style={font=\footnotesize,inner sep=1pt}]

\draw(0,0)--(0,.6)--(1,1.6)--(1,2.2);\draw(1,0)--(1,.6)--(0,1.6)--(0,2.2);
\draw[->](0,.5)--(0,.1);\draw[->](1,.5)--(1,.1);\draw[->](0,2.1)--(0,1.7);\draw[->](1,2.1)--(1,1.7);
\node[below=5pt]at(0,0){$i$};\node[below=5pt]at(1,0){$j$};\node[above=5pt]at(0,2.2){$j$};\node[above=5pt]at(1,2.2){$i$};

\end{tikzpicture}}}}
\newcommand{\mrMixedCrossDefinition}{\vcenter{\hbox{\begin{tikzpicture}[x=.65cm,y=.55cm,line width=.45pt,>=Stealth,every node/.style={font=\footnotesize,inner sep=1pt}]

\draw(0,0)--(0,.6)--(1,1.6)--(1,2.2);\draw(1,0)--(1,.6)--(0,1.6)--(0,2.2);
\draw[->](0,.5)--(0,.1);\draw[->](1,.1)--(1,.5);\draw[->](0,1.7)--(0,2.1);\draw[->](1,2.1)--(1,1.7);
\node[below=5pt]at(0,0){$j$};\node[below=5pt]at(1,0){$i$};\node[above=5pt]at(0,2.2){$i$};\node[above=5pt]at(1,2.2){$j$};
\end{tikzpicture}}}}

\newcommand{\mrUpCross}{\vcenter{\hbox{%
\begin{tikzpicture}[x=.65cm,y=.48cm,line width=.45pt,>=Stealth,
 every node/.style={font=\footnotesize,inner sep=1pt}]
\draw(0,0)--(0,.6)--(1,1.6)--(1,2.2);
\draw(1,0)--(1,.6)--(0,1.6)--(0,2.2);
\draw[->](0,.1)--(0,.5);\draw[->](1,.1)--(1,.5);
\draw[->](0,1.7)--(0,2.1);\draw[->](1,1.7)--(1,2.1);
\node[below=3pt]at(0,0){$i$};\node[below=3pt]at(1,0){$j$};
\node[above=3pt]at(0,2.2){$j$};\node[above=3pt]at(1,2.2){$i$};
\end{tikzpicture}}}}


\newcommand{\mrCancelLast}[3]{%
\vcenter{\hbox{\begin{tikzpicture}[trace diagram,x=.43cm,y=.34cm]
\draw (0,0)--(0,5.2);
\draw[->] (0,1)--(0,1.55);
\draw[->] (0,3.6)--(0,4.15);
\fill (0,2.5) circle (1.25pt);\node[right=3pt] at (0,2.5) {$#1$};
\node[above=4pt] at (0,5.2) {$#2$};
\node[below=4pt] at (0,0) {$#3$};

\end{tikzpicture}}}}

\tikzset{grstrand/.style={line width=.5pt},grarr/.style={-{Stealth[length=2.5pt]},line width=.5pt},grbox/.style={draw,fill=white,inner xsep=2pt,inner ysep=2pt,line width=.5pt}}

\newcommand{\grMarkShift}[4]{\bdBar{#1-.14}{#1+.14}{#2}{#3}{#4}}
\newcommand{\grSeparateShifts}{\vcenter{\hbox{%
\begin{tikzpicture}[x=.50cm,y=.40cm,every node/.style={font=\footnotesize,inner sep=1.2pt}]
\draw[grstrand](0,0)--(0,4);\draw[grstrand](3,0)--(3,4);
\foreach \x in {0,3}{\draw[grarr](\x,.5)--(\x,1);\draw[grarr](\x,3)--(\x,3.5);}
\node[grbox]at(0,2){$\psi$};\node[grbox]at(3,2){$\varphi$};
\grMarkShift{0}{0}{t}{c}\grMarkShift{0}{4}{t'}{d}
\grMarkShift{3}{0}{r}{a}\grMarkShift{3}{4}{r'}{b}
\node[below=7pt]at(0,0){$G$};\node[below=7pt]at(3,0){$F$};
\node[above=7pt]at(0,4){$G'$};\node[above=7pt]at(3,4){$F'$};
\end{tikzpicture}}}}
\newcommand{\grCombinedShifts}{\vcenter{\hbox{%
\begin{tikzpicture}[x=.50cm,y=.40cm,every node/.style={font=\footnotesize,inner sep=1.2pt}]
\draw[grstrand](0,0)--(0,4);\draw[grstrand](2,0)--(2,4);
\foreach \x in {0,2}{\draw[grarr](\x,.5)--(\x,1);\draw[grarr](\x,3)--(\x,3.5);}
\draw[grstrand,fill=white](-.3,1.57)rectangle(2.3,2.43);
\node at(1,2){$\psi\varphi$};
\bdBar{-.14}{2.14}{0}{r+t}{a+c}
\bdBar{-.14}{2.14}{4}{r'+t'}{b+d}
\node[below=7pt]at(0,0){$G$};\node[below=7pt]at(2,0){$F$};
\node[above=7pt]at(0,4){$G'$};\node[above=7pt]at(2,4){$F'$};
\end{tikzpicture}}}}

\newcommand{\wdMark}[4]{%
  \draw (#1-.18,#2)--(#1+.18,#2);
  \node[font=\scriptsize,left=1.5pt] at (#1-.18,#2) {$#3$};
  \node[font=\scriptsize,right=1.5pt] at (#1+.18,#2) {$#4$};}

\newcommand{\wdAdjacentMateStart}{\vcenter{\hbox{%
\begin{tikzpicture}[x=.55cm,y=.34cm,line width=.45pt,>=Stealth,
  every node/.style={font=\footnotesize,inner sep=1.2pt}]
\draw (0,6)--(0,1.05) .. controls (0,.28) and (3,.28) .. (3,1.05)
      --(3,3.95) .. controls (3,4.72) and (4,4.72) .. (4,3.95)--(4,0);
\draw[->] (0,3.55)--(0,3.05);\draw[->] (3,2.15)--(3,2.65);\draw[->] (4,2.45)--(4,1.95);
\draw (1,6)--(1,2.05) .. controls (1,1.28) and (2,1.28) .. (2,2.05)
      --(2,4.95) .. controls (2,5.72) and (5,5.72) .. (5,4.95)--(5,0);
\draw[->] (1,4.05)--(1,3.55);\draw[->] (2,2.55)--(2,3.05);\draw[->] (5,2.45)--(5,1.95);
\node[draw,fill=white] at (2,3.68) {$u$};
\node[below=2pt] at (1.5,1.28) {$\eta_F$};\node[below=2pt] at (1.5,.28) {$\eta_H$};
\node[above=2pt] at (3.5,5.72) {$\varepsilon_G$};\node[above=2pt] at (3.5,4.72) {$\varepsilon_H$};
\node[above=4pt] at (0,6) {$H^*$};\node[above=4pt] at (1,6) {$F^*$};
\node[below=4pt] at (4,0) {$H^*$};\node[below=4pt] at (5,0) {$G^*$};
\end{tikzpicture}}}}

\newcommand{\wdAdjacentMateInterchange}{\vcenter{\hbox{%
\begin{tikzpicture}[x=.55cm,y=.34cm,line width=.45pt,>=Stealth,
  every node/.style={font=\footnotesize,inner sep=1.2pt}]
\draw (0,6)--(0,1.05) .. controls (0,.28) and (3,.28) .. (3,1.05)
      --(3,2.45) .. controls (3,3.22) and (4,3.22) .. (4,2.45)--(4,0);
\draw[->] (0,3.55)--(0,3.05);\draw[->] (3,1.65)--(3,2.15);\draw[->] (4,1.75)--(4,1.25);
\draw (1,6)--(1,2.05) .. controls (1,1.28) and (2,1.28) .. (2,2.05)
      --(2,4.95) .. controls (2,5.72) and (5,5.72) .. (5,4.95)--(5,0);
\draw[->] (1,5.25)--(1,4.75);\draw[->] (2,4.35)--(2,4.85);\draw[->] (5,2.45)--(5,1.95);
\node[draw,fill=white] at (2,4.45) {$u$};
\node[below=2pt] at (1.5,1.28) {$\eta_F$};\node[below=2pt] at (1.5,.28) {$\eta_H$};
\node[above=2pt] at (3.5,5.72) {$\varepsilon_G$};\node[above=2pt] at (3.5,3.22) {$\varepsilon_H$};
\node[above=4pt] at (0,6) {$H^*$};\node[above=4pt] at (1,6) {$F^*$};
\node[below=4pt] at (4,0) {$H^*$};\node[below=4pt] at (5,0) {$G^*$};
\end{tikzpicture}}}}

\newcommand{\wdAdjacentMateMiddle}{\vcenter{\hbox{%
\begin{tikzpicture}[x=.55cm,y=.34cm,line width=.45pt,>=Stealth,
  every node/.style={font=\footnotesize,inner sep=1.2pt}]
\draw (0,6)--(0,1.05) .. controls (0,.28) and (3,.28) .. (3,1.05)
      --(3,2.45) .. controls (3,3.22) and (4,3.22) .. (4,2.45)--(4,0);
\draw[->] (0,3.55)--(0,3.05);\draw[->] (3,1.65)--(3,2.15);\draw[->] (4,1.75)--(4,1.25);
\draw (1,6)--(1,4.05) .. controls (1,3.28) and (2,3.28) .. (2,4.05)
      --(2,4.95) .. controls (2,5.72) and (5,5.72) .. (5,4.95)--(5,0);
\draw[->] (1,5.25)--(1,4.75);\draw[->] (2,4.35)--(2,4.85);\draw[->] (5,2.45)--(5,1.95);
\node[draw,fill=white] at (2,4.45) {$u$};
\node[below=2pt] at (1.5,3.28) {$\eta_F$};\node[below=2pt] at (1.5,.28) {$\eta_H$};
\node[above=2pt] at (3.5,5.72) {$\varepsilon_G$};\node[above=2pt] at (3.5,3.22) {$\varepsilon_H$};
\node[above=4pt] at (0,6) {$H^*$};\node[above=4pt] at (1,6) {$F^*$};
\node[below=4pt] at (4,0) {$H^*$};\node[below=4pt] at (5,0) {$G^*$};
\end{tikzpicture}}}}

\newcommand{\wdAdjacentMateThird}{\vcenter{\hbox{%
\begin{tikzpicture}[x=.55cm,y=.34cm,line width=.45pt,>=Stealth,
  every node/.style={font=\footnotesize,inner sep=1.2pt}]
\draw (0,0)--(0,6);\draw[->] (0,2.95)--(0,2.45);
\draw (1,6)--(1,1.05) .. controls (1,.28) and (3,.28) .. (3,1.05)
      --(3,4.95) .. controls (3,5.72) and (5,5.72) .. (5,4.95)--(5,0);
\draw[->] (1,3.55)--(1,3.05);\draw[->] (3,2.25)--(3,2.75);\draw[->] (5,2.45)--(5,1.95);
\node[draw,fill=white] at (3,3.55) {$u$};
\node[below=2pt] at (2,.28) {$\eta_F$};\node[above=2pt] at (4,5.72) {$\varepsilon_G$};
\node[above=4pt] at (0,6) {$H^*$};\node[above=4pt] at (1,6) {$F^*$};
\node[below=4pt] at (0,0) {$H^*$};\node[below=4pt] at (5,0) {$G^*$};
\end{tikzpicture}}}}

\newcommand{\wdAdjacentMateEnd}{\vcenter{\hbox{%
\begin{tikzpicture}[x=.55cm,y=.34cm,line width=.45pt,>=Stealth,
  every node/.style={font=\footnotesize,inner sep=1.2pt}]
\draw (0,0)--(0,6);\draw (3,0)--(3,6);
\draw[->] (0,2.95)--(0,2.45);\draw[->] (3,2.95)--(3,2.45);
\node[draw,fill=white] at (3,3.55) {$u^R$};
\node[above=4pt] at (0,6) {$H^*$};\node[above=4pt] at (3,6) {$F^*$};
\node[below=4pt] at (0,0) {$H^*$};\node[below=4pt] at (3,0) {$G^*$};
\end{tikzpicture}}}}


\newcommand{\wdWordDot}{\vcenter{\hbox{%
\begin{tikzpicture}[x=.48cm,y=.42cm,line width=.45pt,>=Stealth,
 every node/.style={font=\footnotesize,inner sep=1.2pt}]
\draw (0,0)--(0,3.7);\draw (2.2,0)--(2.2,3.7);\draw (4.4,0)--(4.4,3.7);
\draw[->] (0,1.1)--(0,1.55);\draw[->] (2.2,1.1)--(2.2,1.55);\draw[->] (4.4,1.1)--(4.4,1.55);
\fill (2.2,2.35) circle (1.25pt);\node[right=3pt] at (2.2,2.35) {$x_a$};
\node[above=3pt] at (0,3.7) {$\cdots$};\node[below=3pt] at (0,0) {$\cdots$};
\node[above=3pt] at (2.2,3.7) {$i_a$};\node[below=3pt] at (2.2,0) {$i_a$};
\node[above=3pt] at (4.4,3.7) {$\cdots$};\node[below=3pt] at (4.4,0) {$\cdots$};
\end{tikzpicture}}}}

\newcommand{\wdWordCross}{\vcenter{\hbox{%
\begin{tikzpicture}[x=.48cm,y=.42cm,line width=.45pt,>=Stealth,
 every node/.style={font=\footnotesize,inner sep=1.2pt}]
\draw (0,0)--(0,3.7); \draw[->] (1.8,0)--(4.0,3.7);\draw[->] (4.0,0)--(1.8,3.7);\draw (5.8,0)--(5.8,3.7);
\draw[->] (0,1.1)--(0,1.55);\draw[->] (5.8,1.1)--(5.8,1.55);
\node[above=3pt] at (0,3.7) {$\cdots$};\node[below=3pt] at (0,0) {$\cdots$};
\node[above=3pt] at (1.8,3.7) {$i_{a+1}$};\node[above=3pt] at (4,3.7) {$i_a$};
\node[below=3pt] at (1.8,0) {$i_a$};\node[below=3pt] at (4,0) {$i_{a+1}$};
\node[above=3pt] at (5.8,3.7) {$\cdots$};\node[below=3pt] at (5.8,0) {$\cdots$};
\end{tikzpicture}}}}

\newcommand{\mrLocalPrimedCup}[2]{\vcenter{\hbox{%
\begin{tikzpicture}[x=.70cm,y=.42cm,line width=.45pt,>=Stealth,
 every node/.style={font=\footnotesize,inner sep=1pt}]
\draw(0,2.45)--(0,.58)..controls(0,.05)and(2.2,.05)..(2.2,.58)--(2.2,2.45);
\draw[->](0,.98)--(0,1.42); \draw[->](2.2,1.42)--(2.2,.98);
\node[above=3pt]at(0,2.45){$i$}; \node[above=3pt]at(2.2,2.45){$i$};
\node[below=2pt]at(1.1,.08){$\eta'$};
\bdBar{-.18}{2.38}{2.45}{-#1}{#2}
\end{tikzpicture}}}}

\newcommand{\mrLocalPrimedCap}[2]{\vcenter{\hbox{%
\begin{tikzpicture}[x=.70cm,y=.42cm,line width=.45pt,>=Stealth,
 every node/.style={font=\footnotesize,inner sep=1pt}]
\draw(0,0)--(0,1.87)..controls(0,2.4)and(2.2,2.4)..(2.2,1.87)--(2.2,0);
\draw[->](0,1.4)--(0,.96); \draw[->](2.2,.96)--(2.2,1.4);
\node[below=3pt]at(0,0){$i$}; \node[below=3pt]at(2.2,0){$i$};
\node[above=2pt]at(1.1,2.3){$\varepsilon'$};
\bdBar{-.18}{2.38}{0}{-#1}{#2}
\end{tikzpicture}}}}

\newcommand{\mrLocalDot}{\vcenter{\hbox{%
\begin{tikzpicture}[x=.52cm,y=.43cm,line width=.45pt,>=Stealth,
 every node/.style={font=\footnotesize,inner sep=1pt}]
\draw(0,0)--(0,2.3); \draw[->](0,.22)--(0,.67); \draw[->](0,1.63)--(0,2.08);
\fill(0,1.15)circle(1.25pt);
\node[below=3pt]at(0,0){$i$}; \node[above=3pt]at(0,2.3){$i$};
\end{tikzpicture}}}}

\newcommand{\mrLocalDotMate}[2]{\vcenter{\hbox{%
\begin{tikzpicture}[x=.47cm,y=.40cm,line width=.45pt,>=Stealth,
 every node/.style={font=\footnotesize,inner sep=1pt}]
\draw(0,0)--(0,3.15)..controls(0,3.7)and(1.9,3.7)..(1.9,3.15)
 --(1.9,.62)..controls(1.9,.07)and(3.8,.07)..(3.8,.62)--(3.8,3.75);
\draw[->](0,1.3)--(0,.86); \draw[->](1.9,.9)--(1.9,1.33);
\draw[->](3.8,2.9)--(3.8,2.47);
\fill(1.9,1.92)circle(1.25pt);
\node[right=2pt]at(1.9,1.92){$\xi_i$};

\node[below=3pt]at(0,0){$i$}; \node[above=3pt]at(3.8,3.75){$i$};
\bdBar{-.18}{3.98}{0}{-#1}{#2}\bdBar{-.18}{3.98}{3.75}{-#1}{#2}
\end{tikzpicture}}}}

\newcommand{\mrLocalOddBubble}{\vcenter{\hbox{%
\begin{tikzpicture}[x=.50cm,y=.42cm,line width=.45pt,>=Stealth,
 every node/.style={font=\footnotesize,inner sep=1pt}]
\draw(0,1.05)circle(.20); \draw(-.14,.91)--(.14,1.19);
\draw(-.14,1.19)--(.14,.91);
\node[below=2pt]at(0,.80){$i$}; \node[right=4pt]at(.22,1.05){$\lambda-\nu_i$};
\end{tikzpicture}}}}

\newcommand{\mrLocalShiftedBubble}[2]{\vcenter{\hbox{%
\begin{tikzpicture}[x=.50cm,y=.42cm,line width=.45pt,>=Stealth,
 every node/.style={font=\footnotesize,inner sep=1pt}]
\draw(0,0)--(0,3.7); \draw[->](0,2.2)--(0,1.78);
\draw(1.25,1.85)circle(.20); \draw(1.11,1.71)--(1.39,1.99);
\draw(1.11,1.99)--(1.39,1.71); \node[below=2pt]at(1.25,1.59){$i$};\node[above=3pt]at(1.25,2.05){$\lambda-\nu_i$};
\foreach \y in {0,3.7}{\bdBar{-.18}{1.90}{\y}{-#1}{#2}}
\node[below=3pt]at(0,0){$i$}; \node[above=3pt]at(0,3.7){$i$};
\end{tikzpicture}}}}

\newcommand{\mrLocalRetypedRightDot}[2]{\vcenter{\hbox{%
\begin{tikzpicture}[x=.55cm,y=.42cm,line width=.45pt,>=Stealth,
 every node/.style={font=\footnotesize,inner sep=1pt}]
\draw(0,0)--(0,3.75);
\draw[->](0,1.0)--(0,.55);\draw[->](0,3.2)--(0,2.75);
\node[draw,fill=white,inner xsep=3pt,inner ysep=2pt]at(0,1.875){$\xi_i^R$};
\mrShiftMark{0}{0}{-#1}{#2}\mrShiftMark{0}{3.75}{-#1}{#2}
\node[below=4pt]at(0,0){$i$};\node[above=4pt]at(0,3.75){$i$};
\end{tikzpicture}}}}

\newcommand{\mrDotSecondMate}{\vcenter{\hbox{%
\begin{tikzpicture}[x=.53cm,y=.42cm,line width=.45pt,>=Stealth,
 every node/.style={font=\footnotesize,inner sep=1pt}]
\draw(0,5.2)--(0,.65)..controls(0,-.05)and(3,-.05)..(3,.65)
 --(3,4.35)..controls(3,5.05)and(6,5.05)..(6,4.35)--(6,0);
\draw[->](0,3.7)--(0,4.2);\draw[->](3,3.85)--(3,3.4);
\draw[->](3,1.8)--(3,1.35);\draw[->](6,.65)--(6,1.1);
\node[draw,fill=white,inner xsep=3pt,inner ysep=2pt]at(3,2.6){$\xi_i^R$};
\mrShiftMark{3}{1.05}{-D}{s}
\mrShiftMark{3}{4.1}{-D}{s}
\node[above=4pt]at(0,6.0){$i$};\node[below=4pt]at(6,-.8){$i$};

\draw(6,0)--(6,-.8);\draw(0,5.2)--(0,6.0);

\end{tikzpicture}}}}

\newcommand{\mrDotSubstitute}{%
\begingroup\def\mrDotSubstitutionFlag{}\mrDotCancel{0}\endgroup}
\newcommand{\mrDotCancel}[1]{\vcenter{\hbox{%
\begin{tikzpicture}[x=.43cm,y=.40cm,line width=.45pt,>=Stealth,
 every node/.style={font=\footnotesize,inner sep=1pt}]
\ifnum#1<3
 \ifcase#1
  \def\leftcup{.3}\def\rightcup{1.25}\def\dotheight{2.7}
 \or
  \def\leftcup{1.25}\def\rightcup{.3}\def\dotheight{2.7}
 \or
  \def\leftcup{2.1}\def\rightcup{.3}\def\dotheight{1.8}
 \fi
 \draw(0,6)--(0,{\leftcup+.5})
  ..controls(0,{\leftcup-.15})and(3,{\leftcup-.15})..(3,{\leftcup+.5})
  --(3,3.8)..controls(3,4.5)and(6,4.5)..(6,3.8)
  --(6,{\rightcup+.5})
  ..controls(6,{\rightcup-.15})and(9,{\rightcup-.15})..(9,{\rightcup+.5})
  --(9,4.8)..controls(9,5.5)and(12,5.5)..(12,4.8)--(12,0);
 \draw[->](0,5.0)--(0,5.45);\draw[->](3,3.65)--(3,3.2);
 \ifnum#1=2
  \draw[->](6,1.2)--(6,1.55);
 \else
  \draw[->](6,{\rightcup+.9})--(6,{\rightcup+1.25});
 \fi
 \draw[->](9,4.5)--(9,4.05);\draw[->](12,.5)--(12,.95);
 \ifdefined\mrDotSubstitutionFlag
  \bdBar{2.8}{9.2}{1.0}{-D}{s}
  \bdBar{2.8}{9.2}{4.7}{-D}{s}
 \else
  \mrShiftMark{3}{\leftcup+.7}{-D}{s}\mrShiftMark{3}{3.7}{-D}{s}
  \mrShiftMark{9}{\rightcup+.7}{-D}{s}\mrShiftMark{9}{4.7}{-D}{s}
 \fi

 \draw(12,0)--(12,-.9);\draw(0,6)--(0,6.8);
 \fill(6,\dotheight)circle(1.25pt);
 \node[right=2pt]at(6,\dotheight){$\xi_i$};
 \node[above=4pt]at(0,6.8){$i$};\node[below=4pt]at(12,-.9){$i$};
\else
 \ifnum#1=3 \def\dotheight{1.8}\else\def\dotheight{5.6}\fi
 \draw(0,6)--(0,.8)..controls(0,.15)and(3,.15)..(3,.8)
  --(3,4.8)..controls(3,5.5)and(6,5.5)..(6,4.8)--(6,0);
 \draw[->](0,3.9)--(0,4.35);\draw[->](3,4.5)--(3,4.05);
 \draw[->](6,.5)--(6,.95);
 \mrShiftMark{3}{1.0}{-D}{s}\mrShiftMark{3}{4.7}{-D}{s}

 \draw(6,0)--(6,-.9);\draw(0,6)--(0,6.8);
 \fill(0,\dotheight)circle(1.25pt);
 \node[right=2pt]at(0,\dotheight){$\xi_i$};
 \node[above=4pt]at(0,6.8){$i$};\node[below=4pt]at(6,-.9){$i$};
\fi
\end{tikzpicture}}}}

\tikzset{grAdj/.style={x=.42cm,y=.36cm,line width=.5pt,>=Stealth,every node/.style={font=\footnotesize,inner sep=2pt}}}
\newcommand{\grCompositeMark}[4]{\draw(#1-.15,#2)--(#1+.15,#2);\node[font=\scriptsize,left=1.5pt]at(#1-.15,#2){$#3$};\node[font=\scriptsize,right=1.5pt]at(#1+.15,#2){$#4$};}

\newcommand{\ccShift}[4]{\draw[black](#1-.12,#2)--(#1+.12,#2);\node[text=black,left=2pt,font=\scriptsize]at(#1-.12,#2){$#3$};\node[text=black,right=2pt,font=\scriptsize]at(#1+.12,#2){$#4$};}
\newcommand{\ccBoundary}[5]{\draw[black](#1-.12,#3)--(#2+.12,#3);\node[text=black,left=2pt,font=\scriptsize]at(#1-.12,#3){$#4$};\node[text=black,right=2pt,font=\scriptsize]at(#2+.12,#3){$#5$};}
\newcommand{\ccParity}{\mathrm p}
\newcommand{\ccBeforeBE}[1]{\vcenter{\hbox{\begin{tikzpicture}[x=.50cm,y=.46cm,line width=.45pt,>=Stealth,every node/.style={font=\footnotesize,inner sep=1pt}]
\draw(-2.0,7)--(-2.0,.85)..controls(-2.0,.15)and(3.5,.15)..(3.5,.85)--(3.5,3.5)--(7.0,4.5)--(7.0,7);
\draw(7.0,0)--(7.0,3.5)--(3.5,4.5)--(3.5,7);
\draw[->](-2.0,5.7)--(-2.0,6.25);\draw[->](3.5,6.25)--(3.5,5.7);\draw[->](7.0,6.25)--(7.0,5.7);\draw[->](7.0,.85)--(7.0,.3);
\node[above=5pt]at(-2.0,7){$i$};\node[above=5pt]at(3.5,7){$j$};\node[above=5pt]at(7.0,7){$i$};\node[below=5pt]at(7.0,0){$j$};
\node[left=4pt]at(-2.0,.45){$\eta'$};
\ccShift{7.0}{0}{-D_{j,\alpha+\nu_i}}{s_{j,0}}
\ccBoundary{-2.0}{7.0}{7}{-D_{j,\alpha}-D_{i,\alpha+\nu_j}}{s_{j,0}+s_{i,0}}
\ifnum#1=0
\ccBoundary{3.5}{7.0}{3.05}{-D_{i,\alpha}-D_{j,\alpha+\nu_i}}{s_{i,0}+s_{j,0}}
\ccBoundary{3.5}{7.0}{5.15}{-D_{j,\alpha}-D_{i,\alpha+\nu_j}}{s_{j,0}+s_{i,0}}
\else
\ccBoundary{-2.0}{7.0}{3.05}{-D_{i,\alpha}-D_{j,\alpha+\nu_i}}{s_{i,0}+s_{j,0}}
\fi
\ifnum#1<3
\ccShift{3.5}{1.55}{-D_{i,\alpha}}{s_{i,0}}
\ccShift{7.0}{2.15}{-D_{j,\alpha+\nu_i}}{s_{j,0}}
\fi
\end{tikzpicture}}}}
\newcommand{\ccCupGraph}[2]{\vcenter{\hbox{\begin{tikzpicture}[x=.78cm,y=.50cm,line width=.45pt,>=Stealth,every node/.style={font=\footnotesize,inner sep=1pt}]
\ifnum#1=0
\draw (0,5.6)--(0,.85)..controls(0,.25)and(1.8,.25)..(1.8,.85)--(1.8,2)--(4.5,3)--(4.5,5.6);
\draw (4.5,0)--(4.5,2)--(1.8,3)--(1.8,5.6);
\node[left=3pt]at(0,.5){$\eta'$};
\draw[->](4.5,.9)--(4.5,.3);
\node[below=5pt]at(4.5,0){$j$};
\else
\draw (1.8,2)--(1.8,.85)..controls(1.8,.25)and(4.5,.25)..(4.5,.85)--(4.5,5.6);
\draw(1.8,2)--(0,3)--(0,5.6);\draw(0,0)--(0,2)--(1.8,3)--(1.8,5.6);
\node[right=8pt]at(4.5,.2){$\eta'$};
\draw[->](0,.9)--(0,.3);\node[below=5pt]at(0,0){$j$};
\fi
\foreach \x/\lab in{0/i,1.8/j,4.5/i}{\node[above=5pt]at(\x,5.6){$\lab$};}
\draw[->](0,4.6)--(0,5.1);\draw[->](1.8,5.1)--(1.8,4.6);\draw[->](4.5,5.1)--(4.5,4.6);
\ifnum#2>0
\ccBoundary{0}{4.5}{5.6}{-D_{j,\alpha}-D_{i,\alpha+\nu_j}}{s_{j,0}+s_{i,0}}
\ifnum#1=0
\ccShift{4.5}{0}{-D_{j,\alpha+\nu_i}}{s_{j,0}}
\ifnum#2=1\ccBoundary{0}{4.5}{1.5}{-D_{i,\alpha}-D_{j,\alpha+\nu_i}}{s_{i,0}+s_{j,0}}\fi
\else
\ccShift{0}{0}{-D_{j,\alpha+\nu_i}}{s_{j,0}}
\ifnum#2=1\ccBoundary{0}{4.5}{1.5}{-D_{j,\alpha+\nu_i}-D_{i,\alpha+\nu_j}}{s_{j,0}+s_{i,0}}\fi
\ifnum#2=3
\ccBoundary{0}{1.8}{1.7}{-D_{j,\alpha+\nu_i}}{s_{j,0}}
\ccBoundary{0}{1.8}{4}{-D_{j,\alpha}}{s_{j,0}}
\ccShift{4.5}{.9}{-D_{i,\alpha+\nu_j}}{s_{i,0}}
\ccShift{4.5}{4.6}{-D_{i,\alpha+\nu_j}}{s_{i,0}}
\fi\fi\fi
\end{tikzpicture}}}}


\newcommand{\ccRawGlobalMarks}{%
  \ccShift{0}{0}{-D_{j,\alpha+\nu_i}}{s_{j,0}}%
  \ccBoundary{0}{1.3}{2}{-D_{j,\alpha+\nu_i}}{s_{j,0}}%
  \ccShift{7.4}{1.35}{-D_{i,\alpha+\nu_j}}{s_{i,0}}%
}
\newcommand{\mrRawOuterCup}{%
  \draw(1.3,2)--(1.3,1.25)..controls(1.3,.45)and(7.4,.45)..(7.4,1.25)--(7.4,7.7);%
  \draw[->](7.4,7.1)--(7.4,6.5);\node[below=3pt]at(4.35,.65){$\eta'$};%
}






\newcommand{\ccRawCapLeft}{\vcenter{\hbox{\begin{tikzpicture}[x=.72cm,y=.58cm,line width=.55pt,>=Stealth,every node/.style={font=\scriptsize,inner sep=1pt}]
  \draw (0,0)--(0,2.1)..controls(0,2.75)and(1,2.75)..(1,2.1)--(1,1.35)--(2,.65)--(2,0);
  \draw (1,0)--(1,.65)--(2,1.35)--(2,2.75);
  \draw[->](0,.5)--(0,.1);\draw[->](1,.1)--(1,.5);\draw[->](1,1.45)--(1,1.9);\draw[->](2,.1)--(2,.5);\draw[->](2,2.15)--(2,2.6);

  \node[below=3pt]at(0,0){$j$};\node[below=3pt]at(1,0){$i$};\node[below=3pt]at(2,0){$j$};
  \node[above=3pt]at(2,2.75){$i$};
\end{tikzpicture}}}}

\newcommand{\ccRawCapRight}{\vcenter{\hbox{\begin{tikzpicture}[x=.72cm,y=.58cm,line width=.55pt,>=Stealth,every node/.style={font=\scriptsize,inner sep=1pt}]
  \draw (0,0)--(0,.65)--(1,1.35)--(1,2.1)..controls(1,2.75)and(2,2.75)..(2,2.1)--(2,0);
  \draw (1,0)--(1,.65)--(0,1.35)--(0,2.75);
  \draw[->](0,.5)--(0,.1);\draw[->](1,.1)--(1,.5);\draw[->](2,.1)--(2,.5);\draw[->](0,2.15)--(0,2.6);\draw[->](1,1.9)--(1,1.45);

  \node[below=3pt]at(0,0){$j$};\node[below=3pt]at(1,0){$i$};\node[below=3pt]at(2,0){$j$};
  \node[above=3pt]at(0,2.75){$i$};
\end{tikzpicture}}}}

\newcommand{\ccRawCupLeft}{\vcenter{\hbox{\begin{tikzpicture}[x=.72cm,y=.55cm,line width=.55pt,>=Stealth,every node/.style={font=\scriptsize,inner sep=1pt}]
  \draw (0,3.9)--(0,.8)..controls(0,.2)and(1.6,.2)..(1.6,.8)--(1.6,1.8)--(3.2,2.7)--(3.2,3.9);
  \draw (3.2,0)--(3.2,1.8)--(1.6,2.7)--(1.6,3.9);
  \draw[->](0,3.05)--(0,3.55);\draw[->](1.6,3.55)--(1.6,3.05);\draw[->](3.2,3.55)--(3.2,3.05);\draw[->](3.2,.9)--(3.2,.35);
  \node[below=4pt]at(3.2,0){$j$};
  \foreach \x/\lab in{0/i,1.6/j,3.2/i}{\node[above=4pt]at(\x,3.9){$\lab$};}
\end{tikzpicture}}}}

\newcommand{\ccRawCupRight}{\vcenter{\hbox{\begin{tikzpicture}[x=.72cm,y=.55cm,line width=.55pt,>=Stealth,every node/.style={font=\scriptsize,inner sep=1pt}]
  \draw (1.6,1.8)--(1.6,.8)..controls(1.6,.2)and(3.2,.2)..(3.2,.8)--(3.2,3.9);
  \draw (1.6,1.8)--(0,2.7)--(0,3.9);\draw (0,0)--(0,1.8)--(1.6,2.7)--(1.6,3.9);
  \draw[->](0,3.05)--(0,3.55);\draw[->](1.6,3.55)--(1.6,3.05);\draw[->](3.2,3.55)--(3.2,3.05);\draw[->](0,.9)--(0,.35);
  \node[below=4pt]at(0,0){$j$};
  \foreach \x/\lab in{0/i,1.6/j,3.2/i}{\node[above=4pt]at(\x,3.9){$\lab$};}
\end{tikzpicture}}}}

\newcommand{\ccRawDownTriangle}{\vcenter{\hbox{\begin{tikzpicture}[x=.72cm,y=.55cm,line width=.55pt,>=Stealth,every node/.style={font=\scriptsize,inner sep=1pt}]
  \draw (0,0)--(0,2.75)..controls(0,3.45)and(1.25,3.45)..(1.25,2.75)--(1.25,1.65)..controls(1.25,1.00)and(2.5,1.00)..(2.5,1.65)--(2.5,4.4);
  \draw[->](0,.85)--(0,.35);\draw[->](1.25,1.95)--(1.25,2.45);\draw[->](2.5,4.0)--(2.5,3.5);

  \node[above=3pt]at(2.5,4.4){$j$};\node[below=3pt]at(0,0){$j$};
\end{tikzpicture}}}}

\newcommand{\ccRawDownIdentity}{\vcenter{\hbox{\begin{tikzpicture}[x=.72cm,y=.55cm,line width=.55pt,>=Stealth,every node/.style={font=\scriptsize,inner sep=1pt}]
  \draw (0,4.4)--(0,0);\draw[->](0,3.6)--(0,3.0);\node[above=3pt]at(0,4.4){$j$};\node[below=3pt]at(0,0){$j$};
\end{tikzpicture}}}}


\begin{notation}\label{not:diagram-conventions}
We use the diagrammatic conventions of
\cite[Definition~1.6, (1.30)]{BE} in the mate calculations below.
We read diagrams from bottom to top; in a horizontal composite the
rightmost functor is applied first. Upward and downward strands represent
induction and restriction, respectively. We label strands only by colours,
or by words when a strand represents a word functor.
The diagrams in \eqref{eq:diagram-conventions} represent
$\rho_g$, $\rho_{x_a e(\mathbf i)}$ ($1\le a\le n$), and
$\rho_{\tau_a e(\mathbf i)}$ ($1\le a<n$), respectively; the omitted
positions carry identity strands. Their sources and targets are given by
\eqref{eq:signed-right-multiplication}.
\begin{equation}\label{eq:diagram-conventions}
\trConventions.
\end{equation}
For a homogeneous map $f$, write $f_{(r,s)}^{(r',s')}$ for the same
linear map between source shift $(r,s)$ and target shift $(r',s')$; its
parity is ${\rm p}(f)+s+s'$. In particular,
\begin{equation}\label{eq:retyped-vs-parity-shift}
q^r\Pi^s(f)=(-1)^{s{\rm p}(f)}f_{(r,s)}^{(r,s)}.
\end{equation}
A horizontal bar records the shift of the factor or composite that it
spans, with the degree shift on its left and the parity shift on its right.
Source and target markings are independent: an unmarked factor has shift
$(0,0)$, and every nonzero shift is shown at its own boundary, even if the
opposite boundary carries the same shift. The same rule applies to empty
boundaries. At each boundary we show
either the shifts of the separate factors or their combined shift, as in
\eqref{eq:envelope-sign}, without recording both.
For $f:F\Rightarrow G$, we represent $f_{(r,s)}^{(r',s')}$ by the following box:
\begin{equation*}
\vcenter{\hbox{\begin{tikzpicture}[x=.75cm,y=.65cm,line width=.45pt,>=Stealth,
 every node/.style={font=\footnotesize,inner sep=1.5pt}]
\draw(0,0)--(0,2.1); \draw[->](0,.28)--(0,.68);
\node[draw,fill=white]at(0,1.05){$f$};
\bdBar{-.16}{.16}{0}{r}{s}
\bdBar{-.16}{.16}{2.1}{r'}{s'}
\node[below=3pt]at(0,0){$F$};\node[above=3pt]at(0,2.1){$G$};
\end{tikzpicture}}}.
\end{equation*}
For vertically composable homogeneous maps $f$ and $g$,
\begin{equation}\label{eq:retyped-vertical-composition}
g_{(r',s')}^{(r'',s'')}\circ f_{(r,s)}^{(r',s')}
=(g\circ f)_{(r,s)}^{(r'',s'')}.
\end{equation}
\end{notation}

The diagrammatic identities in \cite[Definition~1.6, (1.30)]{BE}
take the following form in our notation.
\begin{lem}\label{lem:BE-graphical-rules}
For horizontally composable homogeneous maps
$\varphi:F\Rightarrow F'$ and $\psi:G\Rightarrow G'$, one has
\begin{equation}\label{eq:envelope-sign}
\grSeparateShifts
\;=\;(-1)^{c{\rm p}(\varphi)+b{\rm p}(\psi)+ac+bc}
\grCombinedShifts,
\end{equation}
where the boundary marks are the source and target shifts,
$r,r',t,t'\in\mathbb Z$ and
$a,b,c,d\in\mathbb Z_2$; the parities of
$\varphi,\psi$ on the right are taken before shifting. In this formula, any of
$F,F',G,G'$ may be an identity functor; its boundary then has no strands
but retains the displayed bar and labels.
For homogeneous $u:F\Rightarrow F'$ and $v:G\Rightarrow G'$, where
$F,F':\mathcal C\to\mathcal D$ and $G,G':\mathcal D\to\mathcal E$,
the diagrammatic superinterchange law is
\begin{equation}\label{eq:graphical-interchange}
\trInterOne\;=\;(-1)^{{\rm p}(u){\rm p}(v)}\;\trInterTwo.
\end{equation}
\end{lem}

\begin{notation}\label{not:BE-boundary-shifts}
For $i\in I$, $r\in\mathbb Z$ and $s\in\mathbb Z_2$, the marked
strands in \eqref{eq:boundary-shift-conventions} denote the indicated
identity maps, using the boundary notation of \cite[Definition~1.6]{BE}:
\begingroup\postdisplaypenalty=10000\interlinepenalty=10000
\begin{equation}\label{eq:boundary-shift-conventions}
\vcenter{\hbox{\begin{tikzpicture}[x=.85cm,y=.85cm,line width=.48pt,>=Stealth,
 every node/.style={font=\footnotesize,inner sep=2pt}]
\begin{scope}
\draw (0,0)--(0,1.65);
\draw[->] (0,.57)--(0,1.07);
\foreach \y in {0,1.65}{
 \bdBar{-.15}{.15}{\y}{r}{s}}
\node[above=6pt] at (0,1.65) {$i$};
\node[below=6pt] at (0,0) {$i$};
\node[font=\normalsize] at (1.25,.825) {$=$};
\node[anchor=west,font=\normalsize] at (1.85,.825) {$\displaystyle 1_{q^r\Pi^{s}\mathsf F_i}$};
\end{scope}
\begin{scope}[xshift=7cm]
\draw (0,0)--(0,1.65);
\draw[->] (0,1.07)--(0,.57);
\foreach \y in {0,1.65}{
 \bdBar{-.15}{.15}{\y}{r}{s}}
\node[above=6pt] at (0,1.65) {$i$};
\node[below=6pt] at (0,0) {$i$};
\node[font=\normalsize] at (1.25,.825) {$=$};
\node[anchor=west,font=\normalsize] at (1.85,.825) {$\displaystyle 1_{q^r\Pi^{s}\mathsf E_i}$};
\end{scope}
\end{tikzpicture}}},
\end{equation}
where the labels $r$ on the left and $s$ on the right specify the degree and
parity shifts, respectively.
\par\endgroup\pagebreak[0]
\end{notation}

\Needspace{12\baselineskip}
\begingroup\postdisplaypenalty=1000
\begin{notation}\label{not:local-adjunction-maps}
For $\alpha\in Q^+$ and $i\in I$, use the leftward cup $\eta'$ and cap
$\varepsilon'$ of \cite[(1.16), (2.10)--(2.11)]{BE} in the components
$\alpha+\nu_i-\Lambda$ and $\alpha-\Lambda$, respectively, with the
normalization \eqref{eq:c-normalization}.  The unit and counit of the second
adjunction in \eqref{eq:two-adjunctions} are defined by
\begin{equation}\label{eq:normalized-local-bends}
\widehat{\eta}_{i,\alpha}:=(-1)^{s_{i,\alpha}}\,
\mrLocalPrimedCup{D_{i,\alpha}}{s_{i,\alpha}},
\qquad
\widehat{\varepsilon}_{i,\alpha}:=\mrLocalPrimedCap{D_{i,\alpha}}{s_{i,\alpha}},
\end{equation}
where
\[
 \widehat{\eta}_{i,\alpha}:\operatorname{Id}_{R^\Lambda_{\alpha+\nu_i}\text{-}\mathrm{gsmod}}
 \Rightarrow\mathsf F_{i,\alpha}q^{-D_{i,\alpha}}\Pi^{s_{i,\alpha}}\mathsf E_{i,\alpha},
 \qquad
 \widehat{\varepsilon}_{i,\alpha}:q^{-D_{i,\alpha}}\Pi^{s_{i,\alpha}}
 \mathsf E_{i,\alpha}\mathsf F_{i,\alpha}
 \Rightarrow\operatorname{Id}_{R^\Lambda_\alpha\text{-}\mathrm{gsmod}}.
\]
The maps have
$(\deg\widehat{\eta}_{i,\alpha},{\rm p}(\widehat{\eta}_{i,\alpha}))
=(\deg\widehat{\varepsilon}_{i,\alpha},{\rm p}(\widehat{\varepsilon}_{i,\alpha}))=(0,\bar0)$;
$D_{i,\alpha}$ and $s_{i,\alpha}$ are defined in \eqref{eq:local-shifts}.
\end{notation}
\par\endgroup

For $\mathbf k=(k_1,\ldots,k_n)\in I^\alpha$, put
$\beta_0=0$ and $\beta_a=\sum_{b=1}^{a}\nu_{k_b}$ for $1\le a\le n$.
\begin{defn}\label{def:word-adjunction-data}
For the first adjunction in \eqref{eq:word-adjunctions}, take the iterated
standard induction--restriction unit and counit; for the second, iterate
\eqref{eq:normalized-local-bends}:
\begin{equation}\label{eq:trace-adjunction-maps}
 \begin{aligned}
 \eta_{\mathbf k}&:\operatorname{Id}_{\mathbb K\text{-}\mathrm{gsmod}}
          \longrightarrow \mathsf{E}_{\mathbf k}\mathsf{F}_{\mathbf k},\\
 \varepsilon_{\mathbf k}&:\mathsf{F}_{\mathbf k}\mathsf{E}_{\mathbf k}
          \longrightarrow\operatorname{Id}_{R^\Lambda_\alpha\text{-}\mathrm{gsmod}},\\
 \widehat{\eta}_{\mathbf k}&:\operatorname{Id}_{R^\Lambda_\alpha\text{-}\mathrm{gsmod}}
          \longrightarrow \mathsf{F}_{\mathbf k}q^{-d_{\Lambda,\alpha}}
                    \Pi^{\pi_{\Lambda,\alpha}}\mathsf{E}_{\mathbf k},\\
 \widehat{\varepsilon}_{\mathbf k}&:q^{-d_{\Lambda,\alpha}}\Pi^{\pi_{\Lambda,\alpha}}
                   \mathsf{E}_{\mathbf k}\mathsf{F}_{\mathbf k}
          \longrightarrow\operatorname{Id}_{\mathbb K\text{-}\mathrm{gsmod}}.
 \end{aligned}
\end{equation}
Diagrammatically,
\begin{equation}\label{eq:adjunction-diagrams}
\begin{gathered}\begin{array}{cc}\trCup&\trCap\end{array}\\[3pt]\trPrimeCup\\[3pt]\trPrimeCap\end{gathered}.
\end{equation}
The ellipses indicate the intervening nested arcs and their local shifts.
For $n=0$ the diagrams are identities.
Rightward cups and caps denote $\eta$ and $\varepsilon$, and leftward ones
denote $\eta'$ and $\varepsilon'$, with labels omitted when unambiguous.
\end{defn}

\Needspace{28\baselineskip}
\begin{lem}\label{lem:word-triangle-identities}
For $u\in\{\eta_{\mathbf k},\varepsilon_{\mathbf k},
\widehat{\eta}_{\mathbf k},\widehat{\varepsilon}_{\mathbf k}\}$,
$(\deg u,{\rm p}(u))=(0,\bar0)$. The triangle identities are
\begin{equation}\label{eq:triangle-diagrams}
\begin{gathered}\begin{array}{cc}\trTriF&\trTriE\end{array}\\[4pt]\trTriPrimeF\\[4pt]\trTriPrimeE\end{gathered}.
\end{equation}
\end{lem}
\begin{proof}
These are the triangle identities for the two adjunctions in
\eqref{eq:word-adjunctions}.
\end{proof}

\begin{defn}\label{def:successive-mates}
For a homogeneous natural transformation
$\rho:\mathsf F_{\mathbf j}\Rightarrow\mathsf F_{\mathbf i}$, define its
right mate by
\begin{equation}\label{eq:right-mate-composite}
 \rho^R=(\mathsf{E}_{\mathbf j}\varepsilon_{\mathbf i})\circ
       (\mathsf{E}_{\mathbf j}\rho\mathsf E_{\mathbf i})\circ
       (\eta_{\mathbf j}\mathsf{E}_{\mathbf i}):\mathsf{E}_{\mathbf i}\longrightarrow \mathsf{E}_{\mathbf j}.
\end{equation}
Applying the second adjunction in \eqref{eq:word-adjunctions} to $\rho^R$ gives the successive mate
$\rho^{RL}$:
\begin{equation}\label{eq:double-mate-composite}
 \begin{aligned}
 \rho^{RL}={}&(\mathsf{F}_{\mathbf i}\widehat{\varepsilon}_{\mathbf j})\circ
 \bigl(\mathsf{F}_{\mathbf i}(q^{-d_{\Lambda,\alpha}}\Pi^{\pi_{\Lambda,\alpha}}\rho^R)
             \mathsf{F}_{\mathbf j}\bigr)\circ(\widehat{\eta}_{\mathbf i}\mathsf{F}_{\mathbf j})\\
 &\hspace{40mm}:\mathsf{F}_{\mathbf j}\longrightarrow \mathsf{F}_{\mathbf i}.
 \end{aligned}
\end{equation}
\begin{equation}\label{eq:right-mate-diagram}
\trRightMate[\rho]=\trRightPlain[\rho].
\end{equation}
\begin{align}
(-1)^{\pi_{\Lambda,\alpha}{\rm p}(\rho)+\sum_a s_{i_a,0}+\sum_{a<b}s_{i_a,0}s_{i_b,0}}\,\trSecondMate[\rho]
&\xlongequal{\eqref{eq:retyped-vs-parity-shift},\,\eqref{eq:double-mate-composite}}
\trDoublePlain[\rho].\label{eq:second-mate-diagram}
\end{align}
\end{defn}
\enlargethispage{2pt}

\Needspace{8\baselineskip}
\begingroup\postdisplaypenalty=1000
\begin{notation}
The local dot is
\begin{equation*}
\xi_i := \mrLocalDot,\qquad {\rm p}(\xi_i)={\rm p}(i).
\end{equation*}
Thus $\xi_i:\mathsf F_{i,\alpha}\Rightarrow\mathsf F_{i,\alpha}$.
For $\lambda=\Lambda-\alpha$ with $\alpha\in Q^+$, we use the odd bubble
of \cite[(1.23)]{BE}, denoted here by
$B_{i,\lambda-\nu_i}$:
\begin{equation*}
B_{i,\lambda-\nu_i} := \mrLocalOddBubble,
\qquad {\rm p}(B_{i,\lambda-\nu_i})=\bar1
\qquad(i\in I_{\rm odd}).
\end{equation*}
It is a $2$-endomorphism of $1_{\lambda-\nu_i}$. Compose on the left with $\mathsf E_{i,\alpha}$, using the boundary
shifts in \eqref{eq:boundary-shift-conventions}, and define
\begin{equation}\label{eq:shifted-odd-bubble}
\widetilde B_{i,\alpha} := \mrLocalShiftedBubble{D_{i,\alpha}}{s_{i,\alpha}},
\qquad {\rm p}(\widetilde B_{i,\alpha})=\bar1
\qquad(i\in I_{\rm odd}).
\end{equation}
By \eqref{eq:retyped-vs-parity-shift}, \cite[(1.26)--(1.27)]{BE} gives
\begin{equation}\label{eq:BE-even-dot}
\mrLocalRetypedRightDot{D_{i,\alpha}}{s_{i,\alpha}}
=
 \mrLocalDotMate{D_{i,\alpha}}{s_{i,\alpha}}
\qquad(i\in I_{\rm even}),
\end{equation}
\begin{equation}\label{eq:BE-odd-dot}
\mrLocalRetypedRightDot{D_{i,\alpha}}{s_{i,\alpha}}
=2\,\mrLocalShiftedBubble{D_{i,\alpha}}{s_{i,\alpha}}
-\mrLocalDotMate{D_{i,\alpha}}{s_{i,\alpha}}
\quad(i\in I_{\rm odd}),
\end{equation}
where $\xi_i^R$ is the right mate in \eqref{eq:right-mate-diagram} and
${\rm p}(\xi_i^R)={\rm p}(\xi_i)={\rm p}(i)$.
For $i\in I_{\rm odd}$, $\alpha\in Q^+$ and $\mathbf k\in I^\alpha$,
\cite[(1.25)]{BE} gives
\begin{equation}\label{eq:BE-bubble-centrality}
\trBubbleBefore=\trBubbleAfter,
\end{equation}
where the left and right bubbles are $B_{i,\Lambda-\alpha}$ and
$B_{i,\Lambda}$, respectively.
\end{notation}
\par\endgroup

\begin{notation}\label{not:BE-crossings}
For the crossing, fix $i,j\in I$ and $\alpha\in Q^+$. The upward crossing has type
\begin{equation*}
\mrUpCross:
\mathsf F_{i,\alpha+\nu_j}\mathsf F_{j,\alpha}
\Longrightarrow
\mathsf F_{j,\alpha+\nu_i}\mathsf F_{i,\alpha}.
\end{equation*}
The downward crossing of \cite[(2.1)]{BE} is the right mate
in \eqref{eq:right-mate-diagram}:
\begin{equation}\label{eq:direct-crossing-mates}
\mrRightCrossDefinition :=\left(\mrUpCross\right)^R:
\mathsf E_{i,\alpha}\mathsf E_{j,\alpha+\nu_i}
\Longrightarrow\mathsf E_{j,\alpha}\mathsf E_{i,\alpha+\nu_j}.
\end{equation}
The mixed crossing is defined by \cite[(2.6)]{BE}, with $i,j$
interchanged:
\begin{equation}\label{eq:direct-mixed-crossing}
\mrMixedCrossDefinition:
\mathsf E_{j,\alpha+\nu_i}\mathsf F_{i,\alpha+\nu_j}
\Longrightarrow
\mathsf F_{i,\alpha}\mathsf E_{j,\alpha}.
\end{equation}
Each of these three crossings has parity ${\rm p}(i){\rm p}(j)$.
\end{notation}

The first relation of \cite[(7.5)]{BE}, in the weight
$\alpha-\Lambda$ there, gives
\begin{equation}\label{eq:raw-BE-cap}
\ccRawCapLeft\;=\;\ccRawCapRight.
\end{equation}
Its source is
$\mathsf E_{j,\alpha+\nu_i}\mathsf F_{i,\alpha+\nu_j}\mathsf F_{j,\alpha}$
and its target is $\mathsf F_{i,\alpha}$.
The first relation of \cite[(7.6)]{BE}, in the weight
$\alpha+\nu_i+\nu_j-\Lambda$ there, gives
\begin{equation}\label{eq:raw-BE-cup}
\ccRawCupLeft\;=\;
(-1)^{{\rm p}(i){\rm p}(j)}\,\ccRawCupRight.
\end{equation}
Here both sides have source $\mathsf E_{j,\alpha+\nu_i}$ and target
$\mathsf F_{i,\alpha}\mathsf E_{j,\alpha}\mathsf E_{i,\alpha+\nu_j}$;
the left and right cups are the components of $\eta_i'$ at the weights
$\alpha+\nu_i-\Lambda$ and $\alpha+\nu_i+\nu_j-\Lambda$, respectively.
Finally, the downward triangle in \cite[(1.18)]{BE} is
\begin{equation}\label{eq:raw-BE-down-triangle}
\ccRawDownTriangle\;=\;\ccRawDownIdentity.
\end{equation}
This is an identity on $\mathsf E_{j,\alpha}$, using the raw $j$-cup and
$j$-cap.  By \eqref{eq:retyped-vertical-composition}, each of
\eqref{eq:raw-BE-cap}--\eqref{eq:raw-BE-down-triangle} remains valid when
its two complete diagrams have the same source and target boundary shifts.

\subsection{Computation of the Nakayama automorphism}\label{subsec:nakayama}

\begin{lem}\label{lem:multiplication-and-mates}
For $\mathbf i,\mathbf j\in I^\alpha$ and homogeneous
$g\in e(\mathbf j)R^\Lambda_\alpha e(\mathbf i)$, the map $\rho_g$ in
\eqref{eq:signed-right-multiplication} is a homogeneous natural
transformation of bidegree $(\deg g,{\rm p}(g))$. For homogeneous
$b\in e(\mathbf i)R^\Lambda_\alpha e(\mathbf j)$,
\begin{equation*}
 \rho_g\circ\rho_b=(-1)^{{\rm p}(g){\rm p}(b)}\rho_{bg}.
\end{equation*}
For every $M\in R^\Lambda_\alpha\text{-}\mathrm{gsmod}$ and homogeneous
$v\in e(\mathbf i)M$,
\begin{equation}\label{eq:right-mate-left-multiplication}
 (\rho_g^R)_M(v)=gv\in e(\mathbf j)M.
\end{equation}
\end{lem}
\begin{proof}
For
$V,W\in\mathbb K\text{-}\mathrm{gsmod}$, homogeneous
$a\in R^\Lambda_\alpha$, $m\in R^\Lambda_\alpha e(\mathbf j)$
and $v\in V$, \eqref{eq:signed-right-multiplication} gives
\[
 (\rho_g)_V(am\otimes v)
 =(-1)^{{\rm p}(g)({\rm p}(a)+{\rm p}(m))}amg\otimes v
 =(-1)^{{\rm p}(g){\rm p}(a)}a(\rho_g)_V(m\otimes v).
\]
For homogeneous $f:V\to W$,
\begingroup\postdisplaypenalty=10000\interlinepenalty=10000
\[
 (\rho_g)_W \mathsf{F}_{\mathbf j}(f)
 =(-1)^{{\rm p}(g){\rm p}(f)}\mathsf{F}_{\mathbf i}(f)(\rho_g)_V,
\]
using \eqref{eq:signed-right-multiplication} and
$\mathsf F_{\mathbf k}(f)(m\otimes v)
=(-1)^{{\rm p}(f){\rm p}(m)}m\otimes f(v)$.
\par\endgroup\pagebreak[0]
Applying \eqref{eq:signed-right-multiplication} twice, for homogeneous
$m\in R^\Lambda_\alpha e(\mathbf i)$ we obtain
\begin{align*}
 ((\rho_g)_V\circ(\rho_b)_V)(m\otimes v)
 &=(-1)^{{\rm p}(b){\rm p}(m)+{\rm p}(g)({\rm p}(m)+{\rm p}(b))}
       mbg\otimes v\\
 &=(-1)^{{\rm p}(g){\rm p}(b)}(\rho_{bg})_V(m\otimes v).
 \end{align*}
For $M\in R^\Lambda_\alpha\text{-}\mathrm{gsmod}$ and $v\in e(\mathbf i)M$,
\begingroup\postdisplaypenalty=10000\interlinepenalty=10000
\[
v\xmapsto{(\eta_{\mathbf j}\mathsf E_{\mathbf i})_M}e(\mathbf j)\otimes v
\xmapsto{(\mathsf E_{\mathbf j}\rho_g\mathsf E_{\mathbf i})_M}g\otimes v
\xmapsto{(\mathsf E_{\mathbf j}\varepsilon_{\mathbf i})_M}gv,
\]
where the composite is \eqref{eq:right-mate-composite}, and
${\rm p}(e(\mathbf j))=\bar0$ gives
$(-1)^{{\rm p}(g){\rm p}(e(\mathbf j))}=1$
in \eqref{eq:signed-right-multiplication}.\qedhere
\par\endgroup
\end{proof}

\begin{lem}\label{lem:mate-counit-trace}
For $\mathbf i,\mathbf j\in I^\alpha$ and homogeneous
$g\in e(\mathbf j)R^\Lambda_\alpha e(\mathbf i)$, one has
\begingroup\postdisplaypenalty=10000\interlinepenalty=10000
\begin{equation}\label{eq:mate-counit-identity}
 \begin{aligned}
 \widehat{\varepsilon}_{\mathbf i}\circ
   (q^{-d_{\Lambda,\alpha}}\Pi^{\pi_{\Lambda,\alpha}}\mathsf{E}_{\mathbf i}\rho_g^{RL})
 =\widehat{\varepsilon}_{\mathbf j}\circ
   ((q^{-d_{\Lambda,\alpha}}\Pi^{\pi_{\Lambda,\alpha}}\rho_g^R)\mathsf{F}_{\mathbf j}),
 \end{aligned}
\end{equation}
where both sides have source
$q^{-d_{\Lambda,\alpha}}\Pi^{\pi_{\Lambda,\alpha}}\mathsf{E}_{\mathbf i}\mathsf{F}_{\mathbf j}$
and target $\operatorname{Id}_{\mathbb K\text{-}\mathrm{gsmod}}$.
\par\endgroup\pagebreak[0]
For every homogeneous $b\in e(\mathbf i)R^\Lambda_\alpha e(\mathbf j)$,
\begin{equation}\label{eq:mixed-counit-computation}
 t_{\Lambda,\alpha}(gb)
 =t_{\Lambda,\alpha}\bigl((\rho_g^{RL})_{\mathbb K}(b)\bigr).
\end{equation}
\end{lem}
\begin{proof}
\begingroup\allowdisplaybreaks[3]\setlength{\jot}{0pt}
\begin{align*}
&\trCounitStart\\
&\xlongequal{\eqref{eq:second-mate-diagram}}
(-1)^{\pi_{\Lambda,\alpha}{\rm p}(g)+\sum_a s_{i_a,0}+\sum_{a<b}s_{i_a,0}s_{i_b,0}}\,\trCounitExpand\\
&\xlongequal{\eqref{eq:graphical-interchange}}
(-1)^{\pi_{\Lambda,\alpha}{\rm p}(g)+\sum_a s_{i_a,0}+\sum_{a<b}s_{i_a,0}s_{i_b,0}}\,\trCounitMove\\
&\xlongequal{\eqref{eq:triangle-diagrams}}
(-1)^{\pi_{\Lambda,\alpha}{\rm p}(g)}\,\trCounitEnd,
\end{align*}
The parity-transport factor is
introduced by \eqref{eq:retyped-vs-parity-shift}, and the remaining factor
is the coefficient of the flattened raw word unit in
\eqref{eq:adjunction-diagrams}.
Here $({\rm p}(\widehat{\varepsilon}_{\mathbf i}),{\rm p}(\widehat{\varepsilon}_{\mathbf j}))
=(\bar0,\bar0)$ by \eqref{eq:trace-adjunction-maps},
and ${\rm p}(q^{-d_{\Lambda,\alpha}}\Pi^{\pi_{\Lambda,\alpha}}\rho_g^R)={\rm p}(g)$;
the last equality uses
the last triangle in \eqref{eq:triangle-diagrams} with $\mathbf k=\mathbf i$.
\par\endgroup\pagebreak[0]
Taking the $\mathbb K$-component of \eqref{eq:mate-counit-identity} and evaluating at $b$ gives
\begingroup\postdisplaypenalty=10000\interlinepenalty=10000
\begin{equation*}
\begin{aligned}
(-1)^{\pi_{\Lambda,\alpha}{\rm p}(g)}
t_{\Lambda,\alpha}\bigl((\rho_g^{RL})_{\mathbb K}(b)\bigr)
&=(-1)^{\pi_{\Lambda,\alpha}{\rm p}(g)}t_{\Lambda,\alpha}(gb),\\ {\rm i.e.},\quad
t_{\Lambda,\alpha}\bigl((\rho_g^{RL})_{\mathbb K}(b)\bigr)
&=t_{\Lambda,\alpha}(gb),
\end{aligned}
\end{equation*}
using \eqref{eq:parity-shift-on-maps},
\eqref{eq:right-mate-left-multiplication} and \eqref{eq:normalized-trace}, with
${\rm p}(\rho_g^{RL})={\rm p}(\rho_g^R)={\rm p}(g)$.\qedhere
\par\endgroup
\end{proof}

Equation \eqref{eq:mixed-counit-computation} reduces the trace identity
to computing $\rho_g^{RL}$ for
$g=e(\mathbf i),x_ae(\mathbf i),\tau_ae(\mathbf i)$.

\begin{lem}\label{lem:generator-double-mates}
For $\mathbf i\in I^\alpha$, the mates defined by
\eqref{eq:right-mate-composite}--\eqref{eq:double-mate-composite} satisfy
\begin{align}
 (\rho_{e(\mathbf i)})^{RL}
   &= \rho_{e(\mathbf i)}, \label{eq:idempotent-double-mate} \\[2pt]
 (\rho_{x_a e(\mathbf i)})^{RL}
   &= (-1)^{{\rm p}(i_a)+s_{i_a,0}}\rho_{x_a e(\mathbf i)}
      & &\qquad(1\le a\le n), \label{eq:dot-double-mate} \\[2pt]
 (\rho_{\tau_a e(\mathbf i)})^{RL}
   &= (-1)^{{\rm p}(i_a){\rm p}(i_{a+1})+s_{i_a,0}s_{i_{a+1},0}}
      \rho_{\tau_a e(\mathbf i)}
      & &\qquad(1\le a<n). \label{eq:crossing-double-mate}
\end{align}
\end{lem}
\begin{proof}
\postdisplaypenalty=4000
For $g=e(\mathbf i)$, the triangle diagrams
\eqref{eq:triangle-diagrams} give
$\rho_{e(\mathbf i)}=1_{\mathsf F_{\mathbf i}}$,
$(\rho_{e(\mathbf i)})^R=1_{\mathsf E_{\mathbf i}}$ and
$(\rho_{e(\mathbf i)})^{RL}=1_{\mathsf F_{\mathbf i}}$, proving
\eqref{eq:idempotent-double-mate}.

\smallskip
\Needspace{10\baselineskip}
\noindent\textbf{Step 1. Compute the dot.}
Fix $i\in I$ and $\alpha\in Q^+$, and put $\lambda=\Lambda-\alpha$.

\smallskip
\noindent\textbf{Case I. $i\in I_{\rm even}$.}
Since $({\rm p}(i),s_{i,\alpha})=(\bar0,\bar0)$ by \eqref{eq:local-shifts},
\begin{equation*}
\mrLocalRetypedRightDot{D_{i,\alpha}}{s_{i,\alpha}}
\;\xlongequal{\eqref{eq:BE-even-dot}}\;
\mrLocalDotMate{D_{i,\alpha}}{s_{i,\alpha}}.
\end{equation*}

\smallskip
\noindent\textbf{Case II. $i\in I_{\rm odd}$.}
The bubble term vanishes at the highest weight:
\begin{equation}\label{eq:highest-odd-bubble}
(B_{i,\Lambda})_{\mathbb K}\in
\operatorname{End}_{\mathbb K}(\mathbb K)_{\bar1}=0,
\end{equation}
where ${\rm p}(B_{i,\Lambda})=\bar1$ and $\mathbb K_{\bar1}=0$.
\par\Needspace{13\baselineskip}
\begingroup\postdisplaypenalty=10000
For $\gamma\in Q^+$ and $\mathbf k\in I^\gamma$,
\begin{equation*}
\trBubbleBefore[\gamma]
\;\xlongequal{\eqref{eq:BE-bubble-centrality}}\;
\trBubbleAfter[\gamma]
\;\xlongequal{\eqref{eq:highest-odd-bubble}}\;0,
\end{equation*}
where $R^\Lambda_\gamma=\bigoplus_{\mathbf k\in I^\gamma}\mathsf F_{\mathbf k}\mathbb K$.
Thus $B_{i,\lambda-\nu_i}=0$ on $R^\Lambda_{\alpha+\nu_i}$ and hence on
$R^\Lambda_{\alpha+\nu_i}\text{-}\mathrm{gsmod}$, so \eqref{eq:shifted-odd-bubble} gives
$\widetilde B_{i,\alpha}=0$.
\par\endgroup
Substituting $\widetilde B_{i,\alpha}=0$ into \eqref{eq:BE-odd-dot}
and combining with Case~I, we obtain
\begin{equation}\label{eq:combined-retyped-dot}
\mrLocalRetypedRightDot{D_{i,\alpha}}{s_{i,\alpha}}
=(-1)^{{\rm p}(i)}\,
\mrLocalDotMate{D_{i,\alpha}}{s_{i,\alpha}}.
\end{equation}
In the following calculation, write
$D=D_{i,\alpha}$ and $s=s_{i,\alpha}$. The second mate is
\begingroup\allowdisplaybreaks[3]\setlength{\jot}{0pt}
\begin{align}
\xi_i^{RL}
&\xlongequal{\eqref{eq:second-mate-diagram}}
 (-1)^{s{\rm p}(i)+s}\,
 \mrDotSecondMate\notag\\
&\xlongequal{\eqref{eq:combined-retyped-dot}}
 (-1)^{s{\rm p}(i)+s+{\rm p}(i)}\,
 \mrDotSubstitute\notag\\
&\xlongequal{\eqref{eq:envelope-sign}}
 (-1)^{s{\rm p}(i)+s+{\rm p}(i)+s}\,
 \mrDotCancel{0}\notag\\
&\xlongequal{\eqref{eq:graphical-interchange}\ \text{(twice)}}
 (-1)^{s{\rm p}(i)+{\rm p}(i)}\,
 \mrDotCancel{2}\notag\\
&\xlongequal{\eqref{eq:triangle-diagrams}}
 (-1)^{s{\rm p}(i)+{\rm p}(i)+s}\,
 \mrDotCancel{3}\notag\\
&\xlongequal{\eqref{eq:graphical-interchange}}
 (-1)^{s{\rm p}(i)+{\rm p}(i)+s}\,
 \mrDotCancel{4}\notag\\
&\xlongequal{\eqref{eq:triangle-diagrams}}
 (-1)^{s{\rm p}(i)+{\rm p}(i)}\,
 \mrCancelLast{\xi_i}{i}{i}.
\label{eq:dot-cancellation}
\end{align}
\endgroup
The marked cups and caps are even, so the three interchanges have coefficient $1$.
By \eqref{eq:local-parity-independent},
$s_{i,\alpha}{\rm p}(i)+{\rm p}(i)=s_{i,0}+{\rm p}(i)$ in $\mathbb Z_2$.

\smallskip
\Needspace{11\baselineskip}
\noindent\textbf{Step 2. Compute the crossing.}
For $i,j\in I$ and $\alpha\in Q^+$,
\eqref{eq:local-shifts} and ${\rm d}_ia_{ij}={\rm d}_ja_{ji}$ give
\begin{equation*}
D_{i,\alpha}+D_{j,\alpha+\nu_i}
=D_{j,\alpha}+D_{i,\alpha+\nu_j},
\qquad
D_{j,\alpha+\nu_i}-D_{j,\alpha}=-2{\rm d}_ja_{ji}.
\end{equation*}
The boundary marks below denote the retyping of
Notation~\ref{not:BE-boundary-shifts}; the shift functor acts as in
\eqref{eq:retyped-vs-parity-shift}.
We retain the same complete source and target shifts in every diagram:
\begingroup\allowdisplaybreaks[3]\setlength{\jot}{0pt}
\begin{align}
&\mrCrossStage{0}\notag\\
&\xlongequal{\eqref{eq:raw-BE-cap}}\mrCrossStage{1}\notag\\
&\xlongequal{\eqref{eq:graphical-interchange}}
(-1)^{s_{j,0}{\rm p}(i){\rm p}(j)}\,\mrCrossStage{2}\notag\\
&\xlongequal{\eqref{eq:raw-BE-down-triangle}}
(-1)^{s_{j,0}{\rm p}(i){\rm p}(j)}\,\mrCrossStage{3}\notag\\
&\xlongequal{\eqref{eq:raw-BE-cup}}
(-1)^{(1+s_{j,0}){\rm p}(i){\rm p}(j)}\,\mrCrossStage{4}\notag\\
&\xlongequal{\eqref{eq:graphical-interchange}}
(-1)^{(1+s_{i,0}+s_{j,0}){\rm p}(i){\rm p}(j)}\,\mrCrossStage{5}\notag\\
&\xlongequal{\eqref{eq:raw-BE-down-triangle}}
(-1)^{(1+s_{i,0}+s_{j,0}){\rm p}(i){\rm p}(j)}\,\mrCrossStage{6}.
\label{eq:crossing-primed-calculation}
\end{align}
\endgroup
Equation \eqref{eq:crossing-primed-calculation} is the sign-corrected form of
\cite[(1.28)]{BE} in our conventions; see
\cite[(13.18) and footnote~8]{BK:spin} for the case $i=j\in I_{\rm odd}$.
The same calculation as in \eqref{eq:dot-cancellation}, using
\eqref{eq:second-mate-diagram} and \eqref{eq:crossing-primed-calculation}, gives
\begin{equation}\label{eq:crossing-cancellation-last}
\left(\mrUpCross\right)^{RL}
=(-1)^{{\rm p}(i){\rm p}(j)+s_{i,0}s_{j,0}}\mrUpCross.
\end{equation}

\medskip
\noindent\textbf{Step 3. Insert the local maps in the word.}
For horizontally composable homogeneous natural transformations
$u:F\Rightarrow F'$ and $v:G\Rightarrow G'$, \eqref{eq:Pi-can-iso} and \eqref{eq:parity-shift-on-maps} give
\begin{equation*}
 \Pi^s(v)\,\Pi^t(u)=\Pi^{s+t}(vu):\quad
 \Pi^{s+t}(GF)\Rightarrow\Pi^{s+t}(G'F').
\end{equation*}
The triangle identities \eqref{eq:triangle-diagrams} then give, for
composable functors $H,K$,
\begin{equation}\label{eq:word-double-mate-insertion}
(Hu K)^{RL}=Hu^{RL}K.
\end{equation}
In the following diagrams, the unshown strands are identities.  With
$\mathbf i=(i_1,\ldots,i_n)$, the dot diagram has source and target
$\mathsf F_{\mathbf i}$; the crossing diagram has source
$\mathsf F_{\mathbf i}$ and target $\mathsf F_{s_a\mathbf i}$.
\noindent\begin{minipage}[t]{.48\linewidth}
\begin{equation*}
\begin{gathered}
\rho_{x_ae(\mathbf i)}=\wdWordDot\\[-2pt]
(1\le a\le n),
\end{gathered}
\end{equation*}
\end{minipage}\hfill
\begin{minipage}[t]{.48\linewidth}
\begin{equation*}
\begin{gathered}
\rho_{\tau_ae(\mathbf i)}=\wdWordCross\\[-2pt]
(1\le a<n),
\end{gathered}
\end{equation*}
\end{minipage}\par
Substituting \eqref{eq:dot-cancellation} and
\eqref{eq:crossing-cancellation-last}, then applying
\eqref{eq:word-double-mate-insertion}, proves
\eqref{eq:dot-double-mate} and \eqref{eq:crossing-double-mate}.\qedhere
\end{proof}

To state the Nakayama formula uniformly on the generators, we introduce the
following auxiliary $\mathbb Z_2$-degree.
\begin{defn}\label{def:auxiliary-parity}
For $\mathbf i\in I^\alpha$, define the auxiliary parity $\mathsf r(g)\in\mathbb Z_2$ of each generator $g$ by
\begingroup\postdisplaypenalty=10000\interlinepenalty=10000
\begin{equation}\label{eq:auxiliary-parity}
\begin{aligned}
 \mathsf r(e(\mathbf i))&=\bar0,\\
 \mathsf r(x_a e(\mathbf i))&={\rm p}(i_a)(\langle h_{i_a},\Lambda\rangle+1)
       &&(1\le a\le n),\\
 \mathsf r(\tau_a e(\mathbf i))&={\rm p}(i_a){\rm p}(i_{a+1})
 (\langle h_{i_a},\Lambda\rangle+1)(\langle h_{i_{a+1}},\Lambda\rangle+1)
       &&(1\le a<n),
\end{aligned}
\end{equation}
where $n=\operatorname{ht}(\alpha)$.
\par\endgroup\pagebreak[0]
\end{defn}

\begin{lem}\label{lem:auxiliary-parity}
The assignments \eqref{eq:auxiliary-parity} define a
$\mathbb Z_2$-grading of $R^\Lambda_\alpha$.
\end{lem}
\begin{proof}
By \eqref{eq:local-shifts}, $s_{i,0}={\rm p}(i)(\langle h_i,\Lambda\rangle+1)$.
For $\mathbf i\in I^\alpha$ and $1\le a<n$, the two unequal-colour
mixed relations have degrees
\begin{align*}
 \mathsf r(\tau_a x_{a+1}e(\mathbf i))
 &=s_{i_{a+1},0}+s_{i_a,0}s_{i_{a+1},0}
  =\mathsf r(x_a\tau_a e(\mathbf i)),\\
 \mathsf r(\tau_a x_a e(\mathbf i))
 &=s_{i_a,0}+s_{i_a,0}s_{i_{a+1},0}
  =\mathsf r(x_{a+1}\tau_a e(\mathbf i)).
 \end{align*}
For $i_a=i_{a+1}=i$,
\[
\mathsf r(\tau_ax_{a+1}e(\mathbf i))=\mathsf r(x_a\tau_ae(\mathbf i))
=s_{i,0}^2+s_{i,0}=\bar0=\mathsf r(e(\mathbf i)),
\]
and likewise with $x_a,x_{a+1}$ interchanged.
For distinct $i,j\in I$ and $A,B\in\mathbb Z_{\ge0}$,
\eqref{condition1}--\eqref{condition2} give
\[
 t_{i,j;(A,B)}\ne0\ \Longrightarrow\
 A s_{i,0}+B s_{j,0}=\bar0
       =2s_{i,0}s_{j,0}\quad\text{in }\mathbb Z_2.
\]
For the braid relation on three successive colours $i,j,k$, the common
$\mathsf r$-degree of the triple crossings is
\[
 s_{i,0}s_{j,0}+s_{i,0}s_{k,0}+s_{j,0}s_{k,0}.
\]
If $k=i$, this is $s_{i,0}$. For $i_a=i_{a+2}=i$,
$i_{a+1}=j$ and $1\le a<n-1$, the two correction terms satisfy
\begin{align*}
 \mathsf r\!\left(\frac{Q_{i,j}(x_{a+2},x_{a+1})-Q_{i,j}(x_a,x_{a+1})}
                   {x_{a+2}-x_a}e(\mathbf i)\right)
 &=\bar0=s_{i,0}&& (i\in I_{\rm even}),\\
 \mathsf r\!\left((x_{a+2}-x_a)
       \frac{Q_{i,j}(x_{a+2},x_{a+1})-Q_{i,j}(x_a,x_{a+1})}
                   {x_{a+2}^2-x_a^2}e(\mathbf i)\right)
 &=s_{i,0}+\bar0=s_{i,0}&& (i\in I_{\rm odd}),
 \end{align*}
where the quotient in the second line has $\mathsf r$-degree zero by
\eqref{condition2}.
The idempotent and distant supercommutation relations are homogeneous
by additivity. Finally,
\begin{equation*}
 \mathsf r(x_1^{\langle h_i,\Lambda\rangle}e(\mathbf i))
 ={\rm p}(i)\langle h_i,\Lambda\rangle(\langle h_i,\Lambda\rangle+1)=\bar0
 \qquad(\mathbf i\in I^\alpha,\ i_1=i).
\end{equation*}
\end{proof}

\begin{lem}\label{lem:trace-rotation}
Let $\mathbf i\in I^\alpha$, and let $g$ be one of
$e(\mathbf i)$, $x_a e(\mathbf i)$ ($1\le a\le n$), or
$\tau_a e(\mathbf i)$ ($1\le a<n$).
For every parity-homogeneous $b\in R^\Lambda_\alpha$,
\begin{equation}\label{eq:generator-super-law}
 t_{\Lambda,\alpha}(gb)
 =(-1)^{{\rm p}(g){\rm p}(b)+{\rm p}(g)+\mathsf r(g)}t_{\Lambda,\alpha}(bg).
\end{equation}
\end{lem}
\begin{proof}
Let $\mathbf j=\mathbf i$ for $g=e(\mathbf i)$ or $g=x_ae(\mathbf i)$, and
$\mathbf j=s_a\mathbf i$ for $g=\tau_a e(\mathbf i)$.
First suppose $b\in e(\mathbf i)R^\Lambda_\alpha e(\mathbf j)$.
\[\begin{aligned}
 t_{\Lambda,\alpha}(gb)
 &\xlongequal{\eqref{eq:mixed-counit-computation}}
 t_{\Lambda,\alpha}\bigl((\rho_g^{RL})_{\mathbb K}(b)\bigr)\\
 &\xlongequal{\text{Lemma~\ref{lem:generator-double-mates}}}
 (-1)^{{\rm p}(g)+\mathsf r(g)}t_{\Lambda,\alpha}\bigl((\rho_g)_{\mathbb K}(b)\bigr)\\
 &\xlongequal{\eqref{eq:signed-right-multiplication}}
 (-1)^{{\rm p}(g)+\mathsf r(g)+{\rm p}(g){\rm p}(b)}t_{\Lambda,\alpha}(bg).
 \end{aligned}\]
For arbitrary $b$, the definition \eqref{eq:normalized-trace} gives
\[
 t_{\Lambda,\alpha}(gb)=t_{\Lambda,\alpha}(g e(\mathbf i)b e(\mathbf j)),
 \qquad
 t_{\Lambda,\alpha}(bg)=t_{\Lambda,\alpha}(e(\mathbf i)b e(\mathbf j)g).
\]
Applying the displayed equalities to $e(\mathbf i)b e(\mathbf j)$ proves
the result.
\end{proof}

\flushbottom
\begin{thm}\label{thm:nakayama-formula}
The Nakayama automorphism of $t_{\Lambda,\alpha}$ is characterized
as follows. For every $a\in R^\Lambda_\alpha$ homogeneous for both
${\rm p}$ and $\mathsf r$, and every $b\in R^\Lambda_\alpha$, one has
\begin{equation}\label{eq:uniform-trace-law}
 t_{\Lambda,\alpha}(ab)=(-1)^{\pi_{\Lambda,\alpha}{\rm p}(a)+\mathsf r(a)}t_{\Lambda,\alpha}(ba).
\end{equation}
Moreover, $\mathtt{N}^2=\mathrm{id}$.
\end{thm}
\begin{proof}
Let $g$ be a generator in Lemma~\ref{lem:trace-rotation} and let
$b\in R^\Lambda_\alpha$ be parity-homogeneous. By Lemma~\ref{Frobeniius},
\[
{\rm p}(g)+{\rm p}(b)\ne \pi_{\Lambda,\alpha}
\ \Longrightarrow\ t_{\Lambda,\alpha}(gb)=t_{\Lambda,\alpha}(bg)=0.
\]
For ${\rm p}(g)+{\rm p}(b)=\pi_{\Lambda,\alpha}$, in $\mathbb Z_2$,
\begin{equation*}
 {\rm p}(g){\rm p}(b)+{\rm p}(g)
 ={\rm p}(g)(\pi_{\Lambda,\alpha}+{\rm p}(g))+{\rm p}(g)
 =\pi_{\Lambda,\alpha}{\rm p}(g).
\end{equation*}
Thus \eqref{eq:generator-super-law} gives \eqref{eq:uniform-trace-law}
for $g$, with arbitrary $b$ by linearity. Induction on a word in the
standard generators and $\mathsf r$-additivity in Lemma~\ref{lem:auxiliary-parity}
extend it to every $({\rm p},\mathsf r)$-homogeneous $a$.
The nondegeneracy in Lemma~\ref{Frobeniius}
identify this sign automorphism with $\mathtt{N}$; hence $\mathtt{N}^2=\mathrm{id}$.
\end{proof}

\begin{defn}\label{def:Theta}
For $a\in R^\Lambda_\alpha$ homogeneous for both ${\rm p}$ and $\mathsf r$, set
\begin{equation}\label{eq:Theta-definition}
 \mathtt{\Theta}(a)=(-1)^{(\pi_{\Lambda,\alpha}+\bar1){\rm p}(a)}\mathtt{N}(a)
         =(-1)^{{\rm p}(a)+\mathsf r(a)}a,
\end{equation}
and extend $\mathtt{\Theta}$ linearly to $R^\Lambda_\alpha$.
\end{defn}

\begin{cor}\label{cor:Theta-trace-law}
For parity-homogeneous $a,b\in R^\Lambda_\alpha$,
\begin{equation*}
 t_{\Lambda,\alpha}(ab)=(-1)^{{\rm p}(a){\rm p}(b)}t_{\Lambda,\alpha}(b\mathtt{\Theta}(a)).
\end{equation*}
\end{cor}

\begin{defn}\label{def:trace-symmetry}
A homogeneous linear form $t:\mathcal A\longrightarrow\mathbb K$
on a superalgebra $\mathcal A$ is \emph{symmetric} if
\begin{equation*}
 t(ab)=t(ba)\qquad(a,b\in\mathcal A).
\end{equation*}
It is \emph{supersymmetric} if
\begin{equation*}
 t(ab)=(-1)^{{\rm p}(a){\rm p}(b)}t(ba)
 \qquad(a,b\in\mathcal A\text{ parity-homogeneous}).
\end{equation*}
\end{defn}

\begin{rem}\label{rem:trace-symmetry-criteria}
Since $t_{\Lambda,\alpha}$ is nondegenerate, it is symmetric exactly when
$\mathtt{N}=\mathrm{id}$, and supersymmetric exactly when $\mathtt{\Theta}=\mathrm{id}$.
\end{rem}

\begin{cor}\label{cor:trace-symmetry}
Recall that $\alpha=\sum_{i\in I}n_i\nu_i\in Q^+$.
\begin{enumerate}[label=(\roman*)]
\item If $\langle h_{i},\Lambda\rangle$ is odd for every $i\in I_{\rm odd}$ with $n_i>0$,
then $t_{\Lambda,\alpha}(ab)=t_{\Lambda,\alpha}(ba)$ for all $a,b\in R^\Lambda_\alpha$.
\item If $\langle h_{i},\Lambda\rangle$ is even for every $i\in I_{\rm odd}$ with $n_i>0$,
then $t_{\Lambda,\alpha}(ab)=(-1)^{{\rm p}(a){\rm p}(b)}t_{\Lambda,\alpha}(ba)$ for all parity-homogeneous
$a,b\in R^\Lambda_\alpha$.
\end{enumerate}
\end{cor}
\begin{proof}
By \eqref{eq:sum-parity-shifts} and \eqref{eq:auxiliary-parity},
\begingroup\postdisplaypenalty=10000\interlinepenalty=10000
\[
 \begin{array}{c|c|c}
 &\text{degrees on generators}&\text{automorphism}\\\hline
 \text{(i)}&\pi_{\Lambda,\alpha}=\bar0,\quad \mathsf r=\bar0&\mathtt{N}=\mathrm{id}\\
 \text{(ii)}&\mathsf r={\rm p}&\mathtt{\Theta}=\mathrm{id}
 \end{array},
\]
where the last column follows from Theorem~\ref{thm:nakayama-formula} and
\eqref{eq:Theta-definition}, respectively.\qedhere
\par\endgroup
\end{proof}

\section{Semisimplicity and decomposition}\label{sec:semisimplicity-decomposition}

\subsection{Semisimplicity and zero defect}
We will use the Frobenius form to characterize semisimplicity in terms
of the defect.
\begin{defn}\label{def:defect}
For $\Lambda\in P^+$ and $\alpha\in Q^+$, the defect of $\alpha$ at
$\Lambda$ is
\begin{equation}\label{eq:defect-normalization}
 \operatorname{def}(\Lambda,\alpha)
   :=(\Lambda\mid\alpha)-\tfrac12(\alpha\mid\alpha)
   =\tfrac12 d_{\Lambda,\alpha}.
\end{equation}
For $\ell_i=\langle h_i,\Lambda\rangle$ and
$\alpha=\sum_{i\in I}c_i\nu_i\in Q^+$, the Cartan
pairings give
\begin{equation}\label{eq:defect-quadratic}
 \operatorname{def}(\Lambda,\alpha)
 =\sum_i\ell_i {\rm d}_i c_i
  -\sum_i{\rm d}_i c_i^2
  -\frac12\sum_{\substack{i,j\in I\\i\ne j}}{\rm d}_i a_{ij}c_ic_j
  \in\mathbb Z.
\end{equation}
Thus the defect is an integer; its vanishing is understood in
$\mathbb Z$, independently of the characteristic of $\mathbb K$.
\end{defn}

The ordinary graded-module categorification theorem identifies the
nonzero cyclotomic components.
\begin{lem}\label{lem:weight-support}\cite[Theorem~10.2, (10.6)]{KKO1}
Let $\Lambda\in P^+$ and $\alpha\in Q^+$.  The following conditions are
equivalent:
\begingroup\postdisplaypenalty=10000\interlinepenalty=10000
\begin{enumerate}[label=\textup{(\roman*)}]
\item $R^\Lambda_\alpha\ne0$;
\item $L(\Lambda)_{\Lambda-\alpha}\ne0$,
\end{enumerate}
where $L(\Lambda)$ is the usual integrable highest-weight module.
\par\endgroup\pagebreak[0]
\end{lem}

\begin{lem}\label{lem:kac-weight-norm}\cite[Proposition~11.4(a)]{Kac}
For weights $\mu,\eta$ of $L(\Lambda)$,
\begin{equation*}
 (\Lambda\mid\Lambda)-(\mu\mid\eta)\ge0,
 \qquad\text{with equality precisely when }\mu=\eta\in W\Lambda.
\end{equation*}
\end{lem}

The derived equivalence in the following proposition comes from
\cite[Theorems~16.2 and~16.4]{BK:spin}.
\begin{prop}\label{prop:reflection-semisimplicity}
Let $\Lambda\in P^+$, $\alpha\in Q^+$ and $i\in I$, and suppose that
$R^\Lambda_\alpha\ne0$. Put
\begin{equation*}
 \alpha'=\alpha+\langle h_i,\Lambda-\alpha\rangle\nu_i,
 \qquad \Lambda-\alpha'=s_i(\Lambda-\alpha).
\end{equation*}
Then $R^\Lambda_{\alpha'}\ne0$, there is an
derived equivalence
\begin{equation}\label{eq:rickard-equivalence}
 D^b(R^\Lambda_\alpha\text{-}\mathrm{gsmod})
 \simeq D^b(R^\Lambda_{\alpha'}\text{-}\mathrm{gsmod}),
\end{equation}
and the following conditions are equivalent:
\begin{enumerate}[label=\textup{(\roman*)}]
\item $R^\Lambda_\alpha$ is semisimple;
\item $R^\Lambda_{\alpha'}$ is semisimple.
\end{enumerate}
\end{prop}
\begin{proof}
The weight set is $W$-invariant, so
Lemma~\ref{lem:weight-support} gives $R^\Lambda_{\alpha'}\ne0$.
Both algebras are Frobenius by
Lemma~\ref{Frobeniius}, hence self-injective, so \cite[Theorem~4.4]{AP}
gives an equivalence of their stable graded-supermodule categories.  Such a
stable category is zero precisely when the corresponding ordinary algebra is
semisimple, by Lemma~\ref{lem:grading-semisimplicity}. This proves the
equivalence of \textup{(i)} and \textup{(ii)}.
\end{proof}

The next lemma shows that semisimplicity forces the degree of a
homogeneous Frobenius form to be zero.
\begin{lem}\label{lem:frobenius-degree}
Let $A\ne0$ be a finite-dimensional $\mathbb Z$-graded $\mathbb K$-algebra.
If $A$ is semisimple and has a nondegenerate homogeneous Frobenius form of
degree $-D$, then $D=0$.
\end{lem}
\begin{proof}
Let $t:A\to\mathbb K$ be the given Frobenius form. By the graded
Wedderburn--Artin theorem \cite[Theorem 2.10.10]{NO04}, each graded simple summand of $A$ is an (greded)
endomorphism algebra over a finite-dimensional graded division algebra.
The $\mathbb Z$-support of such a graded division algebra is a finite subgroup of $\mathbb Z$, hence is trivial;
its degree-zero part is $\mathbb K$
since $\mathbb K$ is algebraically closed. Consequently, $A$ has
homogeneous matrix units, with diagonal units in degree zero. The sum
of the matrix traces is therefore a degree-zero Frobenius form
$\lambda:A\to\mathbb K$, and the pairing $A_{D}\times A_{-D}\to\mathbb K$ induced by $\lambda$ is
nondegenerate. Since $t$ and $\lambda$ are homogeneous Frobenius forms
of degrees $-D$ and $0$, there is a unique $u\in A_{-D}\cap A^\times$ with
$t(a)=\lambda(ua)$ for all $a\in A$.
Thus $0\ne u^N\in A_{-ND}$ for every $N\ge1$; finite-dimensionality
forces $D=0$.
\end{proof}

The preceding lemma turns the nonzero Frobenius degree into an
obstruction to semisimplicity.
\begin{cor}\label{cor:positive-defect-obstruction}
Let $\Lambda\in P^+$ and $\alpha\in Q^+$ with $R^\Lambda_\alpha\ne0$.
If $\operatorname{def}(\Lambda,\alpha)>0$, then $R^\Lambda_\alpha$
is not semisimple.
\end{cor}
\begin{proof}
Apply Lemmas~\ref{Frobeniius} and~\ref{lem:frobenius-degree}.
\end{proof}

The following lemma restricts a component to a finite full subdiagram.
\begin{lem}\label{lem:finite-support-restriction}
Let $I$ be a possibly infinite Cartan diagram, let $\Lambda\in P^+$,
and let $\alpha\in Q^+$. If $J\subset I$ is a finite full
subdiagram containing $\operatorname{supp}\alpha$, then
\begin{equation}\label{eq:finite-support-restriction}
 R^\Lambda_\alpha(I)=R^{\Lambda|_J}_\alpha(J),\qquad
 \operatorname{def}_I(\Lambda,\alpha)
 =\operatorname{def}_J(\Lambda|_J,\alpha).
\end{equation}
\Needspace{5\baselineskip}
Moreover, the following conditions are equivalent:
\begin{enumerate}[label=\textup{(\roman*)}]
\item $\Lambda-\alpha\in W_I\Lambda$;
\item there is a finite full $J\subset I$ containing $\operatorname{supp}\alpha$ such that $\Lambda|_J-\alpha\in W_J(\Lambda|_J)$.
\end{enumerate}
\end{lem}
\begin{proof}
The word idempotents, their relations and the cyclotomic exponents use
only the colours in $J$, and the pairings in
\eqref{eq:defect-quadratic} use only the corresponding Cartan entries.
This proves \eqref{eq:finite-support-restriction}.  Only finitely many
simple reflections occur in an expression for an element of $W_I$;
enlarging $J$ to contain their vertices proves \textup{(i)}$\Longrightarrow$\textup{(ii)}.  The converse follows by viewing
$W_J$ as the corresponding parabolic subgroup of $W_I$.
\end{proof}

\Needspace{8\baselineskip}
\begin{thm}\label{thm:structural-criterion}
Let $\Lambda\in P^+$ and $\alpha\in Q^+$ with $R^\Lambda_\alpha\ne0$. The following conditions are equivalent:
\begin{enumerate}[label=\textup{(\roman*)}]
\item $R^\Lambda_\alpha$ is semisimple;
\item $\operatorname{def}(\Lambda,\alpha)=0$;
\item $\Lambda-\alpha\in W_I\Lambda$.
\end{enumerate}
\end{thm}
\begin{proof}
Set $\mu=\Lambda-\alpha$. By Lemma~\ref{lem:weight-support}, $\mu$ is a weight
of $L(\Lambda)$, and expansion of \eqref{eq:defect-normalization} gives
\begin{equation*}
 2\operatorname{def}(\Lambda,\alpha)
   =(\Lambda\mid\Lambda)-(\mu\mid\mu).
\end{equation*}
By Lemma~\ref{lem:kac-weight-norm},
\[
\operatorname{def}(\Lambda,\alpha)\ge0,\quad \text{ and }
\operatorname{def}(\Lambda,\alpha)=0\ \Longleftrightarrow\ \mu\in W_I\Lambda.
\]

If $\operatorname{def}(\Lambda,\alpha)=0$, choose $i_1,\ldots,i_l\in I$
with $s_{i_l}\cdots s_{i_1}\mu=\Lambda$. Set $\alpha_0=\alpha$ and, for
$1\le a\le l$, put $\alpha_a=\Lambda-s_{i_a}\cdots s_{i_1}\mu$.
Each $\Lambda-\alpha_a$ is a weight of $L(\Lambda)$, so
Lemma~\ref{lem:weight-support} gives $R^\Lambda_{\alpha_a}\ne0$.
Iterating Proposition~\ref{prop:reflection-semisimplicity} gives
\[
 R^\Lambda_\alpha\text{ semisimple}
 \ \Longleftrightarrow\ R^\Lambda_{\alpha_1}\text{ semisimple}
 \ \Longleftrightarrow\ \cdots
 \ \Longleftrightarrow\ R^\Lambda_{\alpha_l}\text{ semisimple},
\]
where $R^\Lambda_{\alpha_l}=R^\Lambda_0=\mathbb K$.
Conversely, Corollary~\ref{cor:positive-defect-obstruction} shows that
semisimplicity of $R^\Lambda_\alpha$ implies
$\operatorname{def}(\Lambda,\alpha)=0$.

The assertion for the infinite diagrams follows from
Lemma~\ref{lem:finite-support-restriction} and the finite argument.
\end{proof}

\subsection{Decomposition over disconnected diagrams}
The ordinary tensor-product criterion over an algebraically closed field,
together with \cite[\S2-b]{BK} in the super case, gives the following.
\begin{lem}\label{lem:supertensor-semisimplicity}
For nonzero finite-dimensional superalgebras
$\mathcal A^{(1)},\ldots,\mathcal A^{(s)}$ over
$\mathbb K$, the following conditions are equivalent:
\begin{enumerate}[label=\textup{(\roman*)}]
\item $\bigotimes_{j=1}^s\mathcal A^{(j)}$ is semisimple;
\item $\mathcal A^{(j)}$ is semisimple for every $1\le j\le s$.
\end{enumerate}
\end{lem}

\begin{notation}\label{not:connected-components}
Let $I=\bigsqcup_{c=1}^d I_c$ be the decomposition into its finitely
many connected components, and write $\Lambda|_{I_c}$ for the
restriction of $\Lambda$ to $I_c$. For
$\mathbf n=(n_1,\ldots,n_d)$ with $n_c\ge0$ and $\sum_c n_c=n$, set
\begin{equation*}
 M_{\mathbf n}:=\frac{n!}{n_1!\cdots n_d!}.
\end{equation*}
We use $R_0(I_c)=R^{\Lambda|_{I_c}}_0(I_c)=\mathbb K$.
\end{notation}
The following extends the decomposition in
\cite[\S6, Theorems~6.26 and~6.30]{Salim}.

\begin{prop}\label{prop:rostam-decomposition}\label{cyciso}
There are isomorphisms of ordinary algebras
\begin{equation}\label{eq:noncyclotomic-decomposition}
 R_n(I)\cong
 \bigoplus_{n_1+\cdots+n_d=n}
 \operatorname{Mat}_{M_{\mathbf n}}
 \left(\bigotimes_{c=1}^d R_{n_c}(I_c)\right),
\end{equation}
\begin{equation}\label{eq:disconnected-cyclotomic-decomposition}
 R^\Lambda_n(I)\cong
 \bigoplus_{n_1+\cdots+n_d=n}
 \operatorname{Mat}_{M_{\mathbf n}}
 \left(\bigotimes_{c=1}^d
       R^{\Lambda|_{I_c}}_{n_c}(I_c)\right).
\end{equation}
\end{prop}
\begin{proof}
The following proof essentially follows \cite[\S6]{Salim}.\\
\textbf{Step 1. The non-cyclotomic isomorphism.}
For a word $\mathfrak t=(t_1,\ldots,t_n)$ in $\{1,\ldots,d\}$, set
\begin{equation*}
 e(\mathfrak t):=\sum_{i_a\in I_{t_a}\ (1\le a\le n)}
                  e(i_1,\ldots,i_n).
\end{equation*}
Different component words give disjoint sets of colour words, so the
orthogonality of the $e(\mathbf i)$ gives
\begin{equation*}
 e(\mathfrak s)e(\mathfrak t)
   =\delta_{\mathfrak s,\mathfrak t}e(\mathfrak t).
\end{equation*}
Let
\begin{equation*}
 \operatorname{Sh}(\mathbf n):=
 \{\mathfrak t\in\{1,\ldots,d\}^n\mid
   |\{a\mid t_a=c\}|=n_c\ (1\le c\le d)\},
 \qquad |\operatorname{Sh}(\mathbf n)|=M_{\mathbf n}.
\end{equation*}
and let $\mathfrak t_0=(1^{n_1},\ldots,d^{n_d})$, $e_0=e(\mathfrak t_0)$,
and $E_{\mathbf n}=\sum_{\mathfrak t\in\operatorname{Sh}(\mathbf n)}
e(\mathfrak t)$.  Then
\begin{equation*}
 E_{\mathbf n}E_{\mathbf m}=\delta_{\mathbf n,\mathbf m}E_{\mathbf n},
 \qquad
 R_n(I)=\bigoplus_{n_1+\cdots+n_d=n}E_{\mathbf n}R_n(I).
\end{equation*}
Indeed, $\tau_aE_{\mathbf n}=E_{\mathbf n}\tau_a$, and $E_{\mathbf n}$
commutes with $x_k$ and every $e(\mathbf i)$.  If $I$ is infinite,
these sums and identities are read in the direct-sum sense of
\eqref{eq:infinite-index-direct-sum}.

For $i,j$ in distinct connected components,
$Q_{ij}=c_{ij}\in\mathbb K^\times$. Choose
$\rho_{ij}\in\mathbb K^\times$ with
\begin{equation}\label{eq:crossing-normalization-scalars}
 \rho_{ij}\rho_{ji}c_{ij}=1.
\end{equation}
For adjacent colours $i_a=i$, $i_{a+1}=j$, use
$\rho_{ij}\tau_a e(\mathbf i)$ for the forward crossing and
$\rho_{ji}\tau_a e(s_a\mathbf i)$ for the reverse crossing.
The quadratic relation and \eqref{eq:crossing-normalization-scalars} give
\begin{equation}\label{eq:normalized-crossing-inverse}
 (\rho_{ji}\tau_a e(s_a\mathbf i))
 (\rho_{ij}\tau_a e(\mathbf i))
 =\rho_{ji}\rho_{ij}c_{ij}e(\mathbf i)=e(\mathbf i).
\end{equation}

For $\mathfrak t\in\operatorname{Sh}(\mathbf n)$, fix adjacent
transpositions $s_{a_m},\ldots, s_{a_1}$ with
$\mathfrak t=s_{a_m}\cdots s_{a_1}\mathfrak t_0$, interchanging distinct
entries and preserving their order within every component.  Put
$\mathfrak t_r=s_{a_r}\cdots s_{a_1}\mathfrak t_0$ for $1\le r\le m$.
Write $\operatorname{comp}(i_1,\ldots,i_n)=(c_1,\ldots,c_n)$ when
$i_a\in I_{c_a}$. For $1\le r\le m$, set
\begin{equation}\label{eq:sorting-crossings}
 \begin{aligned}
 \gamma_r&:=\sum_{\substack{\mathbf i\in I^n\\
                  \operatorname{comp}(\mathbf i)=\mathfrak t_{r-1}}}
 \rho_{i_{a_r},i_{a_r+1}}\tau_{a_r}e(\mathbf i)
 \in e(\mathfrak t_r)R_n(I)e(\mathfrak t_{r-1}),\\
 \bar\gamma_r&:=\sum_{\substack{\mathbf j\in I^n\\
                  \operatorname{comp}(\mathbf j)=\mathfrak t_r}}
 \rho_{j_{a_r},j_{a_r+1}}\tau_{a_r}e(\mathbf j)
 \in e(\mathfrak t_{r-1})R_n(I)e(\mathfrak t_r),\\
 \bar\gamma_r\gamma_r&=e(\mathfrak t_{r-1}),
 \qquad \gamma_r\bar\gamma_r=e(\mathfrak t_r).
 \end{aligned}
\end{equation}
The last line follows from \eqref{eq:normalized-crossing-inverse}.
Define
\begin{equation*}
 u_{\mathfrak t}:=\gamma_m\cdots\gamma_1
 \in e(\mathfrak t)R_n(I)e_0,
 \qquad
 v_{\mathfrak t}:=\bar\gamma_1\cdots\bar\gamma_m
 \in e_0R_n(I)e(\mathfrak t),
\end{equation*}
with products read in the direct-sum sense for infinite $I$.  The last
line of \eqref{eq:sorting-crossings} gives
\begin{equation}\label{eq:sorting-crossing-orthogonality}
 \begin{aligned}
 v_{\mathfrak s}u_{\mathfrak t}
 &=v_{\mathfrak s}e(\mathfrak s)e(\mathfrak t)u_{\mathfrak t}
  =\delta_{\mathfrak s,\mathfrak t}v_{\mathfrak t}u_{\mathfrak t}
  =\delta_{\mathfrak s,\mathfrak t}e_0,\\
 u_{\mathfrak t}v_{\mathfrak t}&=e(\mathfrak t).
 \end{aligned}
\end{equation}

The distant supercommutation relations and the PBW basis
(see Proposition \ref{PBWbasis}) 
give the algebra isomorphism
\begin{equation}\label{eq:sorted-corner-tensor-isomorphism}
 \begin{aligned}
 \bigotimes_{c=1}^d R_{n_c}(I_c)
   &\xrightarrow{\sim}e_0R_n(I)e_0,\\
 a_1\otimes\cdots\otimes a_d&\longmapsto a_1\cdots a_d.
 \end{aligned}
\end{equation}
Using \eqref{eq:sorted-corner-tensor-isomorphism} to identify the tensor product with $e_0R_n(I)e_0$,
define
\begin{equation}\label{eq:matrix-isomorphism-maps}
 \begin{aligned}
 \Psi:\ &R_n(I)\longrightarrow
 \bigoplus_{\mathbf n}\operatorname{Mat}_{M_{\mathbf n}}
       \left(\bigotimes_{c=1}^d R_{n_c}(I_c)\right),\\
 &a\longmapsto
 \bigl((v_{\mathfrak s}a\,u_{\mathfrak t})_{
          \mathfrak s,\mathfrak t}\bigr)_{\mathbf n},
 \end{aligned}
\end{equation}
\begin{equation*}
 \begin{aligned}
 \Phi:\ &\bigoplus_{\mathbf n}\operatorname{Mat}_{M_{\mathbf n}}
       \left(\bigotimes_{c=1}^d R_{n_c}(I_c)\right)
       \longrightarrow R_n(I),\\
 &\bigl((b^{\mathbf n}_{\mathfrak s\mathfrak t})_{\mathfrak s,\mathfrak t}\bigr)_{\mathbf n}
       \longmapsto\sum_{\mathbf n}\sum_{\mathfrak s,\mathfrak t}
       u_{\mathfrak s}b^{\mathbf n}_{\mathfrak s\mathfrak t}
       v_{\mathfrak t}.
 \end{aligned}
\end{equation*}
For $a,a'\in R_n(I)$, equation~\eqref{eq:sorting-crossing-orthogonality} gives
\[
 \begin{aligned}
 (\Psi(a)\Psi(a'))_{\mathbf n;\mathfrak s\mathfrak t}
 &=\sum_{\mathfrak u}v_{\mathfrak s}a\,
        u_{\mathfrak u}v_{\mathfrak u}a'u_{\mathfrak t}
   =v_{\mathfrak s}aa'u_{\mathfrak t},\\
 \Phi\Psi(a)&=\sum_{\mathbf n}\sum_{\mathfrak s,\mathfrak t}
      e(\mathfrak s)ae(\mathfrak t)=a,\\
 (\Psi\Phi(b))_{\mathbf n;\mathfrak s\mathfrak t}
 &=\sum_{\mathfrak u,\mathfrak v}
   v_{\mathfrak s}u_{\mathfrak u}
   b^{\mathbf n}_{\mathfrak u\mathfrak v}
   v_{\mathfrak v}u_{\mathfrak t}
   =b^{\mathbf n}_{\mathfrak s\mathfrak t},
 \end{aligned}
\]
where $b=\bigl((b^{\mathbf n}_{\mathfrak s\mathfrak t})_{
\mathfrak s,\mathfrak t}\bigr)_{\mathbf n}$,
and the orthogonal summands for all $\mathbf n$ exhaust $R_n(I)$.
This proves \eqref{eq:noncyclotomic-decomposition}.

\textbf{Step 2. The cyclotomic quotient.}
Let $K\triangleleft R_n(I)$
be the cyclotomic ideal, generated by
\begin{equation*}
 g(\mathbf i):=x_1^{\langle h_{i_1},\Lambda\rangle}e(\mathbf i)
 \qquad(\mathbf i\in I^n).
\end{equation*}
For each $\mathbf n$ and each $c$ with $n_c>0$, let
$K_{\mathbf n;c}\triangleleft R_{n_c}(I_c)$ be the cyclotomic
ideal for $\Lambda|_{I_c}$, and set
\begin{equation}\label{eq:component-cyclotomic-ideal}
 K_{\mathbf n}:=\sum_{c=1}^d
 R_{n_1}(I_1)\otimes\cdots\otimes K_{\mathbf n;c}
 \otimes\cdots\otimes R_{n_d}(I_d),
\end{equation}
where the summand is zero when $n_c=0$.
We claim that
\begin{equation}\label{eq:cyclotomic-ideal-image}
 \Psi(K)=\bigoplus_{\mathbf n}
                \operatorname{Mat}_{M_{\mathbf n}}(K_{\mathbf n}),
\end{equation}

For a generator $g(\mathbf i)$, let $\mathfrak t$ be its component
word, with $t_1=c$.
Let $\mathbf i_0$ be obtained by grouping the entries of
$\mathbf i$ in the order $I_1,\ldots,I_d$, preserving their order
within each group, and put $a_c=1+n_1+\cdots+n_{c-1}$.

The unequal-colour dot--crossing relations and
\eqref{eq:normalized-crossing-inverse} give
\begin{equation}\label{eq:sorted-cyclotomic-generator}
 \bigl(\Psi(g(\mathbf i))\bigr)_{\mathbf m;\mathfrak s\mathfrak u}
 =
 \begin{cases}
 r_{\mathbf i}\,
 x_{a_c}^{\langle h_{i_1},\Lambda|_{I_c}\rangle}e(\mathbf i_0),
 &\mathbf m=\mathbf n,\ \mathfrak s=\mathfrak u=\mathfrak t,\\
 0,&\text{otherwise},
 \end{cases}
 \quad \text{ for some } r_{\mathbf i}\in\mathbb K^\times,
\end{equation}
where $\langle h_{i_1},\Lambda|_{I_c}\rangle
=\langle h_{i_1},\Lambda\rangle$ because
$i_1=(\mathbf i_0)_{a_c}\in I_c$.  Hence
\[
 \Psi(K)\subseteq\bigoplus_{\mathbf n}
                       \operatorname{Mat}_{M_{\mathbf n}}(K_{\mathbf n}).
\]
Conversely, each generator
$h=x_{a_c}^{\langle h_{i_1},\Lambda|_{I_c}\rangle}e(\mathbf i_0)$
of $K_{\mathbf n}$ occurs in \eqref{eq:sorted-cyclotomic-generator},
up to the unit $r_{\mathbf i}$.  Thus $\mathrm E_{\mathfrak t\mathfrak t}(h)\in\Psi(K)$, where
$\mathrm E_{\mathfrak s\mathfrak u}(a)$ is the matrix with entry $a$ in
position $(\mathfrak s,\mathfrak u)$ and zero elsewhere. For
$\mathfrak s,\mathfrak u\in\operatorname{Sh}(\mathbf n)$, we have
\begin{equation*}
 \mathrm E_{\mathfrak s\mathfrak t}(e(\mathbf i_0))\,
 \mathrm E_{\mathfrak t\mathfrak t}(h)\,
 \mathrm E_{\mathfrak t\mathfrak u}(e(\mathbf i_0))
 =\mathrm E_{\mathfrak s\mathfrak u}(h)\in\Psi(K).
\end{equation*}
  Therefore
\[
 \bigoplus_{\mathbf n}\operatorname{Mat}_{M_{\mathbf n}}(K_{\mathbf n})
 \subseteq\Psi(K).
\]
This proves \eqref{eq:cyclotomic-ideal-image}. Taking quotients in \eqref{eq:matrix-isomorphism-maps} and
using \eqref{eq:cyclotomic-ideal-image} and then \eqref{eq:component-cyclotomic-ideal} gives
\[
 \begin{aligned}
 R_n^\Lambda(I)=R_n(I)/K
 &\cong\bigoplus_{\mathbf n}\operatorname{Mat}_{M_{\mathbf n}}
 \left(\left(\bigotimes_{c=1}^dR_{n_c}(I_c)\right)/K_{\mathbf n}\right)\\
 &\cong\bigoplus_{\mathbf n}\operatorname{Mat}_{M_{\mathbf n}}
 \left(\bigotimes_{c=1}^d
                  R_{n_c}^{\Lambda|_{I_c}}(I_c)\right),
 \end{aligned}
\]
which proves \eqref{eq:disconnected-cyclotomic-decomposition}.
\end{proof}

\section{Cyclotomic Hecke--Clifford superalgebras and the KKT isomorphism}\label{basic-non-dege}
\subsection{The affine Hecke--Clifford presentation}
Let ${\rm R}$ be an integral domain. Fix $q\in{\rm R}^\times$ and put
$\epsilon:=q-q^{-1}$.
We use the presentation in \cite[\S2-d, (2.5)--(2.22)]{BK}
of the algebra introduced in \cite[\S3]{JN}.
\begin{defn}\label{def:affine-hecke-clifford}
	The non-degenerate affine Hecke-Clifford superalgebra $\mathcal{H}_{\rm R}(n)$ \label{pag:AHCA} is
	the superalgebra over ${\rm R}$ generated by even generators
	$T_1,\ldots,T_{n-1},$ $X_1^{\pm 1},\ldots,X_n^{\pm 1}$ and odd generators
	$C_1,\ldots,C_n$ subject to the following relations, with $1\le i,j<n$ for the $T$
	indices and $1\le i,j\le n$ for the $X,C$ indices:
	\begin{align}
		T_i^2&=\epsilon T_i+1,\qquad T_iT_j=T_jT_i\quad(|i-j|>1),\notag
\\
		T_iT_{i+1}T_i&=T_{i+1}T_iT_{i+1}\quad(1\le i<n-1),\notag
\\
		X_iX_j&=X_jX_i,\quad X_iX^{-1}_i=X^{-1}_iX_i=1
		\quad(1\leq i,j\leq n),\notag
\\
		C_i^2&=1,\qquad C_iC_j=-C_jC_i \quad(i\ne j),\notag 
\\
		T_iX_i&=X_{i+1}T_i-\epsilon(X_{i+1}+C_iC_{i+1}X_i),\notag
\\
		T_iX_{i+1}&=X_iT_i+\epsilon(1-C_iC_{i+1})X_{i+1},\notag
\\
		T_iX_j&=X_jT_i, \quad (j\neq i, i+1),\notag
\\
		T_iC_i&=C_{i+1}T_i,\quad T_iC_{i+1}=C_iT_i-\epsilon(C_i-C_{i+1}),\notag
\\
		T_iC_j&=C_jT_i \quad(j\ne i,i+1),\notag
\\
		X_iC_i&=C_iX_i^{-1},\qquad X_iC_j=C_jX_i \quad(i\ne j).\notag
\end{align}
\end{defn}

\begin{notation}\label{not:affine-PBW-monomials}
	For $\mathbf a=(a_1,\ldots,a_n)\in\mathbb{Z}^n$ and
	$\mathbf b=(b_1,\ldots,b_n)\in\{0,1\}^n$, we set
	$X^{\mathbf a}:=X_1^{a_1}\cdots X_n^{a_n},$
	$C^{\mathbf b}:=C_1^{b_1}\cdots C_n^{b_n}$.
	For $w=s_{j_1}\cdots s_{j_l}\in\mathfrak S_n$ reduced, put
	$T_w:=T_{j_1}\cdots T_{j_l}$.
\end{notation}
	\begin{lem}\cite[Theorem 2.2]{BK}\label{lem:PBWnon-dege}
		The set
        $$\{X^{\mathbf a}C^{\mathbf b}T_w~|~ \mathbf a\in\mathbb{Z}^n,
		\mathbf b\in\{0,1\}^n, w\in \mathfrak{S}_n\}$$
        forms a basis of $\mathcal{H}_{\rm R}(n)$.
	\end{lem}


\subsection{Cyclotomic quotients and the separate condition}\label{section:CHCAs}
Assume $q+q^{-1}\in{\rm R}^\times$.
The following scalar map records the cyclotomic parameters.
\begin{defn}\label{def:nondegenerate-spectral-map}\cite[(3.9)]{SW}
Define
$\mathtt q:{\rm R}^\times\to{\rm R}$ by
\begingroup\postdisplaypenalty=10000\interlinepenalty=10000
\begin{equation*}
 \mathtt q(u):=2\frac{qu+(qu)^{-1}}{q+q^{-1}}.
\end{equation*}
\par\endgroup
\end{defn}
The scalar map determines the three cyclotomic polynomials used below.
\begin{defn}\label{def:cyclotomic-polynomials}\cite[\S3.3, p.~192]{SW}
For $m\ge0$ and $\undQ=(Q_1,\ldots,Q_m)\in({\rm R}^\times)^m$, define
\begin{equation}\label{eq:cyclotomic-parameters}
\begin{aligned}
 f^{(\mathtt{0})}_{\undQ}(X)&:=\prod_{j=1}^m(X+X^{-1}-\mathtt q(Q_j)),\\
 f^{(\mathtt{s})}_{\undQ}(X)&:=(X-1)f^{(\mathtt{0})}_{\undQ}(X),\\
 f^{(\mathtt{ss})}_{\undQ}(X)&:=(X-1)(X+1)f^{(\mathtt{0})}_{\undQ}(X).
\end{aligned}
\end{equation}
\end{defn}
\label{pag:CHCA}\label{pag:Q-parameters}
		Take $f=f^{(\bullet)}_{\undQ}$ from \eqref{eq:cyclotomic-parameters}.

The same component indices index the cyclotomic factors, the simple
modules below, and the later KKT weight.  We fix them here in the
notation of \cite[Definition~3.7]{SW}.
For $r\geq0$, write $[r]=\{1,\ldots,r\}$, with $[0]=\emptyset$.
\begin{notation}\label{not:source-multipartition-data}
Set
\begin{equation}\label{eq:SW-component-indices}
\mathcal C_\bullet:=
\begin{cases}
[m],&\bullet=\mathtt{0},\\
\{0\}\cup[m],&\bullet=\mathtt{s},\\
\{0_-,0_+\}\cup[m],&\bullet=\mathtt{ss},
\end{cases}
\qquad Q_0=Q_{0_+}:=1,\qquad Q_{0_-}:=-1.
\end{equation}
For $n\geq0$, let $\mathscr P_n^m$ be the set of $m$-multipartitions
of total size $n$, let $\mathscr P_n^s$ be the set of strict partitions
of $n$, and set
\begin{equation*}
 \mathscr P_n^{\mathtt{0},m}:=\mathscr P_n^m,\qquad
 \mathscr P_n^{\mathtt{s},m}:=\bigsqcup_{a+b=n}\mathscr P_a^s\times\mathscr P_b^m,
 \qquad
 \mathscr P_n^{\mathtt{ss},m}:=\bigsqcup_{a+b+c=n}
       \mathscr P_a^s\times\mathscr P_b^s\times\mathscr P_c^m.
\end{equation*}
Here $\mathscr P_0^0=\{\emptyset\}$ and $\mathscr P_n^0=\emptyset$
for $n>0$; in the $\mathtt{ss}$ case the two strict components have component
indices $0_-$ and $0_+$, respectively.  For
$\boldsymbol\lambda\in\mathscr P_n^{\bullet,m}$, a box is a triple
$A=(a,b,r)$ with $r\in\mathcal C_\bullet$.  If $1\leq r\leq m$, then
$1\leq b\leq\lambda_a^{(r)}$; if $r\in\mathcal C_\bullet\setminus[m]$,
then $a\leq b<a+\lambda_a^{(r)}$.  Its multiplicative residue is
\begin{equation*}
 \operatorname{res}(A):=Q_rq^{2(b-a)}.
\end{equation*}
\end{notation}

We use the cyclotomic quotient of \cite[\S3-a, (3.1)--(3.3)]{BK}
with the polynomial families and levels of \cite[\S3.3, p.~192]{SW}.
\begin{defn}\label{def:cyclotomic-hecke-clifford}
Let $\langle f(X_1)\rangle$ be the two-sided ideal generated by $f(X_1)$. Since
$X^mf(X)$ is monic, the cyclotomic quotient and its level are
\begin{equation*}
 \mathcal H^f_{\rm R}(n):=\mathcal H_{\rm R}(n)/\langle f(X_1)\rangle,\qquad
 r:=\deg(X^m f(X))=
 \begin{cases}2m,&\bullet=\mathtt{0},\\2m+1,&\bullet=\mathtt{s},\\2m+2,&\bullet=\mathtt{ss}.
 \end{cases}
\end{equation*}
\end{defn}
We use the same symbols for the images of $X^{\mathbf a},C^{\mathbf b},T_w$ in the quotient.
	\begin{lem}\label{lem:cyclotomic-PBW}\cite[Theorem 3.6]{BK}
		The set
        $$\{X^{\mathbf a}C^{\mathbf b}T_w~|~\mathbf a\in\{0,1,\cdots,r-1\}^n,
		\mathbf b\in\{0,1\}^n, w\in {\mathfrak{S}_n}\}$$
        forms an ${\rm R}$-basis of $\mathcal{H}^f_{\rm R}(n)$.
	\end{lem}
For the separate condition, assume $q\in\mathbb K^\times$ and
$q^4\ne1$.
\begin{defn}\label{def:nondegenerate-Poincare-polynomials}
Following \cite[\S3.4, p.~194, preceding Proposition~3.11]{SW},
the Poincar\'e polynomials are
\begin{equation}\label{eq:nondeg-separation}
\begin{aligned}
 P_n^{(\mathtt{0})}(q^2,\undQ)
 :={}&\prod_{t=1}^n(q^{2t}-1)
 \prod_{j=1}^m\left(
  \prod_{t=3-n}^{n-1}(Q_j^2-q^{-2t})
  \prod_{t=1-n}^{n}(Q_j^2-q^{-4t})\right)\\
 &\quad\cdot\prod_{1\le j<l\le m}\prod_{t=1-n}^{n-1}
       (Q_j-Q_lq^{-2t})(Q_jQ_l-q^{-2(t+1)}),
\end{aligned}
\end{equation}
\begin{equation}\label{eq:nondeg-extra-factor}
 P_n^{(\mathtt{s})}(q^2,\undQ)=P_n^{(\mathtt{ss})}(q^2,\undQ):=
 P_n^{(\mathtt{0})}(q^2,\undQ)\prod_{t=1}^n(q^{2t}+1).
\end{equation}
Every product over an empty index interval is one.
\end{defn}

We use the following formulation of the separate condition from
\cite[Definition~3.9 and Proposition~3.11]{SW}.
\begin{defn}\label{def:separate-condition}
For $f=f^{(\bullet)}_{\undQ}$, the parameters satisfy the separate
condition if
\begin{equation*}
 P_n^{(\bullet)}(q^2,\undQ)\ne0.
\end{equation*}
\end{defn}
\begin{rem}\label{rem:LS-parameter-conventions}
For comparison with \cite{LS}, write $\mathtt q'$ and $Q'_r$ for the
function and parameters denoted there by $q$ and $Q_r$, respectively.
The parameter conventions are
\begin{equation*}
 \begin{gathered}
 \mathtt q'(x)=\frac{2(x+x^{-1})}{q+q^{-1}},
 \qquad Q'_r=qQ_r\quad(r\in\mathcal C_\bullet),\\
 \mathtt q'(Q'_rq^{2k})
 =\mathtt q(Q_rq^{2k})\qquad(k\in\mathbb Z).
 \end{gathered}
\end{equation*}
In particular, the strict parameters in \cite{LS} are $q$ and $-q$.
\end{rem}

\subsection{Quiver Hecke-Clifford superalgebra}
To compare the cyclotomic Hecke--Clifford algebra with quiver Hecke
superalgebras, we recall the KKT presentation and its cyclotomic quotient.
\begin{notation}\label{not:KKT-vertex-cover}
Let $\bigl({\rm{A}}=(a_{ij})_{i,j\in I},P,\Pi,\Pi^\vee\bigr)$ be a Cartan superdatum.
\label{pag:[n]}

\label{pag:J}
Set $J=(I_{\rm odd}\times\{0\})\sqcup(I_{\rm even}\times\{\pm\})$.
The involution and projection are
\[
 \begin{aligned}
 c(i,0)&=(i,0)&& (i\in I_{\rm odd}),\\
 c(i,\pm)&=(i,\mp)&& (i\in I_{\rm even}),\\
 \pr(i,\varepsilon)&=i&& ((i,\varepsilon)\in J),
 \end{aligned}
 \qquad J^c:=\{j\in J:c(j)=j\}=I_{\rm odd}\times\{0\}.
\]
Thus $\pr^{-1}(i)=\{(i,+),(i,-)\}$ for $i\in I_{\rm even}$,
and $\pr^{-1}(i)=\{(i,0)\}$ for $i\in I_{\rm odd}$.
The idempotents in the presentation below are indexed by $J^n$.
For $w\in\mathfrak S_n$ and $\mathbf i\in J^n$,
$(w\mathbf i)_a=\mathbf i_{w^{-1}(a)}$ for $1\le a\le n$.
For $p\in[n],$ we define
$c_p: J^n\to J^n$ by
$$c_{p}{\bf i}:=(c^{\delta_{p\ell}}{\bf i}_{\ell})_{1\leq \ell\leq n}
\quad\text{for ${\bf i}=({\bf i}_1,\ldots,{\bf i}_n)\in J^n$.}$$
Thus $c_p\mathbf i$ denotes the action on a word; in the defining relations
below the bare symbol $c_p$ denotes the Clifford generator.
\end{notation}

Recall that for each $i,i'\in I$, we have chosen an element $Q_{i,i'}(u,v)\in {\rm Pol}_{i,i'}$ of the form
\begin{align*}
	Q_{i,i'}(u,v)=\sum_{r,s\geq 0}t_{i,i';(r,s)}u^r v^s.
\end{align*}
\label{pag:widetildeQ polys}
\begin{defn}\label{def:KKT-polynomials}
Following \cite[Remark~3.14]{KKT}, define the polynomials
$\widetilde{Q}_{j,j'}(u,v)\in {\rm R}[u,v]$ for $j,j'\in J$ as follows: for any $(i,\varepsilon),(i',\varepsilon') \in J,$ where $i,i'\in I$ and $\varepsilon,\varepsilon'\in\{0,\pm\}$, set
$(-1)^0=(-1)^+=1$, $(-1)^-=-1$, and
\begin{align}\label{extend Q-polys}
\widetilde{Q}_{(i,\varepsilon),(i',\varepsilon')}(u,v):=\sum_{\substack{r,s\geq 0\\t_{i,i';(r,s)}\ne0}}(-1)^{(\delta_{{\rm p}(i),\bar1}r+\delta_{{\rm p}(i'),\bar1}s)/2}t_{i,i';(r,s)}\left((-1)^{\varepsilon}u\right)^r \left((-1)^{\varepsilon'}v\right)^s .
\end{align}
\end{defn}

\begin{rem}
By \eqref{condition1}--\eqref{condition2}, for $i,i'\in I$ and $r,s\ge0$,
\[
 t_{i,i';(r,s)}\ne0\ \Longrightarrow\
 \begin{cases}r\in2\mathbb Z,&i\in I_{\rm odd},\\
 s\in2\mathbb Z,&i'\in I_{\rm odd}.\end{cases}
\]
Thus every exponent of $-1$ in \eqref{extend Q-polys} is an integer.
In particular,
$\widetilde Q_{j,j'}(-u,v)=\widetilde Q_{j,j'}(u,v)$ for $j\in J^c$, $j'\in J$.
\end{rem}
\begin{defn}\cite[Definition 3.5]{KKT} \label{pag:quiver Hecke-Clifford superalgebra}
Let $\bigl({\rm{A}}=(a_{ij})_{i,j\in I},P,\Pi,\Pi^\vee\bigr)$ be a Cartan superdatum,
$\widetilde{Q}=(\widetilde{Q}_{j,j'}(u,v))_{j,j'\in J}$ be chosen as above, and $n\in \N$.
For finite $I$, the quiver Hecke--Clifford superalgebra
$RC_n=RC_n(I)$ is the ${\rm R}$-superalgebra generated by the even generators
$\{y_p\}_{1\le p\le n}$, $\{\sigma_a\}_{1\le a<n}$,
$\{e({\bf i})\}_{{\bf i}\in J^n}$ and the odd generators $\{c_p\}_{1\leq
p\leq n}$ with the following defining relations: for ${\bf i},{\bf j}\in
J^n$, $1\leq p,\,q\leq n$, $1\leq a,b<n$, with $1\leq a<n-1$ in the last relation, we have
\begin{enumerate}
\item $e({\bf i})e({\bf j})=\delta_{{\bf i} {\bf j}}e({\bf i})$, $1=\sum_{{\bf i}\in J^n}e({\bf i})$,
$y_pe({\bf i})=e({\bf i})y_p$,  $c_pe({\bf i})=e(c_p{\bf i})c_p$,
\item $y_py_q=y_qy_p$, $c_pc_q+c_qc_p=2\delta_{pq}$,
\item $c_py_q=(-1)^{\delta_{p,q}}y_qc_p$,
\item $\sigma_ae({\bf i}) = e(s_a{\bf i})\sigma_a, \sigma_a c_p = c_{s_a(p)}\sigma_a$,
\item $\sigma_a y_pe({\bf i}) = y_{p}\sigma_ae({\bf i})$ if $p\not=a,a+1$,
\item

$$\sigma_ay_{a+1}-y_a\sigma_a=
\sum_{{\bf i}_a={\bf i}_{a+1}}e({\bf i})-\sum_{{\bf i}_a=c{\bf i}_{a+1}}c_ac_{a+1}e({\bf i}),$$
\item

$$y_{a+1}\sigma_a-\sigma_ay_a=
\sum_{{\bf i}_a={\bf i}_{a+1}}e({\bf i})+\sum_{{\bf i}_a=c{\bf i}_{a+1}}c_ac_{a+1}e({\bf i}),$$
\item $\sigma_a^2e({\bf i}) = \tQ_{{\bf i}_a,{\bf i}_{a+1}}(y_a,y_{a+1})e({\bf i})$,
\item $\sigma_a\sigma_{b}=\sigma_{b}\sigma_{a}$ if $|a-b|>1$,
\item

\begin{align*}
\sigma_{a+1}\sigma_a\sigma_{a+1}-\sigma_a\sigma_{a+1}\sigma_a
&=\sum_{{\bf i}_a={\bf i}_{a+2}}
\dfrac{\tQ_{{\bf i}_a,{\bf i}_{a+1}}(y_{a+2},y_{a+1})-\tQ_{{\bf i}_a,{\bf i}_{a+1}}(y_{a},y_{a+1})}
{y_{a+2}-y_{a}}e({\bf i}) \\
&\hs{-1ex}+\kern-2ex\sum_{{\bf i}_a=c{\bf i}_{a+2}}
\dfrac{\tQ_{{\bf i}_a,{\bf i}_{a+1}}(y_{a+2},y_{a+1})-\tQ_{{\bf i}_a,{\bf i}_{a+1}}(-y_{a},y_{a+1})}{y_{a+2}+y_{a}}
c_ac_{a+2}e({\bf i}).
\end{align*}
\end{enumerate}
\end{defn}

\begin{rem}\label{rem:QHC-grading}
$RC_n$ is also $\Z$-graded. For $\mathbf i\in J^n$,
$1\le p\le n$ and $1\le a<n$, its grading is given by
\[
\begin{aligned}
 \deg(e(\mathbf i))&=\deg(c_pe(\mathbf i))=0,\\
 \deg(y_pe(\mathbf i))&=(\nu_{\pr(\mathbf i_p)}\mid\nu_{\pr(\mathbf i_p)}),\\
 \deg(\sigma_ae(\mathbf i))&=-(\nu_{\pr(\mathbf i_a)}\mid\nu_{\pr(\mathbf i_{a+1})}).
\end{aligned}
\]
\end{rem}

For $w=s_{j_1}\cdots s_{j_l}\in\Sym_n$ reduced, put
$\sigma_w=\sigma_{j_1}\cdots\sigma_{j_l}$. For
$a\in\mathbb Z_{\ge0}^n$ and $\eta\in\{0,1\}^n$, put
$y^a=y_1^{a_1}\cdots y_n^{a_n}$ and
$c^\eta=c_1^{\eta_1}\cdots c_n^{\eta_n}$.
\begin{prop}\cite[Corollary~3.9]{KKT}\label{prop:QHC-PBW}
The algebra $RC_n$ is free over ${\rm R}$ with basis
\[
 \{y^ac^\eta\sigma_we(\mathbf i):a\in\mathbb Z_{\ge0}^n,
      \ \eta\in\{0,1\}^n,\ w\in\Sym_n,\ \mathbf i\in J^n\}.
\]
\end{prop}

\label{pag:cyclotomic quiver Hecke-Clifford superalgebra}
\begin{defn}\cite[Section 3.7]{KKT}\label{def:cyclotomic-QHC}
Let $\Lambda\in P^+$.
The cyclotomic quiver Hecke-Clifford superalgebra is
\begin{equation*}
 RC^\Lambda_n:=RC_n\bigg/\left\langle
 \sum_{{\bf i}\in J^n}y_1^{\<h_{\pr({\bf i}_1)},\Lambda\>}e({\bf i})
 \right\rangle .
\end{equation*}
\end{defn}
$RC^\Lambda_n$ inherits the $\Z\times\Z_2$-grading from $RC_n$.
We use the same symbols for the images of the generators in the quotient.

\begin{rem}
The algebras $RC_n$ and $RC^\Lambda_n$ have an anti-involution $*$
that sends the generators $e({\bf i}),$ $y_p,$ $\sigma_a,$ $c_p$ to themselves.
\end{rem}

\begin{notation}\label{not:QHC-components}
\label{pag:J nu}
For $\alpha\in Q_n^+$, set
$$J^\alpha:=\Bigl\{{\bf i}=({\bf i}_1,\cdots, {\bf i}_n)\in J^n\Bigm|\sum_{s=1}^{n}\nu_{\pr({\bf i}_s)}=\alpha\Bigr\}.$$
Let $e^J_\alpha:=\sum_{{\bf i}\in J^{\alpha}}e({\bf i})$ be the corresponding central idempotent, and set
\label{pag:J-central-summands}
$$RC_\alpha:=e^J_\alpha RC_n, \qquad RC^\Lambda_\alpha:=e^J_\alpha RC^\Lambda_n.$$
For infinite $I$, use these finite-word presentations for each $\alpha$
and take their direct sum, as in \eqref{eq:infinite-index-direct-sum}.
Recall the canonical projection $\pr: J\to I$. We choose $J^\dag \subset J$ such that the projection $\pr$ induces a bijection $J^\dag\to I$, and put $e^\dag:=\sum_{{\bf i}\in (J^\dag)^n}e({\bf i})$.  For infinite $I$, the expression $e^\dag RC^\Lambda_\alpha e^\dag$ is interpreted componentwise at fixed $\alpha$, where the words are finite.
The Clifford conjugation in \cite[below Definition~3.10]{KKT} gives
\begin{align}\label{eq:KKT-def3.10}
RC^\Lambda_\alpha e^\dag RC^\Lambda_\alpha = RC^\Lambda_\alpha.
\end{align}

\label{pag:m(nu)}
\label{m(nu)}
For $\alpha=\sum_{i\in I}m_i\nu_i\in Q^+$, put
$m(\alpha):=\sum_{i\in I_{\rm odd}}m_i$.
\end{notation}
Kang, Kashiwara and Tsuchioka \cite{KKT} proved that the (cyclotomic) quiver Hecke superalgebra and the (cyclotomic) quiver Hecke-Clifford superalgebra are weakly Morita superequivalent to each other.

\begin{thm}\cite[Below Definition 3.10, Theorem 3.13]{KKT}\label{supermorita}
Let $\Lambda\in P^+$ and $\alpha=\sum_{i\in I}m_i\nu_i\in Q^+$.

\begin{enumerate}

\item We have a
$\Z\times \Z_2$-graded ${\rm R}$-algebra isomorphism
\begin{align*}
R_{\alpha}^\Lambda\otimes \mathcal{C}_{m(\alpha)}\cong e^\dag RC_{\alpha}^\Lambda e^\dag.
\end{align*}

\item There is a Morita superequivalence over $\mathbb K$,  $$RC_{\alpha}^\Lambda \overset{\rm sMor}{\sim} e^\dag RC_{\alpha}^\Lambda e^\dag.$$
\end{enumerate}

\end{thm}

\subsection{Dynkin diagrams}\label{DynkinDiagrams}
For the spectral orbits in \cite[\S4.4]{KKT}, we record the Dynkin
diagrams and vertex labels of the associated Cartan superdata for use in
the cyclotomic comparison. We use $I=\mathbb Z$ for $A_\infty$, $I=\mathbb Z_{\ge0}$
for $B_\infty,C_\infty$, and $I=\{0,\ldots,e\}$ for the affine types.
An odd vertex is indicated by a crossed circle; the other vertices are even.

\label{pag:lie types}
\begin{enumerate}

\item When $e=\infty$, there are three types of Dynkin diagrams:

	\begin{enumerate}
		\item\leavevmode
		\begin{minipage}[t]{.18\linewidth}\raggedright
		\vspace{0pt}$A_\infty$:
		\end{minipage}\hfill
		\begin{minipage}[t]{.78\linewidth}\centering
		\vspace{0pt}
			\unitlength 0.1in
			\begin{picture}( 20.0000,  3.0000)(  0.0000, -2.5000)%
{\color{black}{%
\special{pn 8}%
\special{ar 600 200 50 50  0.0000000 6.2831853}%
}}%
{\color{black}{%
\special{pn 8}%
\special{ar 1000 200 50 50  0.0000000 6.2831853}%
}}%
{\color{black}{%
\special{pn 8}%
\special{ar 1400 200 50 50  0.0000000 6.2831853}%
}}%
{\color{black}{%
\special{pn 8}%
\special{pa 650 200}%
\special{pa 950 200}%
\special{fp}%
}}%
{\color{black}{%
\special{pn 8}%
\special{pa 1050 200}%
\special{pa 1350 200}%
\special{fp}%
}}%
{\color{black}{%
\special{pn 8}%
\special{pa 1450 200}%
\special{pa 1750 200}%
\special{fp}%
}}%
{\color{black}{%
\special{pn 8}%
\special{pa 1050 200}%
\special{pa 1350 200}%
\special{fp}%
}}%
{\color{black}{%
\special{pn 8}%
\special{pa 250 200}%
\special{pa 550 200}%
\special{fp}%
}}%
{\color{black}{%
\special{pn 8}%
\special{pa 200 200}%
\special{pa 0 200}%
\special{dt 0.045}%
}}%
{\color{black}{%
\special{pn 8}%
\special{pa 2000 200}%
\special{pa 1800 200}%
\special{dt 0.045}%
}}%
\put(6.0000,-0.8000){\makebox(0,0)[c]{\footnotesize$-1$}}%
\put(10.0000,-0.8000){\makebox(0,0)[c]{\footnotesize$0$}}%
\put(14.0000,-0.8000){\makebox(0,0)[c]{\footnotesize$1$}}%
\end{picture}%
		\end{minipage}

		\item\leavevmode
		\begin{minipage}[t]{.18\linewidth}\raggedright
		\vspace{0pt}$B_\infty$:
		\end{minipage}\hfill
		\begin{minipage}[t]{.78\linewidth}\centering
		\vspace{0pt}
			\unitlength 0.1in
			\begin{picture}( 14.5000,  3.0000)(  1.5000, -2.5000)%
{\color{black}{%
\special{pn 8}%
\special{ar 600 200 50 50  0.0000000 6.2831853}%
}}%
{\color{black}{%
\special{pn 8}%
\special{ar 1000 200 50 50  0.0000000 6.2831853}%
}}%
{\color{black}{%
\special{pn 8}%
\special{ar 200 200 50 50  0.0000000 6.2831853}%
}}%
{\color{black}{%
\special{pn 8}%
\special{pa 650 200}%
\special{pa 950 200}%
\special{fp}%
}}%
{\color{black}{%
\special{pn 8}%
\special{pa 1050 200}%
\special{pa 1350 200}%
\special{fp}%
}}%
{\color{black}{%
\special{pn 8}%
\special{pa 650 200}%
\special{pa 950 200}%
\special{fp}%
}}%
{\color{black}{%
\special{pn 8}%
\special{pa 1600 200}%
\special{pa 1400 200}%
\special{dt 0.045}%
}}%
{\color{black}{%
\special{pn 8}%
\special{pa 250 180}%
\special{pa 550 180}%
\special{fp}%
}}%
{\color{black}{%
\special{pn 8}%
\special{pa 250 220}%
\special{pa 550 220}%
\special{fp}%
}}%
{\color{black}{%
\special{pn 8}%
\special{pa 350 200}%
\special{pa 450 150}%
\special{fp}%
}}%
{\color{black}{%
\special{pn 8}%
\special{pa 450 250}%
\special{pa 350 200}%
\special{fp}%
}}%
\put(2.0500,-2.55000){\makebox(0,0)[lb]{\kern-1ex$\times$}}%
\put(2.0000,-0.8000){\makebox(0,0)[c]{\footnotesize$0$}}%
\put(6.0000,-0.8000){\makebox(0,0)[c]{\footnotesize$1$}}%
\put(10.0000,-0.8000){\makebox(0,0)[c]{\footnotesize$2$}}%
\end{picture}%
		\end{minipage}

		\item\leavevmode
		\begin{minipage}[t]{.18\linewidth}\raggedright
		\vspace{0pt}$C_\infty$:
		\end{minipage}\hfill
		\begin{minipage}[t]{.78\linewidth}\centering
		\vspace{0pt}
			\unitlength 0.1in
			\begin{picture}( 16.0000,  3.0000)(  0.0000, -2.5000)%
{\color{black}{%
\special{pn 8}%
\special{ar 600 200 50 50  0.0000000 6.2831853}%
}}%
{\color{black}{%
\special{pn 8}%
\special{ar 1000 200 50 50  0.0000000 6.2831853}%
}}%
{\color{black}{%
\special{pn 8}%
\special{ar 200 200 50 50  0.0000000 6.2831853}%
}}%
{\color{black}{%
\special{pn 8}%
\special{pa 650 200}%
\special{pa 950 200}%
\special{fp}%
}}%
{\color{black}{%
\special{pn 8}%
\special{pa 1050 200}%
\special{pa 1350 200}%
\special{fp}%
}}%
{\color{black}{%
\special{pn 8}%
\special{pa 650 200}%
\special{pa 950 200}%
\special{fp}%
}}%
{\color{black}{%
\special{pn 8}%
\special{pa 1600 200}%
\special{pa 1400 200}%
\special{dt 0.045}%
}}%
{\color{black}{%
\special{pn 8}%
\special{pa 250 180}%
\special{pa 550 180}%
\special{fp}%
}}%
{\color{black}{%
\special{pn 8}%
\special{pa 250 220}%
\special{pa 550 220}%
\special{fp}%
}}%
{\color{black}{%
\special{pn 8}%
\special{pa 350 150}%
\special{pa 450 200}%
\special{fp}%
}}%
{\color{black}{%
\special{pn 8}%
\special{pa 450 200}%
\special{pa 350 250}%
\special{fp}%
}}%
\put(2.0000,-0.8000){\makebox(0,0)[c]{\footnotesize$0$}}%
\put(6.0000,-0.8000){\makebox(0,0)[c]{\footnotesize$1$}}%
\put(10.0000,-0.8000){\makebox(0,0)[c]{\footnotesize$2$}}%
\end{picture}%
		\end{minipage}
	\end{enumerate}

\item  When $e\geq 1$ is finite, there are four types of Dynkin diagrams:

	\begin{enumerate}
\item\leavevmode
\begin{minipage}[t]{.20\linewidth}\raggedright
\vspace{0pt}$A_e^{(1)}$, $e\ge2$:
\end{minipage}\hfill
\begin{minipage}[t]{.76\linewidth}\centering
\vspace{0pt}
			\unitlength 0.1in
			\begin{picture}( 20.5000,  4.5000)(  0.5000, -4.0000)%
{\color{black}{%
\special{pn 8}%
\special{ar 600 200 50 50  0.0000000 6.2831853}%
}}%
{\color{black}{%
\special{pn 8}%
\special{ar 1000 200 50 50  0.0000000 6.2831853}%
}}%
{\color{black}{%
\special{pn 8}%
\special{ar 200 200 50 50  0.0000000 6.2831853}%
}}%
{\color{black}{%
\special{pn 8}%
\special{pa 250 200}%
\special{pa 550 200}%
\special{fp}%
}}%
{\color{black}{%
\special{pn 8}%
\special{pa 650 200}%
\special{pa 950 200}%
\special{fp}%
}}%
{\color{black}{%
\special{pn 8}%
\special{pa 1050 200}%
\special{pa 1350 200}%
\special{fp}%
}}%
{\color{black}{%
\special{pn 8}%
\special{pa 650 200}%
\special{pa 950 200}%
\special{fp}%
}}%
{\color{black}{%
\special{pn 8}%
\special{pa 1600 200}%
\special{pa 1400 200}%
\special{dt 0.045}%
}}%
{\color{black}{%
\special{pn 8}%
\special{pa 1700 200}%
\special{pa 2000 200}%
\special{fp}%
}}%
{\color{black}{%
\special{pn 8}%
\special{ar 2050 200 50 50  0.0000000 6.2831853}%
}}%
{\color{black}{%
\special{pn 8}%
\special{pa 200 250}%
\special{pa 1120 400}%
\special{fp}%
\special{pa 1120 400}%
\special{pa 2050 250}%
\special{fp}%
}}%
\put(2.0000,-0.8000){\makebox(0,0)[c]{\footnotesize$0$}}%
\put(6.0000,-0.8000){\makebox(0,0)[c]{\footnotesize$1$}}%
\put(10.0000,-0.8000){\makebox(0,0)[c]{\footnotesize$2$}}%
\put(20.5000,-0.8000){\makebox(0,0)[c]{\footnotesize$e$}}%
\end{picture}%
	\end{minipage}
		\label{case:Loop}

\item\leavevmode
\begin{minipage}[t]{.20\linewidth}\raggedright
\vspace{0pt}$A_{2e}^{(2)}$, $e\ge1$:
\end{minipage}\hfill
\begin{minipage}[t]{.76\linewidth}\centering
\vspace{0pt}
\begin{minipage}[t]{.24\linewidth}\centering
		\begingroup\centering
			\unitlength 0.1in
			\begin{picture}(  5.0000,  3.0000)(  1.5000, -2.5000)%
{\color{black}{%
\special{pn 8}%
\special{ar 600 200 50 50  0.0000000 6.2831853}%
}}%
{\color{black}{%
\special{pn 8}%
\special{ar 200 200 50 50  0.0000000 6.2831853}%
}}%
{\color{black}{%
\special{pn 8}%
\special{pa 250 170}%
\special{pa 550 170}%
\special{fp}%
}}%
\put(2.0500,-2.55000){\makebox(0,0)[lb]{\kern-1ex$\times$}}%
{\color{black}{%
\special{pn 8}%
\special{pa 250 230}%
\special{pa 550 230}%
\special{fp}%
}}%
{\color{black}{%
\special{pn 8}%
\special{pa 350 200}%
\special{pa 450 150}%
\special{fp}%
}}%
{\color{black}{%
\special{pn 8}%
\special{pa 450 250}%
\special{pa 350 200}%
\special{fp}%
}}%
{\color{black}{%
\special{pn 8}%
\special{pa 250 190}%
\special{pa 550 190}%
\special{fp}%
}}%
{\color{black}{%
\special{pn 8}%
\special{pa 250 210}%
\special{pa 550 210}%
\special{fp}%
}}%
\put(2.0000,-0.8000){\makebox(0,0)[c]{\footnotesize$0$}}%
\put(6.0000,-0.8000){\makebox(0,0)[c]{\footnotesize$1$}}%
\end{picture}%
		\par\smallskip\text{($e=1$)}
		\par\endgroup
\end{minipage}\hfill
\begin{minipage}[t]{.70\linewidth}\centering

		\begingroup\centering
			\unitlength 0.1in
			\begin{picture}( 19.6000,  3.0000)(  1.5000, -2.5000)%
{\color{black}{%
\special{pn 8}%
\special{ar 600 200 50 50  0.0000000 6.2831853}%
}}%
{\color{black}{%
\special{pn 8}%
\special{ar 200 200 50 50  0.0000000 6.2831853}%
}}%
{\color{black}{%
\special{pn 8}%
\special{pa 250 180}%
\special{pa 550 180}%
\special{fp}%
}}%
{\color{black}{%
\special{pn 8}%
\special{pa 650 200}%
\special{pa 950 200}%
\special{fp}%
}}%
{\color{black}{%
\special{pn 8}%
\special{pa 650 200}%
\special{pa 950 200}%
\special{fp}%
}}%
{\color{black}{%
\special{pn 8}%
\special{pa 1200 200}%
\special{pa 1000 200}%
\special{dt 0.045}%
}}%
\put(2.0500,-2.55000){\makebox(0,0)[lb]{\kern-1ex$\times$}}%
{\color{black}{%
\special{pn 8}%
\special{pa 1300 200}%
\special{pa 1600 200}%
\special{fp}%
}}%
{\color{black}{%
\special{pn 8}%
\special{ar 1650 200 50 50  0.0000000 6.2831853}%
}}%
{\color{black}{%
\special{pn 8}%
\special{pa 250 220}%
\special{pa 550 220}%
\special{fp}%
}}%
{\color{black}{%
\special{pn 8}%
\special{pa 350 200}%
\special{pa 450 150}%
\special{fp}%
}}%
{\color{black}{%
\special{pn 8}%
\special{pa 450 250}%
\special{pa 350 200}%
\special{fp}%
}}%
{\color{black}{%
\special{pn 8}%
\special{ar 2060 200 50 50  0.0000000 6.2831853}%
}}%
{\color{black}{%
\special{pn 8}%
\special{pa 1710 180}%
\special{pa 2010 180}%
\special{fp}%
}}%
{\color{black}{%
\special{pn 8}%
\special{pa 1710 220}%
\special{pa 2010 220}%
\special{fp}%
}}%
{\color{black}{%
\special{pn 8}%
\special{pa 1810 200}%
\special{pa 1910 150}%
\special{fp}%
}}%
{\color{black}{%
\special{pn 8}%
\special{pa 1910 250}%
\special{pa 1810 200}%
\special{fp}%
}}%
\put(2.0000,-0.8000){\makebox(0,0)[c]{\footnotesize$0$}}%
\put(6.0000,-0.8000){\makebox(0,0)[c]{\footnotesize$1$}}%
\put(16.5000,-0.8000){\makebox(0,0)[c]{\footnotesize$e-1$}}%
\put(20.6000,-0.8000){\makebox(0,0)[c]{\footnotesize$e$}}%
\end{picture}%
			\par\smallskip\text{($e>1$)}
		\par\endgroup
\end{minipage}
\end{minipage}

		\item
\leavevmode
\begin{minipage}[t]{.20\linewidth}\raggedright
\vspace{0pt}$C_e^{(1)}$, $e\ge2$:
\end{minipage}\hfill
\begin{minipage}[t]{.76\linewidth}\centering
\vspace{0pt}
			\unitlength 0.1in
			\begin{picture}( 19.6000,  3.0000)(  1.5000, -2.5000)%
{\color{black}{%
\special{pn 8}%
\special{ar 600 200 50 50  0.0000000 6.2831853}%
}}%
{\color{black}{%
\special{pn 8}%
\special{ar 200 200 50 50  0.0000000 6.2831853}%
}}%
{\color{black}{%
\special{pn 8}%
\special{pa 1710 180}%
\special{pa 2010 180}%
\special{fp}%
}}%
{\color{black}{%
\special{pn 8}%
\special{pa 650 200}%
\special{pa 950 200}%
\special{fp}%
}}%
{\color{black}{%
\special{pn 8}%
\special{pa 650 200}%
\special{pa 950 200}%
\special{fp}%
}}%
{\color{black}{%
\special{pn 8}%
\special{pa 1200 200}%
\special{pa 1000 200}%
\special{dt 0.045}%
}}%
{\color{black}{%
\special{pn 8}%
\special{pa 1300 200}%
\special{pa 1600 200}%
\special{fp}%
}}%
{\color{black}{%
\special{pn 8}%
\special{ar 1650 200 50 50  0.0000000 6.2831853}%
}}%
{\color{black}{%
\special{pn 8}%
\special{pa 1710 220}%
\special{pa 2010 220}%
\special{fp}%
}}%
{\color{black}{%
\special{pn 8}%
\special{pa 1810 200}%
\special{pa 1910 150}%
\special{fp}%
}}%
{\color{black}{%
\special{pn 8}%
\special{pa 1910 250}%
\special{pa 1810 200}%
\special{fp}%
}}%
{\color{black}{%
\special{pn 8}%
\special{ar 2060 200 50 50  0.0000000 6.2831853}%
}}%
{\color{black}{%
\special{pn 8}%
\special{pa 250 180}%
\special{pa 550 180}%
\special{fp}%
}}%
{\color{black}{%
\special{pn 8}%
\special{pa 250 220}%
\special{pa 550 220}%
\special{fp}%
}}%
{\color{black}{%
\special{pn 8}%
\special{pa 350 250}%
\special{pa 450 200}%
\special{fp}%
}}%
{\color{black}{%
\special{pn 8}%
\special{pa 450 200}%
\special{pa 350 150}%
\special{fp}%
}}%
\put(2.0000,-0.8000){\makebox(0,0)[c]{\footnotesize$0$}}%
\put(6.0000,-0.8000){\makebox(0,0)[c]{\footnotesize$1$}}%
\put(16.5000,-0.8000){\makebox(0,0)[c]{\footnotesize$e-1$}}%
\put(20.6000,-0.8000){\makebox(0,0)[c]{\footnotesize$e$}}%
\end{picture}%
	\end{minipage}

		\item\leavevmode
\begin{minipage}[t]{.20\linewidth}\raggedright
\vspace{0pt}$D_{e+1}^{(2)}$, $e\ge1$:
\end{minipage}\hfill
\begin{minipage}[t]{.76\linewidth}\centering
\vspace{0pt}
\begin{minipage}[t]{.24\linewidth}\centering
		\begingroup\centering
			\unitlength 0.1in
			\begin{picture}(  5.0000,  3.0000)(  1.5000, -2.5000)%
{\color{black}{%
\special{pn 8}%
\special{ar 600 200 50 50  0.0000000 6.2831853}%
}}%
{\color{black}{%
\special{pn 8}%
\special{ar 200 200 50 50  0.0000000 6.2831853}%
}}%
{\color{black}{%
\special{pn 8}%
\special{pa 250 180}%
\special{pa 550 180}%
\special{fp}%
}}%
\put(5.37000,-2.55000){\makebox(0,0)[lb]{$\times$}}%
\put(1.34500,-2.55000){\makebox(0,0)[lb]{$\times$}}%
{\color{black}{%
\special{pn 8}%
\special{pa 250 220}%
\special{pa 550 220}%
\special{fp}%
}}%
{\color{black}{%
\special{pn 8}%
\special{pa 300 200}%
\special{pa 400 150}%
\special{fp}%
}}%
{\color{black}{%
\special{pn 8}%
\special{pa 400 250}%
\special{pa 300 200}%
\special{fp}%
}}%
{\color{black}{%
\special{pn 8}%
\special{pa 400 250}%
\special{pa 500 200}%
\special{fp}%
}}%
{\color{black}{%
\special{pn 8}%
\special{pa 500 200}%
\special{pa 400 150}%
\special{fp}%
}}%
\put(2.0000,-0.8000){\makebox(0,0)[c]{\footnotesize$0$}}%
\put(6.0000,-0.8000){\makebox(0,0)[c]{\footnotesize$1$}}%
\end{picture}%
			\par\smallskip\text{($e=1$)}
		\par\endgroup
\end{minipage}\hfill
\begin{minipage}[t]{.70\linewidth}\centering

		\begingroup\centering
			\unitlength 0.1in
			\begin{picture}( 19.6000,  3.0000)(  1.5000, -2.5000)%
{\color{black}{%
\special{pn 8}%
\special{ar 600 200 50 50  0.0000000 6.2831853}%
}}%
{\color{black}{%
\special{pn 8}%
\special{ar 200 200 50 50  0.0000000 6.2831853}%
}}%
{\color{black}{%
\special{pn 8}%
\special{pa 250 180}%
\special{pa 550 180}%
\special{fp}%
}}%
{\color{black}{%
\special{pn 8}%
\special{pa 650 200}%
\special{pa 950 200}%
\special{fp}%
}}%
{\color{black}{%
\special{pn 8}%
\special{pa 650 200}%
\special{pa 950 200}%
\special{fp}%
}}%
{\color{black}{%
\special{pn 8}%
\special{pa 1200 200}%
\special{pa 1000 200}%
\special{dt 0.045}%
}}%
\put(2.0500,-2.55000){\makebox(0,0)[lb]{\kern-1ex$\times$}}%
{\color{black}{%
\special{pn 8}%
\special{pa 1300 200}%
\special{pa 1600 200}%
\special{fp}%
}}%
{\color{black}{%
\special{pn 8}%
\special{ar 1650 200 50 50  0.0000000 6.2831853}%
}}%
{\color{black}{%
\special{pn 8}%
\special{pa 250 220}%
\special{pa 550 220}%
\special{fp}%
}}%
{\color{black}{%
\special{pn 8}%
\special{pa 350 200}%
\special{pa 450 150}%
\special{fp}%
}}%
{\color{black}{%
\special{pn 8}%
\special{pa 450 250}%
\special{pa 350 200}%
\special{fp}%
}}%
{\color{black}{%
\special{pn 8}%
\special{ar 2060 200 50 50  0.0000000 6.2831853}%
}}%
{\color{black}{%
\special{pn 8}%
\special{pa 1710 180}%
\special{pa 2010 180}%
\special{fp}%
}}%
{\color{black}{%
\special{pn 8}%
\special{pa 1710 220}%
\special{pa 2010 220}%
\special{fp}%
}}%
{\color{black}{%
\special{pn 8}%
\special{pa 1810 250}%
\special{pa 1910 200}%
\special{fp}%
}}%
{\color{black}{%
\special{pn 8}%
\special{pa 1910 200}%
\special{pa 1810 150}%
\special{fp}%
}}%
\put(19.980000,-2.55000){\makebox(0,0)[lb]{$\times$}}%
\put(2.0000,-0.8000){\makebox(0,0)[c]{\footnotesize$0$}}%
\put(6.0000,-0.8000){\makebox(0,0)[c]{\footnotesize$1$}}%
\put(16.5000,-0.8000){\makebox(0,0)[c]{\footnotesize$e-1$}}%
\put(20.6000,-0.8000){\makebox(0,0)[c]{\footnotesize$e$}}%
\end{picture}%
			\par\smallskip\text{($e>1$)}
		\par\endgroup
\end{minipage}
\end{minipage}
	\end{enumerate}
\end{enumerate}

The diagrams determine the Cartan superdata used in the rest of the paper.
\begin{defn}\label{def:KKT-Cartan-superdata}
Orient each single edge arbitrarily.  The generalized Cartan matrix and
the odd vertex set are
\begin{align*}
	a_{ij}&:=
	\begin{cases}
		-1,& \text{ if } i\rightarrow j, i \leftarrow j, i\Rightarrow j \text{ or } i \fourlinerightarrow j, \\
		-2,& \text{ if } i\Leftarrow j \text{ or } i \Leftrightarrow j, \\
        -4,& \text{ if } i \fourlineleftarrow j, \\
		2, & \text{ if } i =j, \\
		0, & \text{ otherwise.}
	\end{cases}\\
 I_{\rm odd}&:=\{i\in I\mid a_{ij}\in 2\mathbb{Z}\text{ for all }j\in I\}.
\end{align*}
Together with the lattices fixed in Section~\ref{preli}, these data form
a Cartan superdatum.
\end{defn}
The following polynomials are the parameters $Q_{i,j}$ for the quiver
Hecke superalgebra.
\begin{defn}\label{def:KKT-Q-polynomials}
For indeterminates $u,v$ over $\mathbb K$ and $i,j\in I$, define
\begin{align*}
	Q_{i,j}(u,v):=
	\begin{cases}
		u-v, & \text{ if } i\rightarrow j, \\
		v-u, & \text{ if } i\leftarrow j, \\
		u-v^{2}, & \text{ if } i\Rightarrow j, \\
		v-u^{2}, & \text{ if } i\Leftarrow j, \\
		(u-v)(v-u), & \text{ if } i\Leftrightarrow j, \\
		u-v^{4}, & \text{ if } i\fourlinerightarrow  j, \\
		v-u^{4}, & \text{ if } i\fourlineleftarrow  j, \\
		0, & \text{ if } i= j, \\
		1, & \text{ otherwise.}
	\end{cases}
\end{align*}
\end{defn}
They define $R_n(I)$ by its presentation above; the parameters for
$RC_n(I)$ are the polynomials $\widetilde Q_{j,j'}$ obtained from them
by \eqref{extend Q-polys}.

\begin{defn}\label{def:KKT-spectral-labelling}
Following \cite[\S4.4]{KKT}, label the spectral vertices as follows.  Let
$$\mathrm h:=\operatorname{ord}(q^2)\in\mathbb Z_{\ge1}\cup\{\infty\},$$
and
\[
 \langle q^2\rangle=\{q^{2k}:k\in\mathbb Z\},\qquad
 \sigma\in\{\pm 1\},\qquad
 a\in\mathbb K^\times,\quad (a^2\notin\langle q^2\rangle).
\]
The parameters $a$ and $\sigma$ are chosen for the orbit in the cases where they occur:
\begin{enumerate}[label=\textup{(\roman*)}]
\item For type $A_e^{(1)}$, let $\mathrm h=e+1$, take $0\le t\le e$, and attach
the spectral vertex $\mathtt q(aq^{2t})$ to $t$.
\item For type $C_e^{(1)}$, let $\mathrm h=2e$, take $0\le t\le e$, and attach
$\mathtt q(\sigma q^{2t-1})$ to $t$.
\item For type $A_{2e}^{(2)}$, let $\mathrm h=2e+1$, take $0\le t\le e$, and attach
$\mathtt q(\sigma q^{2t})$ to $t$.
\item For type $D_{e+1}^{(2)}$, let $\mathrm h=2e+2$, take $0\le t\le e$, and attach
$\mathtt q(\sigma q^{2t})$ to $t$.
\item For type $A_\infty$, let $\mathrm h=\infty$, take $t\in\mathbb Z$, and attach
$\mathtt q(aq^{2t})$ to $t$.
\item For type $B_\infty$, let $\mathrm h=\infty$, take $t\ge0$, and attach
$\mathtt q(\sigma q^{2t})$ to $t$.
\item For type $C_\infty$, let $\mathrm h=\infty$, take $t\ge0$, and attach
$\mathtt q(\sigma q^{2t-1})$ to $t$.
\end{enumerate}
For a spectral vertex $w$, its parity is
\[
 {\rm p}(w)=
 \begin{cases}
 \bar1,&w\in\{2,-2\},\\
 \bar0,&w\notin\{2,-2\}.
 \end{cases}
\]
\end{defn}
\subsection{KKT's isomorphism}\label{KKT's isomorphism}
We associate a dominant weight to $f$ and apply the cyclotomic KKT isomorphism.
\begin{notation}\label{not:KKT-spectral-data}
Fix $q\in\mathbb K^\times$ with $q^4\ne1$, and let
$f=f^{(\bullet)}_{\undQ}$ be as in \eqref{eq:cyclotomic-parameters}.
Using $\mathcal C_\bullet$ from \eqref{eq:SW-component-indices}, set
\begin{equation*}
 I_f:=\bigcup_{r\in\mathcal C_\bullet}
       \{\mathtt q(Q_rq^{2t}):t\in\mathbb Z\}.
\end{equation*}
Equip $I_f$ with the Cartan superdatum of Section~\ref{DynkinDiagrams}
under the labelling in Definition~\ref{def:KKT-spectral-labelling}.
For $w\in I_f$, write $\Lambda_w$ for the fundamental weight at $w$.
\end{notation}

The weight records each parameter with its multiplicity; the additional
factors $X-1$ and $X+1$ in \eqref{eq:cyclotomic-parameters} contribute
the vertices $2$ and $-2$ in $I_f$, respectively.
\begin{defn}\label{def:KKT-weight}
Define
\begin{equation}\label{eq:KKT-weight}
 \Lambda_f:=\begin{cases}
 \displaystyle\sum_{j=1}^m
  2^{\delta_{{\rm p}(\mathtt q(Q_j)),\bar1}}\Lambda_{\mathtt q(Q_j)},
   &\bullet=\mathtt{0},\\[1ex]
 \displaystyle\Lambda_2+\sum_{j=1}^m
  2^{\delta_{{\rm p}(\mathtt q(Q_j)),\bar1}}\Lambda_{\mathtt q(Q_j)},
   &\bullet=\mathtt{s},\\[1ex]
 \displaystyle\Lambda_2+\Lambda_{-2}+\sum_{j=1}^m
  2^{\delta_{{\rm p}(\mathtt q(Q_j)),\bar1}}\Lambda_{\mathtt q(Q_j)},
   &\bullet=\mathtt{ss}.
 \end{cases}
\end{equation}
The sum over $j$ retains repeated parameters, including those with
$\mathtt q(Q_j)=2$ or $-2$; the endpoint terms count only the additional
linear factors in \eqref{eq:cyclotomic-parameters}.
\end{defn}

For fixed $n$, \cite[Lemma~4.7(ii)]{KKT} bounds the possible spectral vertices by the following finite set.
\begin{notation}\label{not:finite-KKT-support}
For $n\ge1$, set
\begin{equation*}
 I_{f,n}:=\bigcup_{r\in\mathcal C_\bullet}
       \{\mathtt q(Q_rq^{2t}):t\in\mathbb Z,\ -n<t<n\}\subseteq I_f,
\end{equation*}
and give $I_{f,n}$ the induced full subdiagram structure.
\end{notation}

\begin{thm}\cite[Corollary~4.8]{KKT}\label{KKTiso}
For $n\ge1$, there is a $\mathbb K$-superalgebra isomorphism
\begin{equation*}
 \mathcal H^f_{\mathbb K}(n)
 \cong RC_n^{\Lambda_f|_{I_{f,n}}}(I_{f,n}).
\end{equation*}
\end{thm}
\begin{proof}
Since $\mathtt q(u)=g(qu)$ in the notation of
\cite[\S4.4]{KKT}, the inverse image of $I_{f,n}$ under
$X\mapsto X+X^{-1}$ is the set $J_n$ of \cite[Lemma~4.7(ii)]{KKT}.
Apply \cite[Corollary~4.8]{KKT} with $J=J_n$.
\end{proof}

The following reflection interchanges the two endpoints of type
$D_{e+1}^{(2)}$.
\begin{lem}\label{lem:D-end-reflection}
In type $D_{e+1}^{(2)}$, let $\varphi:I\to I$ be given by
$\varphi(i)=e-i$.  For $n\ge0$ and $\alpha\in Q_n^+(I)$, there is a
graded superalgebra isomorphism
\begin{equation*}
 R^\Lambda_\alpha(I)\cong
 R^{\varphi_*\Lambda}_{\varphi_*\alpha}(I),
\end{equation*}
where
\[
 \langle h_{e-i},\varphi_*\Lambda\rangle=\langle h_i,\Lambda\rangle,
 \qquad \varphi_*\nu_i=\nu_{e-i},
 \qquad Q^\varphi_{e-i,e-j}(u,v)=Q_{i,j}(u,v).
\]
The algebra on the right uses the transported defining polynomials
$Q^\varphi$.
\end{lem}
\begin{proof}
For $\mathbf j=(j_1,\ldots,j_n)\in I^\alpha$, write
$\varphi\mathbf j:=(e-j_1,\ldots,e-j_n)$ and send
\[
 \begin{aligned}
 e(\mathbf j)&\longmapsto e(\varphi\mathbf j),\\
 x_a e(\mathbf j)&\longmapsto x_a e(\varphi\mathbf j),\\
 \tau_b e(\mathbf j)&\longmapsto\tau_b e(\varphi\mathbf j),
 \end{aligned}
\]
These maps preserve the defining and cyclotomic relations. Applying them
twice is the identity.
\end{proof}

By Theorem \ref{supermorita}, we have
\begin{lem}\cite[Corollary~7.8]{Shi:semisimplicity}\label{lem:kkt-transport}
For $\Lambda\in P^+$ and $n\ge1$, the following conditions are equivalent:
\begin{enumerate}[label=\textup{(\roman*)}]
\item $R_n^\Lambda$ is semisimple.
\item $RC_n^\Lambda$ is semisimple.
\end{enumerate}
\end{lem}

\begin{cor}\label{cor:KKT-semisimplicity}
For $n\ge1$, the following conditions are equivalent:
\begin{enumerate}[label=\textup{(\roman*)}]
\item $\mathcal H^f_{\mathbb K}(n)$ is semisimple;
\item $R_n^{\Lambda_f|_{I_{f,n}}}(I_{f,n})$ is semisimple.
\end{enumerate}
\end{cor}
\begin{proof}
Apply Theorem~\ref{KKTiso} and Lemma~\ref{lem:kkt-transport}.
\end{proof}
\subsection{Residues and nonzero components}\label{subsec:single-orbit-data}
Let $(\mathrm A,P,\Pi,\Pi^\vee)$ be a Cartan superdatum with vertex set
$I$. For the connected types in Section~\ref{DynkinDiagrams}, use
their integer vertex sets $I$ and fix an arbitrary $\Lambda\in P^+$
(of finite support for infinite $I$).

For $t\in\mathbb Z$, set
\[
[t]_{\mathrm h}:=
\begin{cases}
t-\mathrm h\left\lfloor t/\mathrm h\right\rfloor,&\mathrm h\in\mathbb Z_{\ge1},\\
t,&\mathrm h=\infty.
\end{cases}
\]
Define the following map $\mathfrak f_I:\mathbb Z\to I$.
\begin{defn}\label{def:folding-map}
For the fixed Cartan type, define $\mathfrak f_I:\mathbb Z\longrightarrow I$ by
\begin{equation}\label{eq:folding-map}
\begin{array}{c|c|l}
\text{Type}&I&\mathfrak f_I(t)\\\hline
 A_\infty&\mathbb Z&t\\
 B_\infty&\mathbb Z_{\ge0}&\max\{t,-1-t\}\\
 C_\infty&\mathbb Z_{\ge0}&|t|\\
 A_e^{(1)}&\{0,\ldots,e\}&[t]_{e+1}\\
 C_e^{(1)}&\{0,\ldots,e\}&\min\{[t]_{2e},\,2e-[t]_{2e}\}\\
 A_{2e}^{(2)}&\{0,\ldots,e\}&\min\{[t]_{2e+1},\,2e-[t]_{2e+1}\}\\
 D_{e+1}^{(2)}&\{0,\ldots,e\}&\min\{[t]_{2e+2},\,2e+1-[t]_{2e+2}\}
\end{array}
\end{equation}
\end{defn}
To attach partitions to $\Lambda$, fix $m\ge0$,
$\bullet\in\{\mathtt{0},\mathtt{s},\mathtt{ss}\}$ and labels
$i_r\in I$ ($r\in\mathcal C_\bullet$) such that
\begin{equation}\label{eq:source-label-multiplicities}
 \langle h_i,\Lambda\rangle
 =\sum_{r=1}^m2^{\delta_{{\rm p}(i_r),\bar1}}\delta_{i,i_r}
  +\sum_{r\in\mathcal C_\bullet\setminus[m]}\delta_{i,i_r}
 \qquad(i\in I),
\end{equation}
with distinct odd labels at the strict positions. For $\bullet=\mathtt{ss}$,
take $i_{0_+}=0$ and $i_{0_-}=e$ in type $D_{e+1}^{(2)}$.
These labels are fixed throughout the following definitions.
\begin{defn}\label{def:box-vertex-residue}
For a box $A=(a,b,r)$ of
$\boldsymbol\lambda\in\mathscr P_n^{\bullet,m}$, define
\begin{equation}\label{eq:residue-formulas}
 \operatorname{res}_I(A):=
 \begin{cases}
 \mathfrak f_I(i_r+b-a),&r\in[m],\\
 \mathfrak f_I(b-a),&r\notin[m],\ i_r=0,\\
 \mathfrak f_I(e+1+b-a),&r\notin[m],\ i_r=e\text{ in }D_{e+1}^{(2)}.
 \end{cases}
\end{equation}
\end{defn}
For $\boldsymbol\lambda\in\mathscr P_n^{\bullet,m}$, define
\begin{equation}\label{eq:multipartition-root}
 \alpha_{\boldsymbol\lambda}:=
 \sum_{A\in\boldsymbol\lambda}\nu_{\operatorname{res}_I(A)}\in Q_n^+.
\end{equation}
\Needspace{18\baselineskip}
\begin{example}\label{ex:residue-coordinates}
The first diagram records $i+b-a$ for the ordinary partition
$(3,2)$ at label $i$. The second
records the strict partition $(4,2)$ at $0$, with residues
$\mathfrak f_I(b-a)$.  The third records the strict row $(3)$ at $e$ in
type $D_{e+1}^{(2)}$, where $e\ge2$; its residues are
$\mathfrak f_I(e+1+r)=\mathfrak f_I(e-r)$ for $r=0,1,2$.
\leavevmode\par\vspace{-.35\baselineskip}\begin{center}
\begin{tikzpicture}[x=.84cm,y=.80cm,
 every node/.style={font=\normalsize,inner sep=2pt}]
\foreach \x in {0,1,2}{\draw (\x,1) rectangle ++(1,1);}\foreach \x in {0,1}{\draw (\x,0) rectangle ++(1,1);}
\foreach \x/\lab in {.5/$i$,1.5/$i+1$,2.5/$i+2$}
 \node at (\x,1.5){\lab};
\node at (.5,.5){$i-1$};\node at (1.5,.5){$i$};
\foreach \x/\lab in {.5/1,1.5/2,2.5/3}
 \node[above] at (\x,2){$\lab$};
\node[left] at (0,1.5){$a=1$};\node[left] at (0,.5){$a=2$};
\node[above] at (-.45,2){$b$};\node at (1.5,2.9){ordinary $(3,2)$};
\node at (1.5,-.55){$\mathfrak f_I(i+b-a)$};
\begin{scope}[xshift=4.15cm]
 \foreach \x in {0,1,2,3}{\draw (\x,1) rectangle ++(1,1);}\foreach \x in {1,2}{\draw (\x,0) rectangle ++(1,1);}
 \foreach \x/\lab in {.5/0,1.5/1,2.5/2,3.5/3}
  \node at (\x,1.5){$\lab$};
 \node at (1.5,.5){$0$};\node at (2.5,.5){$1$};
 \foreach \x/\lab in {.5/1,1.5/2,2.5/3,3.5/4}
  \node[above] at (\x,2){$\lab$};
 \node[left] at (0,1.5){$a=1$};\node[left] at (1,.5){$a=2$};
 \node[above] at (-.45,2){$b$};\node at (2,2.9){left strict $(4,2)$};
 \node at (2,-.55){$\mathfrak f_I(b-a)$};
\end{scope}
\begin{scope}[xshift=9.1cm,x=.90cm]
 \foreach \x/\u/\v in {0/{e+1}/e,1/{e+2}/{e-1},2/{e+3}/{e-2}}{
  \draw (\x,1) rectangle ++(1,1);
  \node at (\x+.5,1.5){$\u$};
  \draw[->] (\x+.5,.9)--(\x+.5,.2);
  \node at (\x+.5,-.12){$\v$};}
 \node at (1.5,2.9){right strict $(3)$};
\end{scope}
\end{tikzpicture}
\end{center}
\end{example}

\begin{defn}\label{def:interval-root}
For $u,v\in\mathbb Z$, define
\begin{equation}\label{eq:interval-root}
 \gamma_I[u,v]:=
 \begin{cases}
  \displaystyle\sum_{t=u}^v\nu_{\mathfrak f_I(t)},&u\le v,\\
  0,&u>v.
 \end{cases}
\end{equation}
For an integer interval $J=[u,v]$, also write
$\gamma_I(J):=\gamma_I[u,v]$.
\end{defn}

\begin{example}
For an ordinary partition $(2,1)$ at label $i$ and a strict partition
$(3,1)$ at $0$, respectively,
\[
 \begin{aligned}
 \alpha_{(2,1)}&=\gamma_I[i-1,i+1],\\
 \alpha_{(3,1)}&=\gamma_I[0,2]+\gamma_I[0,0].
 \end{aligned}
\]
\end{example}
\begin{defn}\label{def:standard-source-tableaux}
For $\boldsymbol\lambda\in\mathscr P_n^{\bullet,m}$, define
\begin{equation*}
 \operatorname{Std}(\boldsymbol\lambda):=
 \left\{T:\boldsymbol\lambda\xrightarrow{\sim}[n]:
 \begin{aligned}
 T(a,b,r)&<T(a,b+1,r),\\
 T(a,b,r)&<T(a+1,b,r)
 \end{aligned}\right\},
\end{equation*}
where each inequality is imposed when both boxes occur in
$\boldsymbol\lambda$.
\end{defn}
For $r\in\mathcal C_\bullet\setminus[m]$, let
$\ell(\lambda^{(r)})$ be the number of rows of its strict component.

\begin{rem}\label{rem:D-end-reflection-combinatorics}
For type $D_{e+1}^{(2)}$, the reflection in
Lemma~\ref{lem:D-end-reflection} transposes every ordinary component,
replaces its label $i_r$ by $e-i_r$, and interchanges the left and right
strict components.  For the reflected multipartition
$\boldsymbol\lambda'$, one has
\begin{equation}\label{eq:D-folding-reflection}
 \begin{aligned}
 e-\mathfrak f_I(i+k)&=\mathfrak f_I(e-i-k)
       &&(i\in I,\ k\in\mathbb Z),\\
 e-\mathfrak f_I(e+1+k)&=\mathfrak f_I(k)
       &&(k\in\mathbb Z),\\
 \alpha_{\boldsymbol\lambda'}&=\varphi_*\alpha_{\boldsymbol\lambda},
 \qquad
 \#\operatorname{Std}(\boldsymbol\lambda')
 =\#\operatorname{Std}(\boldsymbol\lambda).
 \end{aligned}
\end{equation}
Thus, when needed below, the case with odd
$\langle h_e,\Lambda\rangle$ and even $\langle h_0,\Lambda\rangle$
may be replaced by the case with odd $\langle h_0,\Lambda\rangle$ and
even $\langle h_e,\Lambda\rangle$.
\end{rem}

\begin{lem}\label{lem:LS-multipartition-dimension}\cite[Theorem~1.1]{LS}
Let $n\ge0$, $\alpha\in Q_n^+(I)$, and $\Lambda\in P^+$.
We have the $\mathbb C$-dimension formula
\begin{equation}\label{eq:LS-multipartition-dimension}
 \dim_{\mathbb C}RC^\Lambda_\alpha
 =\sum_{\substack{\boldsymbol\lambda\in\mathscr P_n^{\bullet,m}\\
                   \alpha_{\boldsymbol\lambda}=\alpha}}
  2^{\,2n-\sum_{r\in\mathcal C_\bullet\setminus[m]}
                  \ell(\lambda^{(r)})}
       (\#\operatorname{Std}(\boldsymbol\lambda))^2.
\end{equation}
\end{lem}
\begin{proof}
Suppose first that $I$ is of affine type. Following the parameter
correspondence in \cite[\S8]{Shi:semisimplicity}, work over $\mathbb C$ and choose $q$ and
$u$ so that $i\mapsto\mathtt q(uq^{2i})$ is the labelling of
Definition~\ref{def:KKT-spectral-labelling}. In the cases where $\sigma$
occurs, choose the representative with $\sigma=1$. Set
\begin{equation*}
 Q_r=uq^{2i_r}\ (r\in[m]),\qquad
 Q_0=Q_{0_+}=1,\quad Q_{0_-}=-1.
\end{equation*}
In type $D_{e+1}^{(2)}$, use Lemma~\ref{lem:D-end-reflection} when the
odd endpoint is $e$. Equations~\eqref{eq:KKT-weight},
\eqref{eq:source-label-multiplicities}, and \eqref{eq:residue-formulas} give
\begin{equation}\label{eq:folding-spectral-identity}
 \mathtt q\bigl(uq^{2\operatorname{res}_I(A)}\bigr)
 =\mathtt q\bigl(Q_rq^{2(b-a)}\bigr)
 \qquad(A=(a,b,r)).
\end{equation}
The dimension formula of \cite[Theorem~1.1 and Definitions~6.5--6.8]{LS}
gives \eqref{eq:LS-multipartition-dimension} over $\mathbb C$.

For type $A_\infty$, $B_\infty$, or $C_\infty$, choose a finite full subdiagram containing
$\operatorname{supp}\alpha$, all vertices $i$ with $\langle h_i,\Lambda\rangle\ne0$,
and the supports of all $\alpha_{\boldsymbol\lambda}$ with
$|\boldsymbol\lambda|=n$, together with their adjacent vertices.
Use an affine diagram of type $A_e^{(1)}$,
$D_{e+1}^{(2)}$, or $C_e^{(1)}$ for $A_\infty$, $B_\infty$, or
$C_\infty$, respectively, with $e$ large enough that the chosen subdiagram
is a full subdiagram of both, after translating the labels in type
$A_\infty$ if necessary. The Cartan entries and the residues in
\eqref{eq:residue-formulas} agree on this finite range. The restriction
argument of Lemma~\ref{lem:finite-support-restriction} reduces the
assertion to the affine case.
\end{proof}

\begin{prop}\label{prop:multipartition-support}
Let $n\ge0$, $\alpha\in Q_n^+(I)$, and $\Lambda\in P^+$.
The following conditions are equivalent over $\mathbb K$:
\begin{enumerate}[label=\textup{(\roman*)}]
\item $R^\Lambda_\alpha\ne0$;
\item $\alpha=\alpha_{\boldsymbol\lambda}$ for some
$\boldsymbol\lambda\in\mathscr P_n^{\bullet,m}$.
\end{enumerate}
\end{prop}
\begin{proof}
It follows from \eqref{eq:LS-multipartition-dimension} that, over $\mathbb C$,
\begin{equation}\label{eq:cor-LS-multipartition-dimension-C}
 \begin{aligned}
 RC^\Lambda_\alpha(\mathbb C)\ne0
 \Longleftrightarrow
 \exists\,\boldsymbol\lambda\in\mathscr P_n^{\bullet,m}
 \text{ with }\alpha_{\boldsymbol\lambda}=\alpha.
 \end{aligned}
\end{equation}
By Theorem \ref{supermorita} (1) and \eqref{eq:KKT-def3.10}, over $\mathbb C$,
\begin{equation}\label{eq:cor-KKT-morita-C}
 \begin{aligned}
 RC^\Lambda_\alpha(\mathbb C)\ne0
 \Longleftrightarrow
 R^\Lambda_\alpha(\mathbb C)\ne0.
 \end{aligned}
\end{equation}
The dimension of $R^\Lambda_\alpha$ is independent of the base field by \cite[Theorem~1.2 and Corollary~3.15]{HS},
\begin{equation}\label{eq:HS-dimension-comparison}
 \dim_{\mathbb K}R^\Lambda_\alpha(\mathbb K)
 =\dim_{\mathbb C}R^\Lambda_\alpha(\mathbb C).
\end{equation}
Combining \eqref{eq:cor-LS-multipartition-dimension-C}, \eqref{eq:cor-KKT-morita-C}, and \eqref{eq:HS-dimension-comparison} above,
we obtain the Proposition.
\end{proof}


\section{Combinatorial conditions for zero defect}\label{sec:single-orbit}
\noindent\textbf{Throughout this section, fix a connected Cartan
superdatum $(\mathrm A,P,\Pi,\Pi^\vee)$ of a type in
Section~\ref{DynkinDiagrams}, with index set $I$, and an arbitrary
$\Lambda\in P^+$ (of finite support if $I$ is infinite).}
Recall that $m$, $\bullet$, and the labels $i_r\in I$
($r\in\mathcal C_\bullet$) satisfying
\eqref{eq:source-label-multiplicities} have been fixed.

\begin{defn}\label{def:local-condition}
For $n\ge0$, we say that $\mathtt{S}_\Lambda(n)$ holds if one of the following
conditions is satisfied:
\begin{enumerate}[label=\textup{(\roman*)}]
\item $n=0$ or $\mathcal C_\bullet=\emptyset$.
\item\label{item:condition-ordinary} $n\ge1$ and $m>0$. The bounds for
the type of $I$ listed below hold for all $1\le r\le m$ and all
$1\le r<s\le m$. Any permitted strict components may also occur.
\begin{enumerate}[label=\textup{(\alph*)}]
\item Type $A_e^{(1)}$ ($e\ge2$):
\[
 0\le i_r\le e,\qquad n\le e,\qquad
 n\le|i_r-i_s|\le e+1-n.
\]
\item Type $C_e^{(1)}$ ($e\ge2$):
\[
 \frac{n-1}{2}\le i_r\le e-\frac{n-1}{2},\qquad |i_r-i_s|\ge n.
\]
\item Type $A_{2e}^{(2)}$ ($e\ge1$):
\[
 n\le i_r\le e-\frac{n-1}{2},\qquad |i_r-i_s|\ge n.
\]
\item Type $D_{e+1}^{(2)}$ ($e\ge1$):
\[
 n\le i_r\le e-n,\qquad |i_r-i_s|\ge n.
\]
\item Type $A_\infty$:
\[
 i_r\in\mathbb Z,\qquad |i_r-i_s|\ge n.
\]
\item Type $B_\infty$:
\[
 i_r\ge n,\qquad |i_r-i_s|\ge n.
\]
\item Type $C_\infty$:
\[
 i_r\ge\frac{n-1}{2},\qquad |i_r-i_s|\ge n.
\]
\end{enumerate}
\item\label{item:condition-strict} $n\ge1$, $m=0$, and
$\mathcal C_\bullet\ne\emptyset$. One of the following holds:
\begin{enumerate}[label=\textup{(\alph*)}]
\item $I$ is of type $B_\infty$, with no restriction on $n$;
\item $I$ is of type $A_{2e}^{(2)}$ ($e\ge1$) and $n\le2e$;
\item $I$ is of type $D_{e+1}^{(2)}$ ($e\ge1$) and $n\le e$.
\end{enumerate}
\end{enumerate}
\end{defn}

The following theorem is the main result of this section.
\Needspace{5\baselineskip}
\begin{thm}\label{thm:combinatorial-zero-defect}
If $\Lambda\ne0$, then for every $n\ge0$ the following conditions are equivalent:
\begin{enumerate}[label=\textup{(\roman*)}]
\item $\mathtt{S}_\Lambda(n)$ holds;
\item $\operatorname{def}(\Lambda,\alpha_{\boldsymbol\lambda})=0$ for every $\boldsymbol\lambda\in\mathscr P_n^{\bullet,m}$;
\item $R_n^\Lambda(I)$ is semisimple.
\end{enumerate}
\end{thm}

For $A_e^{(1)}$, $C_e^{(1)}$, and $A_{2e}^{(2)}$, the semisimplicity
criteria are given in
\cite[Corollaries~8.3, 8.7 and~8.11]{Shi:semisimplicity}.
For the corresponding combinatorial calculations, see
Appendix~\ref{app:other-affine-types}.
\textbf{Henceforth, we assume that
$I=\{0,\ldots,e\}$ is of twisted affine type $D_{e+1}^{(2)}$, with $e\ge1$.}

\subsection{Partitions with positive defect}
Equation~\eqref{eq:source-label-multiplicities} gives
\begin{equation}\label{eq:source-component-weight}
 (\Lambda\mid\gamma)=
 \sum_{r=1}^{m}\bigl(2^{\delta_{{\rm p}(i_r),\bar1}}\Lambda_{i_r}\mid\gamma\bigr)
 +\sum_{r\in\mathcal C_\bullet\setminus[m]}(\Lambda_{i_r}\mid\gamma)
 \qquad(\gamma\in Q).
\end{equation}
The norm calculation uses the following form of the Cartan pairing.
\begin{lem}\label{lem:local-norm-formulas}
For $\gamma=\sum_{i=0}^e c_i\nu_i\in Q$,
\begin{equation}\label{eq:local-root-norms}
 (\gamma\mid\gamma)=2\sum_{i=0}^{e-1}(c_i-c_{i+1})^2.
\end{equation}
For $i\in I$,
\begin{equation}\label{eq:fundamental-weight-pairings}
 (\Lambda_i\mid\gamma)=
 \begin{cases}c_i,&i\in\{0,e\},\\2c_i,&1\le i\le e-1.
 \end{cases}
\end{equation}
\end{lem}
\begin{proof}
Expand the symmetrized Cartan matrix of $D_{e+1}^{(2)}$ to obtain
\eqref{eq:local-root-norms}. Since $(\Lambda_i\mid\nu_j)=\rd_i\delta_{ij}$,
with $\rd_0=\rd_e=1$ and $\rd_i=2$ for $0<i<e$,
\eqref{eq:fundamental-weight-pairings} follows.
\end{proof}

Let $\delta=\sum_{j=0}^e\nu_j \in Q^+$ be the minimal imaginary root.
\begin{lem}\label{lem:D-weight-and-null-root}
\begin{enumerate}[label=\textup{(\arabic*)}]
\item For $\gamma=\sum_{j=0}^e c_j\nu_j\in Q^+$,
\[
 (\Lambda\mid\gamma)=2\sum_{r=1}^m c_{i_r}
 +\sum_{r\in\mathcal C_\bullet\setminus[m]}c_{i_r}.
\]
All summands are nonnegative.
\item For every $\eta\in Q$,
\[
 (\delta\mid\eta)=0,\qquad
 (\Lambda_0\mid\delta)=(\Lambda_e\mid\delta)=1.
\]
\end{enumerate}
\end{lem}
\begin{proof}
Part~\textup{(1)} follows by substituting
\eqref{eq:fundamental-weight-pairings} into \eqref{eq:source-component-weight}.
For \textup{(2)}, apply the bilinear identity obtained from
\eqref{eq:local-root-norms}, then use \eqref{eq:fundamental-weight-pairings}.
\end{proof}

\begin{lem}\label{lem:D-prefix-data}
For $0\le c<2e+2$, the following identities hold, with empty sums equal to zero.
\begin{enumerate}[label=\textup{(\arabic*)}]
\item
\[
 \gamma_I[0,c-1]=
 \begin{cases}
 \displaystyle\sum_{j=0}^{c-1}\nu_j,&0\le c\le e+1,\\[3pt]
 \displaystyle\delta+\sum_{j=2e+2-c}^{e}\nu_j,&e+2\le c\le2e+1.
 \end{cases}
\]
\item
\[
 (\gamma_I[0,c-1]\mid\gamma_I[0,c-1])=
 \begin{cases}0,&c=0\text{ or }c=e+1,\\2,&\text{otherwise}.
 \end{cases}
\]
\item
\[
 (\Lambda_0\mid\gamma_I[0,c-1])=
 \begin{cases}0,&c=0,\\1,&1\le c<2e+2.
 \end{cases}
\]
\end{enumerate}
\end{lem}
\begin{proof}
The $D_{e+1}^{(2)}$ formula in \eqref{eq:folding-map} and
\eqref{eq:interval-root} give \textup{(1)}.
The coefficients in \textup{(1)} are constant for $c=0,e+1$;
otherwise exactly one consecutive difference is $1$ or $-1$.
Equation~\eqref{eq:local-root-norms} gives \textup{(2)}, and
\eqref{eq:fundamental-weight-pairings} gives \textup{(3)}.
\end{proof}

We first construct a multipartition with positive defect whenever a bound fails.
\begin{lem}\label{lem:positive-defect-partition}
If $\mathtt{S}_\Lambda(n)$ fails and $\mathcal C_\bullet\ne\emptyset$, there is
$\boldsymbol\lambda\in\mathscr P_n^{\bullet,m}$ such that
$\operatorname{def}(\Lambda,\alpha_{\boldsymbol\lambda})>0$.
\end{lem}
\begin{proof}
Since $\mathtt{S}_\Lambda(n)$ fails, $n\ge1$. In each construction below, all
unspecified components are empty and zero parts are omitted. Write
$\alpha_{\boldsymbol\lambda}=\sum_{j=0}^e c_j\nu_j$.

\medskip\noindent\textbf{Case 1: $m>0$.}
If an individual bound fails, fix $r\in[m]$ such that
$i_r<n$ or $i_r>e-n$. The first two subcases concern this fixed $r$.

\smallskip\noindent\textbf{Subcase 1.1: $i_r<n$.}
Take
\[
 \lambda^{(r)}=(n-i_r,1^{i_r}).
\]
This is a partition of $n$ since $n-i_r\ge1$.
Equations~\eqref{eq:residue-formulas}--\eqref{eq:interval-root} give the
contributions of its first column below the top box and its first row,
and hence the following expression for $\alpha_{\boldsymbol\lambda}$:
\begin{equation}\label{eq:hook-root-boundaries}
 \alpha_{\boldsymbol\lambda}
 =\sum_{t=0}^{i_r-1}\nu_{\mathfrak f_I(t)}
  +\sum_{t=i_r}^{n-1}\nu_{\mathfrak f_I(t)}
 =\gamma_I[0,n-1].
\end{equation}
Since $\mathfrak f_I(i_r)=i_r$ by \eqref{eq:folding-map}, we have
$c_{i_r}\ge1$. Keeping the $r$th term in
Lemma~\ref{lem:D-weight-and-null-root}\textup{(1)} gives
\begin{equation}\label{eq:D-left-hook-pairing}
 (\Lambda\mid\alpha_{\boldsymbol\lambda})\ge2c_{i_r}\ge2.
\end{equation}

Write $n=k(2e+2)+c$, where $k\ge0$ and $0\le c<2e+2$.
By \eqref{eq:folding-map},
$\mathfrak f_I(t+2e+2)=\mathfrak f_I(t)$ and each vertex occurs twice
among $\mathfrak f_I(0),\ldots,\mathfrak f_I(2e+1)$.
Splitting the sum in \eqref{eq:interval-root} into $k$ such intervals
and the remaining $c$ terms gives
\begin{equation}\label{eq:D-period-decomposition}
 \gamma_I[0,n-1]=2k\delta+\gamma_I[0,c-1].
\end{equation}
Consequently, \eqref{eq:hook-root-boundaries} and
\eqref{eq:D-period-decomposition} give $\alpha_{\boldsymbol\lambda}=2k\delta+\gamma_I[0,c-1]$. Lemma~\ref{lem:D-weight-and-null-root}\textup{(2)}
removes the $\delta$ terms from its norm, and
Lemma~\ref{lem:D-prefix-data}\textup{(2)} bounds the remaining term:
\begin{equation}\label{eq:D-left-hook-norm}
 (\alpha_{\boldsymbol\lambda}\mid\alpha_{\boldsymbol\lambda})
 =(\gamma_I[0,c-1]\mid\gamma_I[0,c-1])\le2.
\end{equation}
Using \eqref{eq:D-left-hook-pairing} and \eqref{eq:D-left-hook-norm} in
\eqref{eq:defect-normalization}, we obtain
\begin{equation*}
 \operatorname{def}(\Lambda,\alpha_{\boldsymbol\lambda})
 \ge2-\tfrac12\cdot2=1.
\end{equation*}

\smallskip\noindent\textbf{Subcase 1.2: $i_r\ge n$ and $i_r>e-n$.}
Take
\[
 \lambda^{(r)}=(e-i_r+1,1^{n-e+i_r-1}).
\]
The inequalities $i_r\le e$ and $i_r>e-n$ give
$e-i_r+1\ge1$ and $n-e+i_r-1\ge0$, so this is a partition of $n$.
Moreover, $n\le i_r\le e$ and $i_r>e-n$ imply
$1\le e+1-n\le i_r\le e$. Thus
$\mathfrak f_I(t)=t$ for $e+1-n\le t\le e$ by
\eqref{eq:folding-map}.
Equations~\eqref{eq:residue-formulas}--\eqref{eq:interval-root} give the
contributions of the first column below the top box and the first row,
and hence the following expression for $\alpha_{\boldsymbol\lambda}$:
\begin{equation*}
 \begin{aligned}
 \alpha_{\boldsymbol\lambda}
 &=\sum_{t=e+1-n}^{i_r-1}\nu_{\mathfrak f_I(t)}
   +\sum_{t=i_r}^{e}\nu_{\mathfrak f_I(t)}\\
 &=\gamma_I[e+1-n,e]
  =\sum_{j=e+1-n}^{e}\nu_j.
 \end{aligned}
\end{equation*}
Hence $c_{i_r}=1$, and the only nonzero consecutive coefficient
difference is
\[
 c_{e-n}-c_{e+1-n}=-1.
\]
Lemma~\ref{lem:D-weight-and-null-root}\textup{(1)} and
\eqref{eq:local-root-norms} therefore give, respectively,
\begin{equation}\label{eq:D-right-hook-pairing}
 (\Lambda\mid\alpha_{\boldsymbol\lambda})\ge2c_{i_r}=2,
\end{equation}
\begin{equation}\label{eq:D-right-hook-norm}
 (\alpha_{\boldsymbol\lambda}\mid\alpha_{\boldsymbol\lambda})
 =2(0-1)^2=2.
\end{equation}
Substituting \eqref{eq:D-right-hook-pairing}--\eqref{eq:D-right-hook-norm}
into \eqref{eq:defect-normalization} gives
\begin{equation*}
 \operatorname{def}(\Lambda,\alpha_{\boldsymbol\lambda})
 \ge2-\tfrac12\cdot2=1.
\end{equation*}

\enlargethispage{3\baselineskip}
\smallskip\noindent\textbf{Subcase 1.3: $n\le i_r\le e-n$ for all
$r\in[m]$.}
Since $\mathtt{S}_\Lambda(n)$ fails, there exist distinct $s_1,s_2\in[m]$ with
$|i_{s_1}-i_{s_2}|<n$. Order them so that
$i_{s_1}\le i_{s_2}$, and take
\[
 \lambda^{(s_1)}=(n).
\]
The assumptions give
\[
 1\le i_{s_1}\le i_{s_2}\le i_{s_1}+n-1\le e-1.
\]
Therefore $\mathfrak f_I(t)=t$ for
$i_{s_1}\le t\le i_{s_1}+n-1$ by \eqref{eq:folding-map}, and
\eqref{eq:residue-formulas}--\eqref{eq:multipartition-root} give
\begin{equation*}
 \alpha_{\boldsymbol\lambda}
 =\sum_{b=1}^n\nu_{\mathfrak f_I(i_{s_1}+b-1)}
 =\sum_{t=i_{s_1}}^{i_{s_1}+n-1}\nu_t.
\end{equation*}
The only nonzero consecutive coefficient differences are
\[
 c_{i_{s_1}-1}-c_{i_{s_1}}=-1,\qquad
 c_{i_{s_1}+n-1}-c_{i_{s_1}+n}=1.
\]
Substituting these differences into \eqref{eq:local-root-norms} gives
\begin{equation}\label{eq:D-pair-row-norm}
 (\alpha_{\boldsymbol\lambda}\mid\alpha_{\boldsymbol\lambda})
 =2\bigl((0-1)^2+(1-0)^2\bigr)=4.
\end{equation}
Both $c_{i_{s_1}}$ and $c_{i_{s_2}}$ equal $1$. Since $s_1\ne s_2$,
the two positions contribute separate terms in
Lemma~\ref{lem:D-weight-and-null-root}\textup{(1)}, even if $i_{s_1}=i_{s_2}$. Hence
\begin{equation}\label{eq:D-pair-row-pairing}
 (\Lambda\mid\alpha_{\boldsymbol\lambda})
 \ge2c_{i_{s_1}}+2c_{i_{s_2}}=4.
\end{equation}
Substituting \eqref{eq:D-pair-row-norm}--\eqref{eq:D-pair-row-pairing}
into \eqref{eq:defect-normalization} yields
\begin{equation*}
 \operatorname{def}(\Lambda,\alpha_{\boldsymbol\lambda})
 \ge4-\tfrac12\cdot4=2>0.
\end{equation*}
These three subcases exhaust the failure of $\mathtt{S}_\Lambda(n)$ when $m>0$.

\medskip\noindent\textbf{Case 2: $m=0$.}
Since $\mathcal C_\bullet\ne\emptyset$, there is a strict component;
the failure of $\mathtt{S}_\Lambda(n)$ gives $n>e$.
The reflection in Lemma~\ref{lem:D-end-reflection} interchanges $0$ and
$e$. Reversing the coefficients in
\eqref{eq:local-root-norms}--\eqref{eq:fundamental-weight-pairings}
preserves the norm and the
weight pairing, hence also the defect in
\eqref{eq:defect-normalization}. We may therefore take the strict
component at $0$; denote its position by $r$.

\smallskip\noindent\textbf{Subcase 2.1: $e<n\le2e+1$.}
Take
\[
 \lambda^{(r)}=(e+1,n-e-1).
\]
Since $0\le n-e-1\le e$, this is a strict partition of $n$ after
omitting a possible zero part. We have $\mathfrak f_I(t)=t$ on both
summation ranges by \eqref{eq:folding-map}. Thus
\eqref{eq:residue-formulas}--\eqref{eq:interval-root} give
\begin{equation}\label{eq:D-affine-two-row-root}
 \alpha_{\boldsymbol\lambda}
 =\sum_{t=0}^{e}\nu_t+\sum_{t=0}^{n-e-2}\nu_t
 =\delta+\gamma_I[0,n-e-2].
\end{equation}
By Lemma~\ref{lem:D-weight-and-null-root}\textup{(2)}, its norm equals that of the second
summand. Taking $c=n-e-1\in[0,e]$ in
Lemma~\ref{lem:D-prefix-data}\textup{(2)} gives
\begin{equation}\label{eq:D-two-row-norm}
 (\alpha_{\boldsymbol\lambda}\mid\alpha_{\boldsymbol\lambda})
 =\begin{cases}0,&n=e+1,\\2,&e+2\le n\le2e+1.
 \end{cases}
\end{equation}
The coefficient of $\nu_0$ in \eqref{eq:D-affine-two-row-root}
is $1$ when the second row is empty and $2$ otherwise. Keeping the
$r$th strict term in Lemma~\ref{lem:D-weight-and-null-root}\textup{(1)} gives
\begin{equation}\label{eq:D-two-row-pairing}
 (\Lambda\mid\alpha_{\boldsymbol\lambda})\ge c_0
 =\begin{cases}1,&n=e+1,\\2,&e+2\le n\le2e+1.
 \end{cases}
\end{equation}
Substituting \eqref{eq:D-two-row-norm}--\eqref{eq:D-two-row-pairing}
into \eqref{eq:defect-normalization} yields
\begin{equation*}
 \operatorname{def}(\Lambda,\alpha_{\boldsymbol\lambda})
 \ge\begin{cases}1-\tfrac12\cdot0=1,&n=e+1,\\
                   2-\tfrac12\cdot2=1,&e+2\le n\le2e+1.
     \end{cases}
\end{equation*}

\smallskip\noindent\textbf{Subcase 2.2: $n\ge2e+2$.}
Write $n=k(2e+2)+c$ with $k\ge1$ and $0\le c<2e+2$, and take
\[
 \lambda^{(r)}=(n).
\]
By \eqref{eq:residue-formulas}--\eqref{eq:interval-root} and the
period decomposition \eqref{eq:D-period-decomposition},
\begin{equation*}
 \alpha_{\boldsymbol\lambda}
 =\gamma_I[0,n-1]=2k\delta+\gamma_I[0,c-1].
\end{equation*}
The $\delta$ terms have zero pairing with every element of $Q$ by
Lemma~\ref{lem:D-weight-and-null-root}\textup{(2)}. Thus
Lemma~\ref{lem:D-prefix-data}\textup{(2)} gives
\begin{equation}\label{eq:D-long-strict-row-norm}
 (\alpha_{\boldsymbol\lambda}\mid\alpha_{\boldsymbol\lambda})
 =\begin{cases}0,&c=0\text{ or }c=e+1,\\2,&\text{otherwise}.
 \end{cases}
\end{equation}
Keeping the $r$th strict term in Lemma~\ref{lem:D-weight-and-null-root}\textup{(1)}, and
using Lemma~\ref{lem:D-weight-and-null-root}\textup{(2)} and
Lemma~\ref{lem:D-prefix-data}\textup{(3)} to calculate it, gives
\begin{equation}\label{eq:D-long-strict-row-pairing}
 (\Lambda\mid\alpha_{\boldsymbol\lambda})
 \ge(\Lambda_0\mid\alpha_{\boldsymbol\lambda})
 =\begin{cases}2k,&c=0,\\2k+1,&1\le c<2e+2.
 \end{cases}
\end{equation}
Substituting \eqref{eq:D-long-strict-row-norm}--\eqref{eq:D-long-strict-row-pairing} into
\eqref{eq:defect-normalization} gives
\begin{equation*}
 \operatorname{def}(\Lambda,\alpha_{\boldsymbol\lambda})
 \ge\begin{cases}
      2k,&c=0,\\
      2k+1,&c=e+1,\\
      2k,&c\notin\{0,e+1\}
     \end{cases}
>0.
\end{equation*}
This completes the proof.
\end{proof}

\subsection{Diagonal hooks and zero defect}

The element $\alpha_\lambda$ can be computed from the diagonal hooks of $\lambda$.
\begin{defn}\label{def:diagonal-hook-intervals}
Let $\boldsymbol\lambda$ be an allowed multipartition and let $\lambda$
be an ordinary component of size $n$, at label $i$.  If $s$ is the number of diagonal boxes of $\lambda$, define
\begin{equation*}
a_t=\lambda'_t-t,\qquad b_t=\lambda_t-t,
 \qquad J_t=[i-a_t,i+b_t]\quad(1\le t\le s).
\end{equation*}
Here $\lambda'$ is the conjugate partition, and $a_t$ and $b_t$ are the
leg and arm lengths at its $t$th diagonal box.
\end{defn}

\begin{example}\label{ex:diagonal-hook}
\leavevmode\par\vspace{-.35\baselineskip}\begin{center}
\begin{tikzpicture}[x=.75cm,y=.75cm,
 every node/.style={font=\normalsize,inner sep=2pt}]
 \fill[black!12] (1,2)--(4,2)--(4,1)--(2,1)--(2,0)--(1,0)--cycle;
 \foreach \x in {0,1,2,3,4}{\draw (\x,2) rectangle ++(1,1);}\foreach \x in {0,1,2,3}{\draw (\x,1) rectangle ++(1,1);}\foreach \x in {0,1}{\draw (\x,0) rectangle ++(1,1);}
 \draw[thick] (1,2)--(4,2)--(4,1)--(2,1)--(2,0)--(1,0)--cycle;
 \node at (1.5,1.5){$i$};\node at (2.5,1.5){$i+1$};
 \node at (3.5,1.5){$i+2$};\node at (1.5,.5){$i-1$};
 \node[above] at (2.5,3.25){$\lambda=(5,4,2)$};
 \node at (2.5,-.6){$a_2=1,\quad b_2=2$};
 \begin{scope}[xshift=5.4cm,x=.75cm]
  \draw[thick] (0,2.2)--(6,2.2);
  \draw[thick] (1,1)--(4,1);
  \foreach \x in {0,6}{\draw (\x,2.07)--(\x,2.33);}
  \foreach \x in {1,4}{\draw (\x,.87)--(\x,1.13);}
  \node[above] at (0,2.33){$i-2$};\node[above] at (6,2.33){$i+4$};
  \node[above] at (1,1.13){$i-1$};\node[above] at (4,1.13){$i+2$};
  \node[right] at (6.3,2.2){$J_1$};\node[right] at (4.3,1){$J_2$};
  \draw[densely dotted] (2,.55)--(2,2.2);\node[below] at (2,.55){$i$};
 \end{scope}
\end{tikzpicture}
\end{center}
For $\lambda=(5,4,2)$ at label $i$, the shaded diagonal hook at $(2,2)$
has leg length $a_2=1$, arm length $b_2=2$, and interval
$J_2=[i-1,i+2]$. The other diagonal hook has interval
$J_1=[i-2,i+4]$, so $J_2\subsetneq J_1$, as shown on the right.
\end{example}

The next lemma computes $\alpha_\lambda$ from diagonal hooks and $\alpha_\sigma$ from rows.
\begin{lem}\label{lem:diagonal-hook-intervals}
\leavevmode
\begin{enumerate}[label=\textup{(\arabic*)}]
\item For an ordinary component $\lambda$ of size $n$ at label $i$,
with $s$ diagonal hooks and the intervals $J_t$ of
Definition~\ref{def:diagonal-hook-intervals},
\[
 \alpha_\lambda=\sum_{t=1}^s\gamma_I(J_t),\qquad
 \sum_{t=1}^s|J_t|=n.
\]
Moreover,
\[
 i-a_t-1<i-a_{t+1}-1<i+b_{t+1}<i+b_t
 \qquad(1\le t<s).
\]
\item For a strict component $\sigma=(\sigma_1>\cdots>\sigma_l>0)$,
\[
 \alpha_\sigma=
 \begin{cases}
 \displaystyle\sum_{t=1}^l\gamma_I[0,\sigma_t-1],&
       \text{if $\sigma$ is at $0$},\\[3pt]
 \displaystyle\sum_{t=1}^l\gamma_I[e-\sigma_t+1,e],&
       \text{if $\sigma$ is at $e$},
 \end{cases}
 \qquad |\sigma|=\sum_{t=1}^l\sigma_t.
\]
\end{enumerate}
\end{lem}
\begin{proof}
\textup{(1)} The diagonal hook at $(t,t)$ consists of
\[
 \{(t,b):t\le b\le\lambda_t\}
 \sqcup\{(a,t):t<a\le\lambda'_t\},
\]
whose residues are determined by $i-a_t,i-a_t+1,\ldots,i+b_t$.
These hooks partition the diagram, so
Equations~\eqref{eq:multipartition-root} and~\eqref{eq:interval-root}
give the formulas for $\alpha_\lambda$ and $|\lambda|$. Also,
\[
 a_t-a_{t+1}=\lambda'_t-\lambda'_{t+1}+1\ge1,\qquad
 b_t-b_{t+1}=\lambda_t-\lambda_{t+1}+1\ge1,
\]
which gives the stated inequalities since $a_{t+1},b_{t+1}\ge0$.

\textup{(2)} At the strict position $0_+$, row $t$ has residues
$0,\ldots,\sigma_t-1$.
At $e$, its residues are those of
$e-\sigma_t+1,\ldots,e$, by
\eqref{eq:residue-formulas} and \eqref{eq:D-folding-reflection}.
Summing over the rows and using
Equations~\eqref{eq:multipartition-root} and~\eqref{eq:interval-root}
gives the formulas.
\end{proof}

We compute the bilinear pairings using consecutive coefficient differences.
For $\gamma=\sum_{j=0}^e c_j\nu_j\in Q$, put
$\partial(\gamma)=(c_j-c_{j+1})_{0\le j<e}$, and let
$\varepsilon_0,\ldots,\varepsilon_{e-1}$ be the standard coordinate
vectors in $\mathbb Z^e$.
\begin{lem}\label{lem:D-coefficient-differences}
\leavevmode
\begin{enumerate}[label=\textup{(\arabic*)}]
\item For $\gamma,\gamma'\in Q$,
\[
 (\gamma\mid\gamma')=2\partial(\gamma)\cdot \partial(\gamma').
\]
\item For $1\le L\le R\le e-1$,
\[
 \partial(\gamma_I[L,R])=\varepsilon_R-\varepsilon_{L-1}.
\]
\end{enumerate}
\end{lem}
\begin{proof}
\textup{(1)} Substituting $\gamma+\gamma'$ into
\eqref{eq:local-root-norms} yields the stated equality.

\textup{(2)} By \eqref{eq:folding-map} and \eqref{eq:interval-root},
$\gamma_I[L,R]=\sum_{j=L}^R\nu_j$.
Its only nonzero consecutive coefficient differences are $-1$ at
$L-1$ and $1$ at $R$.
\end{proof}

The next lemma computes the contribution of each diagonal hook or
strict row to the defect, and shows that the cross terms between
distinct hooks or rows in the same component vanish.
\begin{lem}\label{lem:hook-norm-and-orthogonality}
Let $n\ge1$, assume that $\mathtt{S}_\Lambda(n)$ holds, and fix
$\boldsymbol\lambda\in\mathscr P_n^{\bullet,m}$.
\begin{enumerate}[label=\textup{(\arabic*)}]
\item For a diagonal hook interval $J=[i-a,i+b]$ in an ordinary
component of $\boldsymbol\lambda$ at label $i$,
\[
 (\gamma_I(J)\mid\gamma_I(J))=4,\qquad
 (\Lambda_i\mid\gamma_I(J))=2.
\]
\item Let $l$ be the length of a row of a strict component of
$\boldsymbol\lambda$. If the component is at $0$, then
\[
 (\gamma_I[0,l-1]\mid\gamma_I[0,l-1])=2,\qquad
 (\Lambda_0\mid\gamma_I[0,l-1])=1.
\]
If the component is at $e$, then
\[
 (\gamma_I[e-l+1,e]\mid\gamma_I[e-l+1,e])=2,\qquad
 (\Lambda_e\mid\gamma_I[e-l+1,e])=1.
\]
\item Let $J_1,\ldots,J_s$ be the diagonal hook intervals of one
ordinary component of $\boldsymbol\lambda$. Then
\[
 (\gamma_I(J_t)\mid\gamma_I(J_u))=0
 \qquad(1\le t\ne u\le s).
\]
\item Let $\sigma=(\sigma_1>\cdots>\sigma_l>0)$ be one strict
component of $\boldsymbol\lambda$, and let $1\le t\ne u\le l$.
If the component is at $0$, then
\[
 (\gamma_I[0,\sigma_t-1]\mid\gamma_I[0,\sigma_u-1])=0.
\]
If the component is at $e$, then
\[
 (\gamma_I[e-\sigma_t+1,e]\mid\gamma_I[e-\sigma_u+1,e])=0.
\]
\end{enumerate}
\end{lem}
\begin{proof}
\textup{(1)} Definition~\ref{def:local-condition}\textup{(ii)(d)} gives
$n\le i\le e-n$. Since $a,b\le n-1$, we have
$1\le i-a\le i+b\le e-1$.
Lemma~\ref{lem:D-coefficient-differences}\textup{(1), (2)}
gives $(\gamma_I(J)\mid\gamma_I(J))=2(1^2+(-1)^2)=4$.
The coefficient of $\nu_i$ in $\gamma_I(J)$ is $1$, so
\eqref{eq:fundamental-weight-pairings} gives
$(\Lambda_i\mid\gamma_I(J))=2$.

\textup{(2)} Definition~\ref{def:local-condition}\textup{(ii)(d), (iii)(c)}
gives $n\le e$, so $1\le l\le e$.
The two row intervals have consecutive coefficient differences
\begin{equation}\label{eq:D-strict-row-difference}
 \partial(\gamma_I[0,l-1])=\varepsilon_{l-1},\qquad
 \partial(\gamma_I[e-l+1,e])=-\varepsilon_{e-l}.
\end{equation}
Applying Lemma~\ref{lem:D-coefficient-differences}\textup{(1)} to \eqref{eq:D-strict-row-difference} gives norm $2$ in
each case. Each interval contains its corresponding endpoint with
coefficient $1$, so \eqref{eq:fundamental-weight-pairings} gives
both weight pairings.

\textup{(3)} For $t<u$, Lemma~\ref{lem:diagonal-hook-intervals}\textup{(1)} gives
\[
 i-a_t-1<i-a_u-1<i+b_u<i+b_t.
\]
By Lemma~\ref{lem:D-coefficient-differences}\textup{(2)},
\[
 \partial(\gamma_I(J_t))=\varepsilon_{i+b_t}-\varepsilon_{i-a_t-1},\qquad
 \partial(\gamma_I(J_u))=\varepsilon_{i+b_u}-\varepsilon_{i-a_u-1}.
\]
The four coordinate indices are distinct by Lemma~\ref{lem:diagonal-hook-intervals}\textup{(1)}. Lemma~\ref{lem:D-coefficient-differences}\textup{(1)}
therefore gives
\[
 (\gamma_I(J_t)\mid\gamma_I(J_u))
 =2(\varepsilon_{i+b_t}-\varepsilon_{i-a_t-1})\cdot
   (\varepsilon_{i+b_u}-\varepsilon_{i-a_u-1})=0.
\]

\textup{(4)} Since $t\ne u$, strictness gives $\sigma_t\ne\sigma_u$.
Applying \eqref{eq:D-strict-row-difference} and
Lemma~\ref{lem:D-coefficient-differences}\textup{(1)} gives, at~$0$,
\[
 (\gamma_I[0,\sigma_t-1]\mid\gamma_I[0,\sigma_u-1])
 =2\varepsilon_{\sigma_t-1}\cdot\varepsilon_{\sigma_u-1}=0,
\]
and, at $e$,
\[
 (\gamma_I[e-\sigma_t+1,e]\mid\gamma_I[e-\sigma_u+1,e])
 =2(-\varepsilon_{e-\sigma_t})\cdot(-\varepsilon_{e-\sigma_u})=0.
\]
\end{proof}

We now show that the cross terms between different components also
make zero contribution to the defect.
\begin{lem}\label{lem:cross-terms-vanish}
Fix $\boldsymbol\lambda\in\mathscr P_n^{\bullet,m}$.
If $\mathtt{S}_\Lambda(n)$ holds, then, for distinct
$r,s\in\mathcal C_\bullet$,
\begin{equation*}
 (\Lambda_{i_r}\mid\alpha_{\lambda^{(s)}})=0,\qquad (\alpha_{\lambda^{(r)}}\mid\alpha_{\lambda^{(s)}})=0.
\end{equation*}
\end{lem}
\begin{proof}
If $n=0$, then $\alpha_{\boldsymbol\lambda}=0$. Assume $n\ge1$.
Definition~\ref{def:local-condition}\textup{(ii)(d), (iii)(c)} gives
$n\le e$. For each nonempty ordinary component $\lambda^{(r)}$, let
$a_r,b_r$ be the leg and arm lengths of its first diagonal hook.
By Lemma~\ref{lem:diagonal-hook-intervals}\textup{(1)}, $\alpha_{\lambda^{(r)}}$ is supported in $[i_r-a_r,i_r+b_r]$. Since
$a_r,b_r\le |\lambda^{(r)}|-1$, the bounds in
Definition~\ref{def:local-condition}\textup{(ii)(d)} give
\begin{equation}\label{eq:D-folded-support-bounds}
 1\le n-(|\lambda^{(r)}|-1)\le i_r-a_r\le i_r+b_r
       \le e-n+(|\lambda^{(r)}|-1)\le e-1.
\end{equation}
Since $\alpha_\emptyset=0$, we may assume
that the components considered below are nonempty.

\smallskip\noindent\textbf{Case 1: two ordinary components.}
Let $r,s\in[m]$ be distinct. Definition~\ref{def:local-condition}\textup{(ii)(d)}
gives $|i_r-i_s|\ge n$. Every $j\in\operatorname{supp}(\alpha_{\lambda^{(s)}})$ satisfies
\[
 |j-i_s|\le\max\{a_s,b_s\}\le|\lambda^{(s)}|-1<n\le|i_r-i_s|.
\]
Thus $i_r\notin\operatorname{supp}(\alpha_{\lambda^{(s)}})$. By
\eqref{eq:fundamental-weight-pairings}, and then interchanging $r,s$,
\[
 (\Lambda_{i_r}\mid\alpha_{\lambda^{(s)}})=0,\qquad
 (\Lambda_{i_s}\mid\alpha_{\lambda^{(r)}})=0.
\]

For the remaining pairing, order the components so that $i_r<i_s$. Definition~\ref{def:local-condition}\textup{(ii)(d)}
and $|\lambda^{(r)}|+|\lambda^{(s)}|\le n$ give
\[
 (i_s-a_s)-(i_r+b_r)
 \ge n-(|\lambda^{(s)}|-1)-(|\lambda^{(r)}|-1)\ge2.
\]
The nonzero coordinates of $\partial(\alpha_{\lambda^{(r)}})$ and
$\partial(\alpha_{\lambda^{(s)}})$ lie, respectively, in
$[i_r-a_r-1,i_r+b_r]$ and $[i_s-a_s-1,i_s+b_s]$.
These intervals are disjoint by the preceding inequality, so
Lemma~\ref{lem:D-coefficient-differences}\textup{(1)} gives
\[
 (\alpha_{\lambda^{(r)}}\mid\alpha_{\lambda^{(s)}})
 =2\partial(\alpha_{\lambda^{(r)}})\cdot \partial(\alpha_{\lambda^{(s)}})=0.
\]

\smallskip\noindent\textbf{Case 2: an ordinary component and a strict component.}
Let $s\in[m]$.

\smallskip\noindent\textbf{Subcase 2.1: the strict component is at $0$.}
Its position is $0_+$.
Lemma~\ref{lem:diagonal-hook-intervals}\textup{(2)}
places its residues in $[0,|\lambda^{(0_+)}|-1]$. Its largest vertex
is less than $n\le i_s$, by
Definition~\ref{def:local-condition}\textup{(ii)(d)}.
Conversely, \eqref{eq:D-folded-support-bounds} gives $0\notin\operatorname{supp}(\alpha_{\lambda^{(s)}})$. Equation~\eqref{eq:fundamental-weight-pairings}
therefore gives
\[
 (\Lambda_{i_s}\mid\alpha_{\lambda^{(0_+)}})=0,\qquad
 (\Lambda_0\mid\alpha_{\lambda^{(s)}})=0.
\]
\eqref{eq:D-folded-support-bounds}
and $|\lambda^{(s)}|+|\lambda^{(0_+)}|\le n$ give
\[
 (i_s-a_s)-(|\lambda^{(0_+)}|-1)
 \ge n-|\lambda^{(s)}|-|\lambda^{(0_+)}|+2\ge2.
\]
The nonzero coordinates of the two coefficient-difference vectors
lie in $[i_s-a_s-1,i_s+b_s]$ and $[0,|\lambda^{(0_+)}|-1]$,
respectively. These intervals are disjoint, so
Lemma~\ref{lem:D-coefficient-differences}\textup{(1)} gives
\[
 (\alpha_{\lambda^{(s)}}\mid\alpha_{\lambda^{(0_+)}})
 =2\partial(\alpha_{\lambda^{(s)}})\cdot \partial(\alpha_{\lambda^{(0_+)}})=0.
\]

\smallskip\noindent\textbf{Subcase 2.2: the strict component is at $e$.}
Its position is $0_-$.
Lemma~\ref{lem:diagonal-hook-intervals}\textup{(2)}
places its residues in $[e-|\lambda^{(0_-)}|+1,e]$. Its smallest vertex
is greater than $e-n\ge i_s$, by
Definition~\ref{def:local-condition}\textup{(ii)(d)}.
Conversely, \eqref{eq:D-folded-support-bounds} gives $e\notin\operatorname{supp}(\alpha_{\lambda^{(s)}})$. Equation~\eqref{eq:fundamental-weight-pairings}
therefore gives
\[
 (\Lambda_{i_s}\mid\alpha_{\lambda^{(0_-)}})=0,\qquad
 (\Lambda_e\mid\alpha_{\lambda^{(s)}})=0.
\]
\eqref{eq:D-folded-support-bounds}
and $|\lambda^{(s)}|+|\lambda^{(0_-)}|\le n$ give
\[
 (e-|\lambda^{(0_-)}|+1)-(i_s+b_s)
 \ge n-|\lambda^{(s)}|-|\lambda^{(0_-)}|+2\ge2.
\]
The nonzero coordinates of the two coefficient-difference vectors
lie in $[i_s-a_s-1,i_s+b_s]$ and $[e-|\lambda^{(0_-)}|,e-1]$,
respectively. These intervals are disjoint, so
Lemma~\ref{lem:D-coefficient-differences}\textup{(1)} gives
\[
 (\alpha_{\lambda^{(s)}}\mid\alpha_{\lambda^{(0_-)}})
 =2\partial(\alpha_{\lambda^{(s)}})\cdot \partial(\alpha_{\lambda^{(0_-)}})=0.
\]

\smallskip\noindent\textbf{Case 3: two strict components, at the positions $0_+$ and $0_-$.}
By Lemma~\ref{lem:diagonal-hook-intervals}\textup{(2)}, their
residues lie in $[0,|\lambda^{(0_+)}|-1]$ and
$[e-|\lambda^{(0_-)}|+1,e]$, respectively.
Since both sizes are at most $n\le e$, the first interval excludes
$e$ and the second excludes $0$. Thus
\eqref{eq:fundamental-weight-pairings} gives
\[
 (\Lambda_e\mid\alpha_{\lambda^{(0_+)}})=0,\qquad
 (\Lambda_0\mid\alpha_{\lambda^{(0_-)}})=0.
\]
Their total size is at most $n\le e$, so
\[
 (e-|\lambda^{(0_-)}|+1)-(|\lambda^{(0_+)}|-1)
 =e-|\lambda^{(0_-)}|-|\lambda^{(0_+)}|+2\ge2.
\]
The nonzero coordinates of their coefficient-difference vectors lie
in the disjoint intervals $[0,|\lambda^{(0_+)}|-1]$ and
$[e-|\lambda^{(0_-)}|,e-1]$. Hence
Lemma~\ref{lem:D-coefficient-differences}\textup{(1)} gives
\[
 (\alpha_{\lambda^{(0_+)}}\mid\alpha_{\lambda^{(0_-)}})
 =2\partial(\alpha_{\lambda^{(0_+)}})\cdot \partial(\alpha_{\lambda^{(0_-)}})=0.
\]
\end{proof}

\begin{proof}[Proof of Theorem~\ref{thm:combinatorial-zero-defect}]
By Proposition~\ref{prop:multipartition-support}, the nonzero summands $R^\Lambda_\alpha$ of $R_n^\Lambda(I)$ are indexed by the elements $\alpha_{\boldsymbol\lambda}$.
The equivalence \textup{(i)}$\Longleftrightarrow$\textup{(ii)} in
Theorem~\ref{thm:structural-criterion} therefore gives
\textup{(ii)}$\Longleftrightarrow$\textup{(iii)}.
It remains to prove \textup{(i)}$\Longleftrightarrow$\textup{(ii)}.
For $n=0$, these two conditions hold since $\alpha_\emptyset=0$.
Assume $n\ge1$. We first consider $D_{e+1}^{(2)}$ and suppose that
$\mathtt{S}_\Lambda(n)$ holds. Expanding the defect by \eqref{eq:defect-normalization}, using
\eqref{eq:source-component-weight} and \eqref{eq:multipartition-root},
and cancelling the cross terms by Lemma~\ref{lem:cross-terms-vanish}, we obtain
\begin{equation}\label{eq:local-defect-expansion}
\operatorname{def}(\Lambda,\alpha_{\boldsymbol\lambda})
=\sum_{r=1}^{m}\operatorname{def}
 \bigl(2^{\delta_{{\rm p}(i_r),\bar1}}\Lambda_{i_r},\alpha_{\lambda^{(r)}}\bigr)
 +\sum_{r\in\mathcal C_\bullet\setminus[m]}\operatorname{def}(\Lambda_{i_r},\alpha_{\lambda^{(r)}}).
\end{equation}

Definition~\ref{def:local-condition}\textup{(ii)(d)} gives
$n\le i_r\le e-n$ for $r\in[m]$, so ${\rm p}(i_r)=\bar0$.
By Lemmas~\ref{lem:diagonal-hook-intervals} and
\ref{lem:hook-norm-and-orthogonality}, each summand on the right-hand
side of \eqref{eq:local-defect-expansion} is zero. Hence
$\operatorname{def}(\Lambda,\alpha_{\boldsymbol\lambda})=0$.
The converse follows from Lemma~\ref{lem:positive-defect-partition}.

For $I$ of type $A_e^{(1)}$, $C_e^{(1)}$, or $A_{2e}^{(2)}$,
we replace Lemma~\ref{lem:positive-defect-partition} and Lemma~\ref{lem:cross-terms-vanish} by the corresponding lemmas in Appendix~\ref{app:other-affine-types}.

Finally, if $I$ is of type $A_\infty$, $B_\infty$, or $C_\infty$, take
an affine vertex set $I_e$ of type $A_e^{(1)}$, $D_{e+1}^{(2)}$, or
$C_e^{(1)}$, respectively, with $e$ sufficiently large.
As in the finite restriction argument of Proposition~\ref{prop:multipartition-support},
the supports of all $\alpha_{\boldsymbol\lambda}$, with $|\boldsymbol\lambda|=n$, lie in a common finite full subdiagram of $I$ and $I_e$,
after translating the labels in type $A_\infty$ if necessary.
On this subdiagram, the Cartan entries and the residues in
\eqref{eq:residue-formulas} agree;
Lemma~\ref{lem:finite-support-restriction} preserves their defects.
For large $e$, the bounds in Definition~\ref{def:local-condition}\textup{(ii)--(iii)}
involving the added endpoint or affine period are automatic; the remaining
bounds are exactly those for $I$. The result follows from the affine case.

\end{proof}

\section{The general semisimplicity criterion}\label{sec:general-semisimplicity}
The aim of this section is to prove Theorem~\ref{mainthm}.
Fix $f=f^{(\bullet)}_{\undQ}$ for $\bullet\in\{\mathtt{0},\mathtt{s},\mathtt{ss}\}$, with $(m,\bullet)\ne(0,\mathtt{0})$.
We begin with the single-orbit case.

\begin{prop}\label{prop:single-orbit-semisimplicity}
Suppose that $I_f$ is connected. Identify it with the integer vertex
set $I$ in Definition~\ref{def:KKT-spectral-labelling}, and let
$\Lambda=\Lambda_f$. For $n\ge1$, the following conditions are equivalent:
\begin{enumerate}[label=\textup{(\roman*)}]
\item $P_n^{(\bullet)}(q^2,\undQ)\ne0$;
\item $\mathtt{S}_\Lambda(n)$ holds;
\item $R_n^\Lambda(I)$ is semisimple;
\item $\mathcal H^f_{\mathbb K}(n)$ is semisimple.
\end{enumerate}
\end{prop}
\begin{proof}
For the four affine types, \eqref{eq:nondeg-separation} and the
parameter correspondence in Definition~\ref{def:KKT-spectral-labelling}\textup{(i)--(iv)}
reduce \textup{(i)}$\Longleftrightarrow$\textup{(ii)} to
\cite[\S8, Propositions~8.1, 8.5, 8.9 and~8.14]{Shi:semisimplicity}.
For types $A_\infty$, $B_\infty$, and $C_\infty$, substituting the representatives from
Definition~\ref{def:KKT-spectral-labelling}\textup{(v)--(vii)} into
\eqref{eq:nondeg-separation} gives the corresponding bounds in
Definition~\ref{def:local-condition}\textup{(ii)}.
When $m=0$, \eqref{eq:nondeg-separation}--\eqref{eq:nondeg-extra-factor}
give Definition~\ref{def:local-condition}\textup{(iii)}.

Theorem~\ref{thm:combinatorial-zero-defect}\textup{(i), (iii)} gives
\textup{(ii)}$\Longleftrightarrow$\textup{(iii)}.
The equivalence \textup{(iii)}$\Longleftrightarrow$\textup{(iv)} follows from
the residue correspondence
\eqref{eq:folding-spectral-identity}, Corollary~\ref{cor:KKT-semisimplicity}, and \eqref{eq:finite-support-restriction}.
\end{proof}

We now combine the connected components.
\begin{proof}[Proof of Theorem~\ref{mainthm}]
Write $I_f=\bigsqcup_{c=1}^d I_c$ and $f=\prod_{c=1}^d f_c$, where
$f_c$ corresponds to $\Lambda_f|_{I_c}$ under
Definition~\ref{def:KKT-spectral-labelling} and \eqref{eq:KKT-weight}.
By \eqref{eq:nondeg-separation}--\eqref{eq:nondeg-extra-factor} and the
orbit definition, the factors involving parameters from distinct
orbits are nonzero. Thus $P_n^{(\bullet)}(q^2,\undQ)\ne0$ if and only
if the rank-$n$ Poincar\'e polynomial of every $f_c$ is nonzero.
Applying Proposition~\ref{prop:single-orbit-semisimplicity} to each
orbit gives
\begin{equation}\label{eq:orbitwise-semisimplicity}
 P_n^{(\bullet)}(q^2,\undQ)\ne0
 \quad\Longleftrightarrow\quad
 R_n^{\Lambda_f|_{I_c}}(I_c)\text{ is semisimple for every }c.
\end{equation}

Suppose $P_n^{(\bullet)}(q^2,\undQ)\ne0$.
Fix $n_1+\cdots+n_d=n$. For every $c$ with $n_c>0$, the Poincar\'e polynomial
for $f_c$ at rank $n_c$ is also nonzero, since its factor ranges in
\eqref{eq:nondeg-separation}--\eqref{eq:nondeg-extra-factor} are
contained in those at rank $n$.
Proposition~\ref{prop:single-orbit-semisimplicity} gives the
semisimplicity of every positive-rank factor
$R_{n_c}^{\Lambda_f|_{I_c}}(I_c)$.
Rostam's decomposition in Proposition~\ref{prop:rostam-decomposition}
and Lemma~\ref{lem:supertensor-semisimplicity} therefore imply that
$R_n^{\Lambda_f}(I_f)$ is semisimple.
Conversely, fix $c$ and in that decomposition take $n_c=n$ and $n_j=0$ for
$j\ne c$. Each $R_n^{\Lambda_f|_{I_c}}(I_c)$ is then a direct
summand of $R_n^{\Lambda_f}(I_f)$, so global semisimplicity implies
$P_n^{(\bullet)}(q^2,\undQ)\ne0$ by \eqref{eq:orbitwise-semisimplicity}.

The theorem now follows from Corollary~\ref{cor:KKT-semisimplicity}.
\end{proof}

In the degenerate setting, let $\mathfrak H(n)$ be the affine Sergeev
algebra of \cite[\S5.1]{SW}. For $m\ge0$,
$\sigma\in\{0,1\}$ with $2m+\sigma >0$, and
$\undQ=(Q_1,\ldots,Q_m)\in\mathbb K^m$, set
\begin{equation}\label{eq:degenerate-cyclotomic-data}
 g_\sigma(x)=x^\sigma\prod_{j=1}^m
                  \bigl(x^2-Q_j(Q_j+1)\bigr),\qquad
 \mathfrak H^{g_\sigma}(n)
 =\mathfrak H(n)/\langle g_\sigma(x_1)\rangle.
\end{equation}
For $n\ge1$, write
\begin{equation}\label{eq:degenerate-separation-polynomial}
\begin{aligned}
 P_{n,\mathrm{deg}}(\undQ)
 ={}&n!\prod_{j=1}^m
 \left(\prod_{t=3-n}^{n-1}(2Q_j+t)
       \prod_{t=1-n}^{n}(Q_j+t)\right)\\
 &\quad\cdot\prod_{1\le j<k\le m}\prod_{t=1-n}^{n-1}
 (Q_j-Q_k+t)(Q_j+Q_k+t+1),
\end{aligned}
\end{equation}
where an empty product is one; this is the polynomial in \cite[\S5.4]{SW}.
\begin{thm}\label{thm:degenerate-semisimplicity}
For the datum in \eqref{eq:degenerate-cyclotomic-data} and
$n\ge1$,
\[
 \mathfrak H^{g_\sigma}(n)\text{ is semisimple}
 \quad\Longleftrightarrow\quad
 P_{n,\mathrm{deg}}(\undQ)\ne0.
\]
\end{thm}
\begin{proof}[Sketch]
The degenerate KKT isomorphism of \cite[Theorem~5.4 and the
cyclotomic paragraph after Lemma~5.5]{KKT}, with the parameter correspondence
in \cite[\S5.2, \S5.4]{SW}, gives a Cartan superdatum with vertex set $I$
and a dominant weight $\Lambda$.
Comparing the factors in \eqref{eq:degenerate-separation-polynomial}
with Definition~\ref{def:local-condition} gives
$P_{n,\mathrm{deg}}(\undQ)\ne0$ if and only if
$C_{\Lambda|_{I_c}}(n)$ holds for every connected component $I_c$ of $I$.
The dimension in Lemma~\ref{lem:LS-multipartition-dimension}, together with the argument used in Proposition~\ref{prop:single-orbit-semisimplicity} and Theorem~\ref{mainthm}, then yields the assertion.
\end{proof}

\appendix
\section{Statements for the remaining affine types}\label{app:other-affine-types}

In this appendix, we sketch the proof of
Theorem~\ref{thm:combinatorial-zero-defect} for the remaining three
affine types. We retain the notation of
Section~\ref{sec:single-orbit}. All unspecified components are empty,
and positions $r,s$ in a pair condition are distinct.

\subsection*{Type $A_e^{(1)}$ ($e\ge2$)}

There are only ordinary components.  For
$\gamma=\sum_{j\in I}c_j\nu_j\in Q^+$, with indices modulo $e+1$, let
$\varepsilon_0,\ldots,\varepsilon_e$ be the standard coordinate basis in $\mathbb{Z}^{e+1}$, with
subscripts modulo $e+1$, and set
\[
 \partial(\gamma)=(c_j-c_{j+1})_{j\bmod e+1}\in\mathbb{Z}^{e+1}.
\]

For type $A_e^{(1)}$, $\delta=\sum_{j=0}^e\nu_j \in Q^+$ be the minimal imaginary root.
The implication
\textup{(ii)}$\Longrightarrow$\textup{(i)} in
Theorem~\ref{thm:combinatorial-zero-defect} follows from the next lemma.

\begin{lem}\label{lem:appendix-A-failure}
Fix $n\ge1$.  If $\mathtt{S}_\Lambda(n)$ fails, one of the following
constructions gives an allowed $\boldsymbol\lambda\in
\mathscr P_n^{\bullet,m}$ with positive defect.
\begin{enumerate}[label=\textup{(\arabic*)}]
\item If $n>e$, choose an ordinary position $r$ and take
$\lambda^{(r)}=(n)$.  Then
\[
 \alpha_{\lambda^{(r)}}=\gamma_I[i_r,i_r+n-1].
\]
Write $n=k(e+1)+c$, where $0\le c<e+1$. We have
\[
 \operatorname{def}(\Lambda_{i_r},\alpha_{\lambda^{(r)}})=k>0.
\]
\item Suppose $n\le e$, $i_r\le i_s$, and $i_s-i_r<n$.  Take
$\lambda^{(r)}=(n)$.  Then
\[
 \alpha_{\lambda^{(r)}}=\gamma_I[i_r,i_r+n-1],
\]
and
\[
 \operatorname{def}(\Lambda_{i_r},\alpha_{\lambda^{(r)}})=0,\qquad
 (\Lambda_{i_s}\mid\alpha_{\lambda^{(r)}})>0,
\]
and hence $\operatorname{def}(\Lambda,\alpha_{\boldsymbol\lambda})>0$.
\item Suppose $n\le e$, $i_r<i_s$, and $i_s-i_r>e+1-n$.  Take
$\lambda^{(s)}=(n)$.  Then
\[
 \alpha_{\lambda^{(s)}}=\gamma_I[i_s,i_s+n-1],
\]
where the interval is read modulo $e+1$, and
\[
 \operatorname{def}(\Lambda_{i_s},\alpha_{\lambda^{(s)}})=0,\qquad
 (\Lambda_{i_r}\mid\alpha_{\lambda^{(s)}})>0.
\]
Hence $\operatorname{def}(\Lambda,\alpha_{\boldsymbol\lambda})>0$.
\end{enumerate}
\end{lem}

For type $A_e^{(1)}$, the following lemma and
Lemma~\ref{lem:diagonal-hook-intervals} give the implication
\textup{(i)}$\Longrightarrow$\textup{(ii)} in
Theorem~\ref{thm:combinatorial-zero-defect}.

\Needspace{5\baselineskip}
\begin{lem}\label{lem:appendix-A-differences}
Fix $\boldsymbol\lambda\in\mathscr P_n^{\bullet,m}$ with $n\ge1$.
\begin{enumerate}[label=\textup{(\arabic*)}]
\item For all $\gamma,\gamma'\in Q$, applying the bilinear identity from
\eqref{eq:local-root-norms} gives
\[
 (\gamma\mid\gamma')=\partial(\gamma)\cdot \partial(\gamma'),\qquad
 (\gamma\mid\gamma)=\sum_{j\bmod e+1}(c_j-c_{j+1})^2.
\]
For $J=[L,R]$ with $1\le |J|\le n\le e$,
\[
 \partial(\gamma_I(J))=\varepsilon_R-\varepsilon_{L-1},\qquad
 (\gamma_I(J)\mid\gamma_I(J))=2.
\]
\item Assume $\mathtt{S}_\Lambda(n)$ holds.  For every diagonal hook interval
$J=[i_r-a,i_r+b]$,
\[
 \bigl((\gamma_I(J)\mid\gamma_I(J)),
       (\Lambda_{i_r}\mid\gamma_I(J))\bigr)=(2,1).
\]
If $J_t\ne J_u$ are diagonal hook intervals of the same component, then
\[
 (\gamma_I(J_t)\mid\gamma_I(J_u))=0.
\]
\item Assume $\mathtt{S}_\Lambda(n)$ holds.  For distinct ordinary positions
$r,s$,
\[
 (\Lambda_{i_r}\mid\alpha_{\lambda^{(s)}})=0,\qquad
 (\alpha_{\lambda^{(r)}}\mid\alpha_{\lambda^{(s)}})=0.
\]
\end{enumerate}
\end{lem}

\subsection*{Type $C_e^{(1)}$ ($e\ge2$)}

There are only ordinary components.  For
$\gamma=\sum_{j=0}^ec_j\nu_j\in Q^+$, let
$\varepsilon_0,\ldots,\varepsilon_{e-1}$ be the standard coordinate basis in $\mathbb{Z}^{e}$ and
set
\[
 \partial(\gamma)=(2c_0-c_1,(c_j-c_{j+1})_{1\le j\le e-2},c_{e-1}-2c_e)\in \mathbb{Z}^{e}.
\]
Extend these coordinate vectors by
\[
 \varepsilon_{j+2e}=\varepsilon_j,\qquad
 \varepsilon_{-j-1}=-\varepsilon_j\qquad(j\in\mathbb Z).
\]

For type $C_e^{(1)}$, let $\delta=\nu_0+2\sum_{j=1}^{e-1}\nu_j+\nu_e \in Q^+$ be the minimal imaginary root.
The implication
\textup{(ii)}$\Longrightarrow$\textup{(i)} in
Theorem~\ref{thm:combinatorial-zero-defect} follows from the next lemma.

\begin{lem}\label{lem:appendix-C-failure}
Fix $n\ge1$.  If $\mathtt{S}_\Lambda(n)$ fails, one of the following
constructions gives an allowed $\boldsymbol\lambda\in
\mathscr P_n^{\bullet,m}$ with positive defect.
\begin{enumerate}[label=\textup{(\arabic*)}]
\item If $n>2i_r+1$, take $\lambda^{(r)}=(1^n)$.  Then
\[
 \alpha_{\lambda^{(r)}}=\gamma_I[i_r-n+1,i_r].
\]
\item If $n>2(e-i_r)+1$, take $\lambda^{(r)}=(n)$.  Then
\[
 \alpha_{\lambda^{(r)}}=\gamma_I[i_r,i_r+n-1].
\]
For the two preceding items, write $n=2ek+c$, where $0\le c<2e$.
The defect in item~\textup{(1)} is
\[
 \operatorname{def}(\Lambda_{i_r},\alpha_{\lambda^{(r)}})=
 \begin{cases}
 2k+\mathbf1_{c>2i_r+1},&0\le i_r<e,\\
 2k+\mathbf1_{c>1},&i_r=e,
 \end{cases}
\]
and in item~\textup{(2)} it is
\[
 \operatorname{def}(\Lambda_{i_r},\alpha_{\lambda^{(r)}})=
 \begin{cases}
 2k+\mathbf1_{c>2(e-i_r)+1},&0<i_r\le e,\\
 2k+\mathbf1_{c>1},&i_r=0.
 \end{cases}
\]
Hence $\operatorname{def}(\Lambda,\alpha_{\boldsymbol\lambda})>0$ in
item~\textup{(1)} or~\textup{(2)}.
\item Suppose the two one-charge bounds hold at distinct ordinary positions
$r,s$, $i_r\le i_s$, and
$i_s-i_r<n$.  Take $\lambda^{(r)}=(n)$.  Then
\[
 \alpha_{\lambda^{(r)}}=\gamma_I[i_r,i_r+n-1],
\]
and
\[
 \operatorname{def}(\Lambda_{i_r},\alpha_{\lambda^{(r)}})=0,\qquad
 (\Lambda_{i_s}\mid\alpha_{\lambda^{(r)}})>0.
\]
Hence $\operatorname{def}(\Lambda,\alpha_{\boldsymbol\lambda})>0$.
\end{enumerate}
\end{lem}

For type $C_e^{(1)}$, the following lemma and
Lemma~\ref{lem:diagonal-hook-intervals} give the implication
\textup{(i)}$\Longrightarrow$\textup{(ii)} in
Theorem~\ref{thm:combinatorial-zero-defect}.

\begin{lem}\label{lem:appendix-C-differences}
Fix $\boldsymbol\lambda\in\mathscr P_n^{\bullet,m}$ with $n\ge1$.
\begin{enumerate}[label=\textup{(\arabic*)}]
\item For all $\gamma,\gamma'\in Q$,
\[
 (\gamma\mid\gamma')=\partial(\gamma)\cdot \partial(\gamma'),\qquad
 (\gamma\mid\gamma)=(2c_0-c_1)^2+
 \sum_{j=1}^{e-2}(c_j-c_{j+1})^2+(c_{e-1}-2c_e)^2.
\]
For every interval $J=[L,R]$,
\[
 \partial(\gamma_I(J))=\varepsilon_R-\varepsilon_{L-1}.
\]
\item Assume $\mathtt{S}_\Lambda(n)$ holds, so
\[
 n\le2i_r+1,\qquad n\le2(e-i_r)+1.
\]
For every diagonal hook $J=[i_r-a,i_r+b]$,
\[
 \bigl((\gamma_I(J)\mid\gamma_I(J)),
       (\Lambda_{i_r}\mid\gamma_I(J))\bigr)=
 \begin{cases}
 (4,2),&J=[-i_r,i_r]\ \text{or}\ J=[i_r,2e-i_r],\\
 (2,1),&\text{otherwise}.
 \end{cases}
\]
If $J_t\ne J_u$ are
diagonal hook intervals of one component, then
\[
 (\gamma_I(J_t)\mid\gamma_I(J_u))=0.
\]
\item Assume $\mathtt{S}_\Lambda(n)$ holds.  For distinct ordinary positions
$r,s$,
\[
 (\Lambda_{i_r}\mid\alpha_{\lambda^{(s)}})=0,\qquad
 (\alpha_{\lambda^{(r)}}\mid\alpha_{\lambda^{(s)}})=0.
\]
\end{enumerate}
\end{lem}

\subsection*{Type $A_{2e}^{(2)}$ ($e\ge1$)}

For $\gamma=\sum_{j=0}^ec_j\nu_j\in Q^+$, let
$\varepsilon_0,\ldots,\varepsilon_{e-1}$ be the standard coordinate basis in $\mathbb{Z}^{e}$,
extended by
\[
 \varepsilon_{j+2e}=\varepsilon_j,\qquad
 \varepsilon_{-j-1}=-\varepsilon_j\qquad(j\in\mathbb Z),
\]
and set
\[
 \partial(\gamma)=((c_j-c_{j+1})_{0\le j\le e-2},c_{e-1}-2c_e)\in\mathbb Z^e.
\]

For type $A_{2e}^{(2)}$,
let $\delta=2\sum_{j=0}^{e-1}\nu_j+\nu_e \in Q^+$ be the minimal imaginary root.
The implication
\textup{(ii)}$\Longrightarrow$\textup{(i)} in
Theorem~\ref{thm:combinatorial-zero-defect} follows from the next lemma.

\Needspace{5\baselineskip}
\begin{lem}\label{lem:appendix-twisted-A-failure}
Fix $n\ge1$.  If $\mathtt{S}_\Lambda(n)$ fails, one of the following
constructions gives an allowed $\boldsymbol\lambda\in
\mathscr P_n^{\bullet,m}$ with positive defect.
\begin{enumerate}[label=\textup{(\arabic*)}]
\item If $n>i_r$ at an ordinary position, take
$\lambda^{(r)}=(n-i_r,1^{i_r})$.  Then
\[
 \alpha_{\lambda^{(r)}}=\gamma_I[0,n-1],\qquad
 \bigl(2^{\delta_{{\rm p}(i_r),\bar1}}\Lambda_{i_r}
       \mid\alpha_{\lambda^{(r)}}\bigr)\ge2,\qquad
 (\alpha_{\lambda^{(r)}}\mid\alpha_{\lambda^{(r)}})\le2,
\]
and
\[
 \operatorname{def}\bigl(2^{\delta_{{\rm p}(i_r),\bar1}}\Lambda_{i_r},
                           \alpha_{\lambda^{(r)}}\bigr)\ge1.
\]
\item If $n\le i_r$ and $n>2(e-i_r)+1$, take
$\lambda^{(r)}=(n)$.  Then ${\rm p}(i_r)=\bar0$,
\[
 \alpha_{\lambda^{(r)}}=\gamma_I[i_r,i_r+n-1],\qquad
 \operatorname{def}(\Lambda_{i_r},\alpha_{\lambda^{(r)}})=2.
\]
\item Suppose the two one-charge bounds hold at distinct ordinary positions
$r,s$, $i_r\le i_s$, and
$i_s-i_r<n$.  Take $\lambda^{(r)}=(n)$.  Then
$\alpha_{\lambda^{(r)}}=\gamma_I[i_r,i_r+n-1]$ and
${\rm p}(i_r)={\rm p}(i_s)=\bar0$, and
\[
 \operatorname{def}(\Lambda_{i_r},\alpha_{\lambda^{(r)}})=0,\qquad
 (\Lambda_{i_s}\mid\alpha_{\lambda^{(r)}})>0,
\]
so $\operatorname{def}(\Lambda,\alpha_{\boldsymbol\lambda})>0$.
\item If $m=0$ and $n>2e$, choose the strict position $r$ at $0$
and take $\lambda^{(r)}=(n)$.  Then $\alpha_{\lambda^{(r)}}=\gamma_I[0,n-1]$.  Write
$n=k(2e+1)+c$, where $0\le c<2e+1$.
We have
\[
 \operatorname{def}(\Lambda_0,\alpha_{\lambda^{(r)}})=2k>0.
\]
\end{enumerate}
\end{lem}

For type $A_{2e}^{(2)}$, the following lemma and
Lemma~\ref{lem:diagonal-hook-intervals} give the implication
\textup{(i)}$\Longrightarrow$\textup{(ii)} in
Theorem~\ref{thm:combinatorial-zero-defect}.

\begin{lem}\label{lem:appendix-twisted-A-differences}
Fix $\boldsymbol\lambda\in\mathscr P_n^{\bullet,m}$ with $n\ge1$.
\begin{enumerate}[label=\textup{(\arabic*)}]
\item For all $\gamma,\gamma'\in Q$,
\[
 (\gamma\mid\gamma')=2\partial(\gamma)\cdot \partial(\gamma'),\qquad
(\gamma\mid\gamma)=2\left(
 \sum_{j=0}^{e-2}(c_j-c_{j+1})^2+(c_{e-1}-2c_e)^2\right).
\]
For an interval $J=[L,R]$ with $1\le L\le R\le2e-1$,
\[
 \partial(\gamma_I(J))=\varepsilon_R-\varepsilon_{L-1}.
\]
\item Assume $\mathtt{S}_\Lambda(n)$ holds.  For an ordinary diagonal hook
$J=[i_r-a,i_r+b]$, the bounds
\[
 n\le i_r,\qquad n\le2(e-i_r)+1
\]
hold, and
\[
 \bigl((\gamma_I(J)\mid\gamma_I(J)),
       (\Lambda_{i_r}\mid\gamma_I(J))\bigr)=
 \begin{cases}
 (8,4),&J=[i_r,2e-i_r],\\
 (4,2),&\text{otherwise}.
 \end{cases}
\]
For distinct diagonal hooks $J_t,J_u$ of one ordinary component,
\[
 (\gamma_I(J_t)\mid\gamma_I(J_u))=0.
\]
We have $n\le2e$. Let
$\sigma=(\sigma_1>\cdots>\sigma_h>0)$ be a strict component of
$\boldsymbol\lambda$ at $0$.  For $1\le t\le h$ and a row length
$l=\sigma_t$,
\[
 \partial(\gamma_I[0,l-1])=
 \begin{cases}
 \varepsilon_{l-1},&1\le l\le e,\\
 -\varepsilon_{2e-l},&e+1\le l\le2e,
 \end{cases}
\]
and
\[
 \bigl((\gamma_I[0,l-1]\mid\gamma_I[0,l-1]),
       (\Lambda_0\mid\gamma_I[0,l-1])\bigr)=(2,1)
 \qquad(1\le l\le2e),
\]
and, for $l=\sigma_t$ and $l'=\sigma_u$ with $1\le t\ne u\le h$, the two
rows satisfy $l+l'\le|\sigma|\le n\le2e$ and
\[
 (\gamma_I[0,l-1]\mid\gamma_I[0,l'-1])=0.
\]
\item Assume $\mathtt{S}_\Lambda(n)$ holds.  For distinct component indices
$r,s\in\mathcal C_\bullet$,
\[
 (\Lambda_{i_r}\mid\alpha_{\lambda^{(s)}})=0,\qquad
 (\alpha_{\lambda^{(r)}}\mid\alpha_{\lambda^{(s)}})=0.
\]
\end{enumerate}
\end{lem}

\enlargethispage{4\baselineskip}
\end{document}